\documentclass[a4paper,11pt]{amsart}

\usepackage{hyperref}
\hypersetup{backref,colorlinks=true,allcolors=black}
\usepackage[foot]{amsaddr}
\usepackage{mathrsfs}
\usepackage{enumerate}
\usepackage{dsfont}
\usepackage{mathtools}
\usepackage{amssymb}
\usepackage{xcolor}
\usepackage{bm}

\newtheorem{prop}{Proposition}[section]
\newtheorem{lemma}{Lemma}[section]
\newtheorem{Lemma}{Lemma}[section]

\newtheorem{theorem}{Theorem}[section]
\newtheorem{remark}{Remark}[section]

\newtheorem{assumption}{Assumption}[section]

\newtheorem*{acknow*}{Acknowledgments}

\newcommand{\ignore}[1]{}
\newcommand{\R}{\mathbb{R}}

\newcommand{\on}[1]{\operatorname{#1}}
\newcommand{\norm}[1]{\left\lVert #1 \right\rVert}
\newcommand{\abs}[1]{\left\vert #1 \right\rvert}

\newcommand{ \parbar}[1]{ { \left( #1 \right)} }

\allowdisplaybreaks[3]

\author{S. Bourguin, and K. Spiliopoulos}
\address{Boston University, Department of Mathematics and Statistics\\ 665 Commonwealth Ave, Boston, MA 02215, USA}
\email[Solesne Bourguin]{bourguin@math.bu.edu}
\email[Konstantinos Spiliopoulos]{kspiliop@math.bu.edu}
\thanks{S.B was partially supported by the Simons Foundation Award
  635136.}
\thanks{K.S. was partially supported by the National Science
  Foundation DMS 2107856, 2311500 and Simons Foundation Award
  672441.}
\thanks{The authors would like to thank Shivam Singh Dhama for spotting a few typos at
 an earlier version of the article.}
\title[Quantitative fluctuation analysis of multiscale
diffusions]{Quantitative fluctuation analysis of multiscale diffusion systems via Malliavin calculus}
\date{\today}
\begin{document}
\begin{abstract}
We study fluctuations of small noise multiscale diffusions around their homogenized deterministic limit. We derive quantitative rates of convergence of the fluctuation processes to their Gaussian limits in the appropriate Wasserstein metric requiring detailed estimates of the first and second order Malliavin derivative of the slow component. We study a fully coupled system and the derivation of the quantitative rates of convergence depends on a very careful decomposition of the first and second Malliavin derivatives of the slow and fast component to terms that have different rates of convergence depending on the strength of the noise and timescale separation parameter.


\end{abstract}	
\bibliographystyle{amsalpha}

	\maketitle

\section{Introduction}\label{S:Introduction}

The goal of this paper is to develop quantitative convergence results
for the fluctuations of multiscale diffusion systems. In particular, we consider two-dimensional multiscale systems of the form
 \begin{equation}
   \label{slowfastsystem}
   \begin{cases}
     \displaystyle dX_t^{\varepsilon} = c \left(
       X_t^{\varepsilon},Y_t^{\eta} \right) dt + \sqrt{\varepsilon}
     \sigma \left( X_t^{\varepsilon},Y_t^{\eta} \right)dW_t^1, \quad X^{\varepsilon}_{0}=x_{0} \\
     dY_t^{\eta} = \frac{1}{\eta} f \left(X_t^{\varepsilon},Y_t^{\eta}  \right) dt +\frac{1}{\sqrt{\eta}}
     \tau \left( X_t^{\varepsilon},Y_t^{\eta} \right)dW_t^2, \quad Y^{\eta}_{0}=y_{0}
   \end{cases}.
 \end{equation}
Here $\eta=\eta(\varepsilon)\downarrow0$ as $\varepsilon\downarrow0$ and $(W^1, W^2)$ is a $2$-dimensional standard Wiener process. Assumptions on the coefficients $c(x,y),\sigma(x,y), f(x,y)$ and $\tau(x,y)$ are given in Condition \ref{A:Assumption1}. Here $X^{\varepsilon}$ and $Y^{\eta}$ can be viewed as the slow and fast component, respectively.

One question that has naturally been posed and answered in the literature is in regards to the limiting behavior of the slow component $X^{\varepsilon}$ as $\varepsilon,\eta\downarrow 0$. As it is standard by now, under various assumptions on the coefficients, see for example \cite{DupuisSpiliopoulos,Freidlin1978, FS, Guillin,LiptserPaper,PardouxVeretennikov1, PardouxVeretennikov2,RocknerAveraging2021a,RocknerAveraging2021b,Spiliopoulos_CLT_Multiscale_2014} there is a well defined limit  $\bar{X}=\lim_{\varepsilon,\eta\rightarrow 0} X^{\varepsilon}$ oftentimes in probability (recalled in Section \ref{S:MainResults}). The next order of convergence, i.e., fluctuations of $X^{\varepsilon}$ around its limit $\bar{X}$ were studied in \cite{Spiliopoulos_CLT_Multiscale_2014}. In particular, it was established in \cite{Spiliopoulos_CLT_Multiscale_2014} that under the proper assumptions, the fluctuations process
\begin{equation}
\theta^{\varepsilon}_t = \frac{1}{\sqrt{\varepsilon}} \left(
   X_t^{\varepsilon} - \bar{X}_t \right),\label{Eq:FluctuationsProcess}
\end{equation}
converges weakly in the space of continuous functions $C([0,T];\mathbb{R})$ to a process $\theta_{\cdot}$, which at a given time $t$ is distributed according to a Gaussian random variable $ \mathscr{N} \left(\mu_t,\sigma_t^2   \right)$ for appropriate mean and variance $\mu_t,\sigma_t^2$ respectively. We recall in detail this result in  Section \ref{S:MainResults}.

The fluctuation results of \cite{Spiliopoulos_CLT_Multiscale_2014} were qualitative. The goal of this paper is to prove quantitative results of convergence of the law of $\theta^{\varepsilon}_t$ to the law of its limit $\theta_t$ (with $ \textrm{Law}(\theta_t)=\mathscr{N} \left(\mu_t,\sigma_t^2   \right)$) in the appropriate Wasserstein metric. To accomplish this, we use certain bounds involving Malliavin derivatives, see \cite{nourdin_second_2009,nualart_malliavin_2006}. The methodology we follow and the main result are presented in Section \ref{S:MainResults}.

The contribution of this paper is two-fold. Firstly, we quantitatively characterize the convergence of fluctuations of fully coupled multiscale systems with precise rates of convergence in the Wasserstein metric. Secondly, we accomplish this goal exploiting a deep connection with Malliavin calculus \cite{nourdin_second_2009} that is generic enough to be of use in other settings where one is interested in quantitative convergence of fluctuations when the limiting fluctuations are of Gaussian nature.

In order to accomplish this, we need to derive tight bounds for
quantities involving the first and second order Malliavin derivatives of the slow and fast components $X^{\varepsilon}, Y^{\eta}$ with respect to the driving noises $W^{1},W^{2}$. This is done in the technical sections \ref{S:FirstOrderMalliavinDer}-\ref{S:SecondOrderMalliavinDer}. These sections contain the technical heart of what we do in this paper. Both the first and the second order Malliavin derivatives satisfy affine stochastic differential equations and can be solved explicitly. However, we remark here that even though the Malliavin derivatives are given in closed form,  the bounds have to be done very carefully in order to extract the best possible bounds. As a matter of fact, as we will see in Sections \ref{S:FirstOrderMalliavinDer}-\ref{S:SecondOrderMalliavinDer}, the key is to realize that different terms contribute differently in the final rate and thus one needs to form the proper decompositions and study the different terms differently in order to extract the best possible rate in terms of $\varepsilon$ and $\eta$.

In particular, the derivation of the bound for the first and second Malliavin derivatives of the fast component $Y^{\eta}$ with respect to $\varepsilon$ and $\eta$ is particularly complex. Naive bounds quickly lead to insufficient control of the Wasserstein metric characterizing the convergence of the law of the fluctuation process to the limit Gaussian law. Different components behave differently with respect to the size of the noise $\varepsilon$, but more importantly with respect to the fast oscillating  parameter $\eta$. We decompose the Malliavin derivatives appropriately obtaining the best possible bounds for each one of the terms and then we collect the bounds together to form the final bound. An added benefit of the analysis is that one sees directly which terms contributes what to the final bound.

It is also worth noting that our analysis constitutes an example of
error bounds obtained via Malliavin calculus for stochastic
differential equations (and hence for functionals of Gaussian fields
with infinite chaos expansions). Indeed, as the
action of the Malliavin operators is easier to characterize for
individual chaos elements (multiple Wiener integrals), it is not
surprising that a lot of work has been done for elements of a single
or finitely many Wiener chaoses. In our work, we
deal with solutions of fully coupled SDE systems directly without having to go through a
chaos decomposition and analyze the contribution of each chaos. This way of proceeding is in our opinion
of independent interest as it is general enough to apply in many other situations.

The only other result in the literature that studies related
quantitative rates of convergence of fluctuations for multiscale
systems that we are aware of is that of
\cite{RocknerAveraging2021a}. The authors in
\cite{RocknerAveraging2021a} study a system analogous to
(\ref{slowfastsystem}) with (in our notation) $\sigma(x,y)=\sigma(x)$ with $\varepsilon=1$
and they obtain bounds for quantities of the form
$\sup_{t\in[0,T]}\left|\mathbb{E}[\phi(\theta^{\eta}_{t})]-\mathbb{E}[\phi(\theta_{t})]\right|$
for test functions $\phi\in C^{4}_{b}(\mathbb{R})$. Their proof relies
on viewing the functions $\mathbb{E}[\phi(\theta^{\eta}_{t})],
\mathbb{E}[\phi(\theta_{t})]$ (functions of the initial point and time
$t$) as solutions to the appropriate Cauchy problems and then
carefully studying their difference via the PDEs that they satisfy and
getting bounds for the required terms based on analyzing appropriate
Poisson equations via It\^{o}'s lemma. Our work is different: we study the fully coupled case, we obtain rates of convergence in the Wasserstein metric directly and we use tools from Malliavin analysis.  Instead of studying the related Cauchy problems, we study the first and second order Malliavin derivatives of slow and fast motion with respect to the driving noises. It is worth pointing out that considering the fully coupled case results in the need to appropriately bound a number of Malliavin derivative related terms that would have been absent otherwise. Our method of proof relies on  generic objects (i.e., Malliavin derivatives) and can be useful in generic settings in which limit fluctuations are Gaussian and Malliavin derivatives exist.

Lastly, we mention that we present the results of this paper when both slow and fast component are in dimension one. The proofs make it clear that the results are valid in any finite dimension. However, we have chosen to present the result in the $1+1$ dimension in order to avoid unnecessary complicated notation as far as the technical results of Sections \ref{S:FirstOrderMalliavinDer}-\ref{S:SecondOrderMalliavinDer} are concerned. Hopefully, this allows the reader to focus on the essence of the arguments.

The rest of the paper is organized as follows. In Section \ref{S:MainResults}, we present our main assumptions, present the related results from Malliavin analysis that are relevant for us, and then state our main result, Theorem \ref{T:MainQuantitative}. In Section \ref{S:ProofMainResult}, we present the proof of our main Theorem \ref{T:MainQuantitative} using the delicate technical bounds proven in the later sections. In Section \ref{S:PreliminaryCLTresults}, we recall the fluctuations results from \cite{Spiliopoulos_CLT_Multiscale_2014} and we use those results to prove some of the bounds needed for Theorem \ref{T:MainQuantitative}. Sections \ref{S:FirstOrderMalliavinDer}-\ref{S:SecondOrderMalliavinDer} contain the main technical work of this section that is composed of precise bounds in terms of $\varepsilon$ and $\eta$ of the first and second order Malliavin derivatives of $X^{\varepsilon}, Y^{\eta}$ with respect to the driving noises $W^{1},W^{2}$ respectively. Section \ref{S:AuxiliaryBounds} contains auxiliary bounds that are used throughout the results of Sections \ref{S:FirstOrderMalliavinDer}-\ref{S:SecondOrderMalliavinDer}.

\section{Assumptions, methodology and main results}\label{S:MainResults}

Before presenting the methodology and the main results of this paper, let us first establish some notation and pose the assumptions on the coefficients of the model (\ref{slowfastsystem}).

The functions $c,f,\sigma$ and $\tau$ satisfy the following conditions.
\begin{assumption}
\label{A:Assumption1}
\begin{enumerate}[(i)]
\item The diffusion coefficient $\tau^{2}$ is uniformly nondegenerate.
\item  We assume that $c,\sigma\in C^{2,2}_{b}(\mathbb{R}\times\mathbb{R})$, i.e., they are uniformly bounded with bounded mixed derivatives up to order two.
  \item{The function $f(x,y)$ has two bounded derivatives in $x$ and two derivatives in $y$. The function $\tau(x,y)$ has two bounded derivatives in $x$ and $y$. For both functions $f,\tau$, all partial derivatives are H\"{o}lder continuous, with exponent $\alpha$, with respect to $y$, uniformly in $x$. Given that the following combination appears many times later on, we let $0<M<\infty$ be the constant such that
    \[
    \sup_{x,y}\left\{   |\partial_{1}f(x,y)|+ 6
    |\partial_{1}\tau(x,y)|^2 + |\partial_{2}\tau(x,y)|^2\right\}<M.
    \]}
\item{There is a uniform constant $0<K<\infty$ such that
\[\sup_{x,y}\left\{   \left[ 3|\partial_1 f|+6
  |\partial_1 \tau|^{2}+12
  |\partial_2 \tau|^{2} +4\partial_2 f\right]\left(x,y\right)\right\}\leq -K<0.    \]
}
\end{enumerate}
\end{assumption}

A few comments in regards to Assumption \ref{A:Assumption1} are in order.
 Assumption \ref{A:Assumption1} part $(iv)$ implies that there exists $K^{*} >0$ such that, for any $x,y \in \R$ we have  $\partial_2f \leq -K^{*}<0$. Consequently this immediately implies that
\[
\lim_{|y|\rightarrow\infty}\sup_{x\in\mathbb{R}}f(x,y)\cdot y=-\infty,
\]
which together with Assumption \ref{A:Assumption1} part $(i)$
guarantees the existence of a unique invariant measure associated to
the fast process $Y^{\eta}$. The stronger Assumption \ref{A:Assumption1} part $(iv)$ that requires $\partial_2 f$ to
be negative enough, allows us to obtain
appropriate bounds for the Malliavin derivatives of  $X^{\varepsilon},
Y^{\eta}$ with respect to the driving noises $W^{1},W^{2}$
respectively in Sections
\ref{S:FirstOrderMalliavinDer}-\ref{S:SecondOrderMalliavinDer}. 

The order in which $\varepsilon,\eta$ go to zero affect the results. For this reason we have the following assumption.
\begin{assumption}\label{A:Assumption3} We assume that $\eta=\eta(\varepsilon)\rightarrow 0$ as $\varepsilon\rightarrow 0$ such that
\begin{enumerate}[(i)]
\item{ $\lim_{\varepsilon\downarrow0}\sqrt{\frac{\varepsilon}{\eta}}=\gamma\in(0,\infty]$, and
}
\item{$\lim_{\varepsilon\downarrow0}\frac{\sqrt{\varepsilon}}{\sqrt{\frac{\eta}{\varepsilon}}\mathds{1}_{\gamma=\infty}+\left(\sqrt{\frac{\varepsilon}{\eta}}-\gamma\right)\mathds{1}_{\gamma\in(0,\infty)}}>0 $.
}
\end{enumerate}
\end{assumption}

We remark here that Assumption \ref{A:Assumption3}(ii) is posed in order for the fluctuations result of \cite{Spiliopoulos_CLT_Multiscale_2014} to hold with $1/\sqrt{\varepsilon}$ normalization, i.e., for $\theta^{\varepsilon}_{t}=\frac{X^{\varepsilon}_{t}-\bar{X}_{t}}{\sqrt{\varepsilon}}$ to have a non-trivial stochastic limit. If on the other hand, the limit in Assumption \ref{A:Assumption3}(ii) is zero, then in order to obtain a non-trvial fluctuations limit one normalizes the difference $X^{\varepsilon}_{t}-\bar{X}_{t}$ with the term $1/\left(\sqrt{\frac{\eta}{\varepsilon}}\mathds{1}_{\gamma=\infty}+\left(\sqrt{\frac{\varepsilon}{\eta}}-\gamma\right)\mathds{1}_{\gamma\in(0,\infty)}\right)$.
In this paper we focus for simplicity in the case considered in Assumption \ref{A:Assumption3}(ii).

Let $\mu(dy|x)$ be the invariant measure associated to the $Y^{\eta}$ process. Due to Assumption \ref{A:Assumption1} this is well defined and together with Theorem 2 of \cite{PardouxVeretennikov2} we get that it is also once continuously differentiable with respect to $x$. By known results, see for instance Theorem 2.8 in \cite{Spiliopoulos_CLT_Multiscale_2014} among others, we then have that
 \[
 X^{\varepsilon}(\cdot)\rightarrow \bar{X}_{\cdot} \text{ in } L^{2}(\Omega\times [0,T]) \text{ as }\varepsilon,\eta\downarrow 0,
 \]
 where
 \begin{align}
 d \bar{X}_{t}&=\bar{c}(\bar{X}_{t})dt, \bar{X}_{0}=x_{0}, \text{ with } \bar{c}(x)=\int_{\mathbb{R}} c(x,y)\mu(dy|x)\label{Eq:LimitingODE}
 \end{align}
is the ODE governing the limiting dynamics. Notice that by our assumptions $\bar{c}\in C^{1}(\mathbb{R})$ and thus (\ref{Eq:LimitingODE}) has a well defined unique solution.

Next, in order to control fluctuations, we need to introduce an auxiliary Poisson PDE. Let $\mathcal{L}_{y}$ be the infinitesimal generator associated with the $Y$ process. Consider the Poisson equation
 \begin{align}
 &\mathcal{L}_{y} \Phi(x,y)=c(x,y)-\bar{c}(x), \quad \int_{\mathbb{R}} \Phi(x,y)\mu(dy|x)=0\nonumber\\
 &\Phi \text{ grows at most polynomially in }y\text{ as }|y|\rightarrow\infty.\label{Eq:PoissonPDEFluctuations0}
 \end{align}
By Fredholm alternative, this PDE has a unique solution.

 Let us now review the fluctuation result for (\ref{slowfastsystem}) as presented in \cite{Spiliopoulos_CLT_Multiscale_2014}. The parameter $\delta$ in \cite{Spiliopoulos_CLT_Multiscale_2014} corresponds $\delta=\sqrt{\eta\varepsilon}$ in our present notation. It turns out that the fluctuations behavior may depend on the order in which $\varepsilon$ and $\eta$ go to zero.  The main result of  \cite{Spiliopoulos_CLT_Multiscale_2014} reads as follows in our case. We set
  \begin{align*}
 q(x,y)&= \left[\sigma^{2} +\frac{1}{\gamma^{2}}\left(\frac{\partial\Phi}{\partial y}\tau\right)^{2}\right](x,y). \nonumber
 \end{align*}
 Notice that if $\gamma=\infty$, then $q(x,y)=\sigma^{2} (x,y)$. Recall that $\bar{q}(x)=\int_{\mathbb{R}} q(x,y) \mu(d y|x)$. With these definitions at hand, we have the following fluctuations result.
 \begin{theorem}[Theorem 3.1 in  \cite{Spiliopoulos_CLT_Multiscale_2014}]\label{T:CLT2}
 Let $T>0$.  Assume that Assumptions \ref{A:Assumption1} and \ref{A:Assumption3} hold.
 The process
 $$\theta^{\varepsilon}_{t}=\frac{X^{\varepsilon}_{t}-\bar{X}_{t}}{\sqrt{\varepsilon}}$$ converges weakly in the space of continuous functions in $\mathcal{C}\left([0,T];\mathbb{R}\right)$
 to the solution of the Ornstein-Uhlenbeck type process
 \begin{eqnarray}
  d\theta_{t}&=&\bar{c}^{\prime}(\bar{X}_{t}(x_{0}))\theta_{t}dt +\bar{q}^{1/2}(\bar{X}_{t}(x_{0}))d\tilde{W}_{t}\nonumber\\
 \theta_{0}&=&0.\label{Eq:LimitingProcess0}
 \end{eqnarray}
 where $\tilde{W}$ is a one dimensional standard Wiener process.
 \end{theorem}

 Notice now that it is not difficult to solve  the SDE (\ref{Eq:LimitingProcess0}) explicitly. In particular, letting for $x \in \R$, $\Psi_{x}$  be the linearization of $\bar X$ along the orbit of $x$:
 \begin{equation*}
  \frac{d}{dt}\Psi_{x}(t)=\bar{c}^{\prime}( \bar X_t )\Psi_{x}(t), \text{  } \Psi_{x}(0)=x
 \end{equation*}
  and we shall have
 \begin{align}
 \theta_{t} &= \Psi_{x_{0}} (t) \int_{0}^{t} \left[\Psi_{x_{0}} (s)\right]^{-1} \bar{q}^{1/2}( \bar X_s )  d\tilde{W}_{s} , \quad t \geq 0. \label{eqn: thetadef_00}
 \end{align}

Thus, we obtain that the limit fluctuations are Gaussian with zero mean (i.e. $\mu_{t}=0$) in this model. Since the fluctuations are Gaussian, the idea is to make use of the following result about the Wasserstein distance between an
 element of $\mathbb{D}^{2,4}$ and the Gaussian distribution. At this point, let us recall the definition of the Wasserstein distance $d_W$ between the laws of two random variables, say $U$ and $Z$,
 \[
  d_W(U,Z)=\sup_{f\colon \|f\|_{\on{Lip}}\leq 1}|\mathbb{E}[f(U)]-\mathbb{E}[f(Z)]|.
 \]

 \begin{theorem}[Corollary 4.2 in \cite{nourdin_second_2009}]
   \label{secondorderpoincare}
 Let $F \in \mathbb{D}^{2,4}$ with $\mathbb{E}(F)=\mu$ and
 $\on{Var}(F)=\sigma^2$. Assume that $N \sim \mathscr{N}(\mu,\sigma^2)$. Then,
 \begin{equation*}
 d_W(F,N) \leq \frac{\sqrt{10}}{2\sigma^2}\mathbb{E}\left[ \norm{D^2 F
     \otimes_1 D^2 F}_{\frak{H}^{\otimes 2}}^2 \right]^{\frac{1}{4}}\mathbb{E}\left[ \norm{D F}_{\frak{H}}^4 \right]^{\frac{1}{4}}.
 \end{equation*}
 \end{theorem}
 \noindent In the above inequality, for any $f,g \in \frak{H}^2$, $f \otimes_1 g$ is defined as
 \begin{equation}
   \label{deffirstcontraction}
 [f \otimes_1 g] (x,y) = \int_0^1 f(x,z)g(y,z)dz.
 \end{equation}
Note that $DF$ and $D^{2}F$ are the vector and matrix, respectively,
of the first and second Malliavin derivatives of $F$ with respect to
the Brownian motion $\left( W^{1},W^{2} \right)$. We recall the definitions of these objects (e.g.,  $ \mathbb{D}^{2,4}$, $\frak{H}^2$ and the Malliavin derivatives) in more detail in Section \ref{S:ProofMainResult}.
\\~\\
 The idea is to apply the above theorem to
 \begin{equation*}
 \theta^{\varepsilon}_t = \frac{1}{\sqrt{\varepsilon}} \left(
   X_t^{\varepsilon} - \bar{X}_t \right),
 \end{equation*}
 where the underlying (two dimensional) system of stochastic differential
 equations is given by (\ref{slowfastsystem}).  As $\theta^{\varepsilon}_t$ does not have the same expectation and
 variance as the limiting Gaussian distribution (as is required in
 Theorem \ref{secondorderpoincare}), we introduce
 \begin{equation*}
 \tilde{\theta}^{\varepsilon}_t = \frac{\sigma_t}{\sqrt{\on{Var}\left(
       \theta^{\varepsilon}_t \right)}} \left[\theta^{\varepsilon}_t -
   \mathbb{E}\left( \theta^{\varepsilon}_t \right) + \mu_t  \right],
 \end{equation*}
 where $\mu_t$ and $\sigma_t^2$ are the expectation and variance of the
 limiting Gaussian distribution. Note that
 $\mathbb{E}(\tilde{\theta}^{\varepsilon}_t) = \mu_t$ and
 $\on{Var}(\tilde{\theta}^{\varepsilon}_t) = \sigma_t^2$. We can then
 write (assuming that $\theta^{\varepsilon}_t \in \mathbb{D}^{2,4}$)
 \begin{equation}
 d_W \left(\theta^{\varepsilon}_t, \mathscr{N} \left(\mu_t,\sigma_t^2
   \right)  \right) \leq d_W \left(\theta^{\varepsilon}_t, \tilde{\theta}^{\varepsilon}_t  \right) + d_W \left(\tilde{\theta}^{\varepsilon}_t, \mathscr{N} \left(\mu_t,\sigma_t^2
   \right)  \right).\label{Eq:WassensteinMetricDecomp}
 \end{equation}

Notice that in our case we have by (\ref{eqn: thetadef_00}) that $\mu_t=0$ and the limiting variance is
 \begin{align}
 \sigma_{t}^{2}&= \int_{0}^{t} e^{\int_{s}^{t}2 \bar{c}^{\prime}(\bar{X}_{u})du} \bar{q}( \bar X_s )ds\label{Eq:LimitingVariance0}
 \end{align}

By carefully analyzing the right hand side of (\ref{Eq:WassensteinMetricDecomp}) we can then establish the following result which is also the main Theorem of this work.
\begin{theorem}\label{T:MainQuantitative}
 Let $0<T<\infty$.  Assume that Assumptions \ref{A:Assumption1} and \ref{A:Assumption3} hold. Then, there is a finite constant $C<\infty$ that depends on $T$  such that for all $0<\zeta<1/2$ and for $\varepsilon,\eta$ sufficiently small one has
\begin{align}
\sup_{t\in(0,T]}d_W \left(\theta^{\varepsilon}_t, \mathscr{N} \left(\mu_t,\sigma_t^2
  \right)  \right) &\leq
                     C\left(\eta^{1/4}+\varepsilon^{1/4}+\left(\frac{\eta}{\varepsilon}-\frac{1}{\gamma^{2}}\right)^{1/2}+
                     \left( \frac{\eta}{\varepsilon} \right)^{1/2}\eta^{1/2-\zeta}+\varepsilon^{1/2-\zeta}\right)\nonumber\\
  & \quad+C\left(\left( \frac{\eta}{\varepsilon} \right)^{1/2}\eta^{1/4}+\left( \frac{\eta}{\varepsilon} \right)\eta^{1/4}+ \left(1+\frac{\eta}{\varepsilon}\right)e^{-\frac{K}{16\eta}T}\right).\nonumber
\end{align}
\end{theorem}

It is easy to see that under Assumption \ref{A:Assumption3}(i), the latter goes to zero as $\varepsilon,\eta$ at a rate that ultimately depends on Assumption \ref{A:Assumption3}(i).

The overview of the proof of Theorem \ref{T:MainQuantitative} is in Section \ref{S:ProofMainResult} which is using the main estimates that are then being derived in Sections \ref{S:FirstOrderMalliavinDer}, \ref{S:SecondOrderMalliavinDer} and \ref{S:AuxiliaryBounds}.

\section{Proof of main result}\label{S:ProofMainResult}

In this section we present the overview of the proof of Theorem \ref{T:MainQuantitative} referring to the detailed estimates established in Sections \ref{S:FirstOrderMalliavinDer}, \ref{S:SecondOrderMalliavinDer} and \ref{S:AuxiliaryBounds} as needed.

\subsection{Preparation and necessary notions from Malliavin analysis}
~\\
\noindent We outline here the main tools of Malliavin calculus needed in this
paper. For a complete treatment of this topic, we refer the reader to
\cite{nualart_malliavin_2006}.
\\~\\
Let $\mu$ be the product measure
between the Lebesgue measure on $\left[ 0,T \right]$ and the uniform
measure on $\left\{ 1,2 \right\}$ (which gives mass one to each of the
points $1,2$). Then, define
\begin{equation*}
\frak{H} = L^2
 \left([0,T]\times \left\{ 1,2 \right\},\mu  \right) \cong L^2
 \left([0,T];\R^2  \right)
\end{equation*}
and consider the isonormal Gaussian
 process $\left\{ W(h)
   \colon h \in \frak{H} \right\}$, that is, the collection of centered
 Gaussian random variables with covariance given by
\begin{equation*}
\mathbb{E}\left( W(h)W(g) \right) = \left\langle h,g \right\rangle_{\frak{H}}.
\end{equation*}
In this situation, we have that
\begin{equation*}
\left\{W_t = \left( W^1_t,W^2_t
  \right) = \left(
  W(\mathds{1}_{\left[ 0,t \right]}\otimes \mathds{1}_{\left\{ 1
    \right\}}), W(\mathds{1}_{\left[ 0,t \right]}\otimes \mathds{1}_{\left\{ 2
    \right\}}) \right)
  \colon t \in [0,T] \right\}
\end{equation*}
is a standard 2-dimensional Brownian motion on $[0,T]$, which is the
process we work with in this paper.
\\~\\
Denote by $\mathcal{S}$ the set of smooth cylindrical random variables
of the form $F = f
\left( W(\varphi_1), \cdots , W(\varphi_n) \right)$, where $n \geq 1$, $\{\varphi_i\}^n_{i=1} \subset
\mathfrak{H}$, and $f \in C_b^{\infty} \left(
  \mathbb{R}^n \right)$ ($f$ and all of its partial derivatives of all
orders are bounded functions). The Malliavin derivative of such a smooth cylindrical random variable
$F$ is defined as the $\mathfrak{H}$-valued random variable given by
\begin{equation*}
DF = \sum_{i=1}^n \frac{\partial f}{\partial x_i} \left( W^H(\varphi_1),
  \cdots, W^H(\varphi_n) \right)\varphi_i.
\end{equation*}
The derivative operator $D$ is closable from $L^2(\Omega)$
into $L^2(\Omega ; \mathfrak{H})$, and we continue to denote by $D$
its closure, the domain of which we denote by $\mathbb{D}^{1,2}$, and which is a Hilbert space in the Sobolev-type norm
\begin{equation*}
\left\lVert F \right\rVert_{1,2}^2 = E (F^2) + E \left( \left\lVert DF \right\rVert_{\mathfrak{H}}^2 \right).
\end{equation*}
Similarly, one can obtain a derivative operator $D
\colon\mathbb{D}^{1, 2}(\mathfrak{H}) \to L^2(\Omega; \mathfrak{H}
\otimes \mathfrak{H})$ as the closure of $D \colon L^2(\Omega;
\mathfrak{H}) \to L^2(\Omega; \mathfrak{H} \otimes \mathfrak{H})$. We
then set $D^2F = D(DF)$. Note that more generally with $p > 1$ one can
analogously obtain $\mathbb{D}^{1, p}$ as Banach spaces of Sobolev
type by working with $L^p(\Omega)$.
\\~\\
With this notation, one should view the Malliavin derivative of a
random variable $F$ which is measurable with respect to the
$\sigma$-field generated by our 2-dimensional Brownian motion $W$ as the two
 dimensional vector
 \begin{equation*}
 DF = \left( D^{W^1} F, D^{W^2} F \right),
 \end{equation*}
 where $D^{W^1}$ denotes the Malliavin derivative with respect to $W^1$ and
 $D^{W^2}$ denotes the Malliavin derivative with respect to $W^2$. The
 second order Malliavin derivative $D^2 F$ (when it exists) is then an
 element of $L^2
 \left( \Omega ; \frak{H}^{\otimes 2} \right) \cong L^2
 \left( \Omega ; L^2
 \left([0,T]^2; M_2(\R) \right) \right)$, (usually identified with the space
 of random Hilbert-Schmidt operators on $\frak{H} = L^2
 \left([0,T];\R^2  \right)$) given by
 \begin{equation}
   \label{secondorderMallder}
 D^2 F= \begin{pmatrix}
 D^{W^1,W^1}F & D^{W^1,W^2}F \\
 D^{W^2,W^1}F & D^{W^2,W^2}F
 \end{pmatrix},
 \end{equation}
 where the action of the second order Malliavin derivatives
 $D^{W^i,W^j}$ can be identified with that of the iterated first order
 derivatives $D^{W^i}D^{W^j}$. We also write $\mathbb{D}^{k,p}$ for the
 Sobolev space of random variables in $L^p(\Omega)$ that are $k$ times
 Malliavin differentiable.
\\~\\
Now we go back to (\ref{Eq:WassensteinMetricDecomp}).  Starting with the first term on the right hand side of (\ref{Eq:WassensteinMetricDecomp}), we have
 \begin{align*}
 d_W \left(\theta^{\varepsilon}_t, \tilde{\theta}^{\varepsilon}_t
 \right) &\leq \mathbb{E} \left( \abs{\theta^{\varepsilon}_t -
           \tilde{\theta}^{\varepsilon}_t} \right)\\
   &= \mathbb{E} \left( \abs{\frac{\sigma_t}{\sqrt{\on{Var}\left(
       \theta^{\varepsilon}_t \right)}} \theta^{\varepsilon}_t -
     \theta^{\varepsilon}_t +  \frac{\sigma_t}{\sqrt{\on{Var}\left(
       \theta^{\varepsilon}_t \right)}}\left( \mu_t - \mathbb{E}\left(
     \theta^{\varepsilon}_t \right) \right)} \right)\\
   &\leq \mathbb{E}\left(\abs{
     \theta^{\varepsilon}_t} \right) \abs{1 - \frac{\sigma_t}{\sqrt{\on{Var}\left(
       \theta^{\varepsilon}_t \right)}}} + \frac{\sigma_t}{\sqrt{\on{Var}\left(
       \theta^{\varepsilon}_t \right)}} \abs{\mu_t - \mathbb{E}\left(
     \theta^{\varepsilon}_t \right)}.
 \end{align*}
 Continuing with the second term, we have by Theorem
 \ref{secondorderpoincare} that
 \begin{equation*}
 d_W \left(\tilde{\theta}^{\varepsilon}_t, \mathscr{N} \left(\mu_t,\sigma_t^2
   \right)  \right) \leq \frac{\sqrt{10}}{2\sigma_t^2}\mathbb{E}\left[ \norm{D^2 \tilde{\theta}^{\varepsilon}_t
     \otimes_1 D^2 \tilde{\theta}^{\varepsilon}_t}_{\frak{H}^{\otimes 2}}^2 \right]^{\frac{1}{4}}\mathbb{E}\left[ \norm{D \tilde{\theta}^{\varepsilon}_t}_{\frak{H}}^4 \right]^{\frac{1}{4}}.
 \end{equation*}
 Note that by the definition of $\tilde{\theta}^{\varepsilon}_t$, it
 holds that
 \begin{equation*}
 D \tilde{\theta}^{\varepsilon}_t = \frac{\sigma_t}{\sqrt{\on{Var}\left(
       \theta^{\varepsilon}_t \right)}\sqrt{\varepsilon}} D
 X_t^{\varepsilon}\quad \mbox{and}\quad D^2 \tilde{\theta}^{\varepsilon}_t = \frac{\sigma_t}{\sqrt{\on{Var}\left(
       \theta^{\varepsilon}_t \right)}\sqrt{\varepsilon}} D^2
 X_t^{\varepsilon},
 \end{equation*}
 so that, in total, we get
 \begin{align*}
 d_W \left(\theta^{\varepsilon}_t, \mathscr{N} \left(\mu_t,\sigma_t^2
   \right)  \right) &\leq \mathbb{E}\left(\abs{
     \theta^{\varepsilon}_t} \right) \abs{1 - \frac{\sigma_t}{\sqrt{\on{Var}\left(
       \theta^{\varepsilon}_t \right)}}} + \frac{\sigma_t}{\sqrt{\on{Var}\left(
       \theta^{\varepsilon}_t \right)}} \abs{\mu_t - \mathbb{E}\left(
                      \theta^{\varepsilon}_t \right)}\\
   & + \frac{\sqrt{10}}{2 \varepsilon \on{Var}\left(
       \theta^{\varepsilon}_t \right)}\mathbb{E}\left[ \norm{D^2 X_t^{\varepsilon}
     \otimes_1 D^2 X_t^{\varepsilon}}_{\frak{H}^{\otimes 2}}^2 \right]^{\frac{1}{4}}\mathbb{E}\left[ \norm{D X_t^{\varepsilon}}_{\frak{H}}^4 \right]^{\frac{1}{4}}.
 \end{align*}

Summarizing:
\begin{enumerate}[(i)]
\item{We need good bounds on the convergence of mean and variance of the prelimit process to the limit ones. Namely, we need to bound $\abs{1 - \frac{\sigma_t}{\sqrt{\on{Var}\left(
       \theta^{\varepsilon}_t \right)}}}$ and $\abs{\mu_t - \mathbb{E}\left(
                      \theta^{\varepsilon}_t \right)}$. This is the content of Subsection \ref{SS:BoundsMeanVariance}.}
\item{We need good bounds for the first and second order Malliavin derivatives. Namely, we need to bound the term $\mathbb{E}\left[ \norm{D^2 X_t^{\varepsilon}
     \otimes_1 D^2 X_t^{\varepsilon}}_{\frak{H}^{\otimes 2}}^2 \right]^{\frac{1}{4}}\mathbb{E}\left[ \norm{D X_t^{\varepsilon}}_{\frak{H}}^4 \right]^{\frac{1}{4}}$. This is the content of Subsection \ref{SS:BoundsMalliavinDer}.}
\item{Lastly, we need to combine the different bounds to bound $d_W \left(\theta^{\varepsilon}_t, \mathscr{N} \left(\mu_t,\sigma_t^2
   \right)  \right)$. This is the content of Subsection \ref{SS:BoundsCompilation}.}
\end{enumerate}

 \subsection{Bounds associated to convergence of the pre-limit mean and variance}\label{SS:BoundsMeanVariance}

By the expression (\ref{Eq:PrelimitFluctuationsSolution}), the bound (\ref{Eq:Rbound0}) and Remark \ref{R:BoundQterm}, we obtain that there is a constant $C<\infty$ such that for any $\zeta>0$
 \begin{align}
 \sup_{t\in[0,T]}|\mathbb{E}\left(  \theta^{\varepsilon}_{t}\right)-
 \mu_{t}|&=\sup_{t\in[0,T]}|\mathbb{E}\left(\theta^{\varepsilon}_{t}\right)|\leq C \left( \frac{\eta^{1-\zeta}}{\sqrt{\varepsilon}}+\varepsilon^{1/2-\zeta}\right).\label{Eq:FirstMomentBound0}
 \end{align}

 Next, we want to get an estimate on the difference of the prelimit and limit variance of $\theta^{\varepsilon}_{t}$.
 By (\ref{eqn: thetadef_00}), the limiting variance is given by (\ref{Eq:LimitingVariance0}), i.e.,
  $\sigma_{t}^{2}= \int_{0}^{t} e^{\int_{s}^{t}2 \bar{c}^{\prime}(\bar{X}_{u})du} \bar{q}( \bar X_s )ds$.  Next, let us compute the prelimit variance of $\theta^{\varepsilon}_{t}$. For the second moment, (\ref{Eq:PrelimitFluctuationsSolution}) gives
 \begin{align}
 \mathbb{E}\left(|\theta^{\varepsilon}_{t}|^{2}\right)&=\mathbb{E}\bigg(\left|\Psi_{x_{0}} (t) \int_{0}^{t} \left[\Psi_{x_{0}} (s)\right]^{-1} \sigma(X^{\varepsilon}_{s}, Y^{\eta}_{s})  dW^1_{s}-\sqrt{\frac{\eta}{\varepsilon}}\Psi_{x_{0}} (t)\int_0^{t}\left[\Psi_{x_{0}} (s)\right]^{-1}\left(\partial_{y}\Phi\tau\right)(X^{\varepsilon}_s,Y^{\eta}_s)dW^{2}_{s}\right.\nonumber\\
 &\qquad\left.+\Psi_{x_{0}} (t)\int_0^t \left[\Psi_{x_{0}} (s)\right]^{-1}\frac{1}{\sqrt{\varepsilon}}\Lambda[c]\parbar{ \bar X_s,X^{\varepsilon}_s}ds+R^{\varepsilon}(t;\Psi)\right|^{2}\bigg)
 \nonumber\\
 &\leq 2 \mathbb{E}\bigg(\left|\Psi_{x_{0}} (t) \int_{0}^{t} \left[\Psi_{x_{0}} (s)\right]^{-1} \sigma(X^{\varepsilon}_{s}, Y^{\eta}_{s})  dW^1_{s}-\sqrt{\frac{\eta}{\varepsilon}}\Psi_{x_{0}} (t)\int_0^{t}\left[\Psi_{x_{0}} (s)\right]^{-1}\left(\partial_{y}\Phi\tau\right)(X^{\varepsilon}_s,Y^{\eta}_s)dW^{2}_{s}\right|^{2}\nonumber\\
 &\qquad+2\mathbb{E}\left|\Psi_{x_{0}} (t)\int_0^t \left[\Psi_{x_{0}} (s)\right]^{-1}\frac{1}{\sqrt{\varepsilon}}\Lambda[c]\parbar{ \bar X_s,X^{\varepsilon}_s}ds+R^{\varepsilon}(t;\Psi)\right|^{2}\bigg).\nonumber
 \end{align}
Now, using the bounds from (\ref{Eq:Rbound0}) and (\ref{Eq:Qbound0}), we obtain that there is some finite constant $C<\infty$ such that for $\zeta>0$
 \begin{align}
 \mathbb{E}\left(|\theta^{\varepsilon}_{t}|^{2}\right)&\leq \mathbb{E}\left( \int_{0}^{t} e^{\int_{s}^{t}2\bar{c}^{\prime}(\bar{X}_{u})du} q( X^{\varepsilon}_s, Y^{\eta}_{s} )ds\right)\nonumber\\
 &\quad+\left(\frac{\eta}{\varepsilon}-\frac{1}{\gamma^{2}}\right)\mathbb{E}\left( \int_{0}^{t} e^{\int_{s}^{t}2\bar{c}^{\prime}(\bar{X}_{u})du} (\partial_{y}\Phi\tau)^{2}( X^{\varepsilon}_s, Y^{\eta}_{s} )ds\right)\nonumber\\
 &\quad + C \left( \frac{\eta^{2-\zeta}}{\varepsilon}+\varepsilon^{1-\zeta}\right).\nonumber
 \end{align}
 Hence, using (\ref{Eq:FirstMomentBound0})  and (\ref{Eq:LimitingVariance0}) we obtain
 \begin{align}
 \on{Var}\left(\theta^{\varepsilon}_t \right)-\sigma^{2}_{t}
 &=\mathbb{E}\left(|\theta^{\varepsilon}_{t}|^{2}\right)-|\mathbb{E}\left(\theta^{\varepsilon}_{t}\right)|^{2}-\sigma^{2}_{t}\nonumber\\
 &\leq\mathbb{E}\left( \int_{0}^{t} e^{\int_{s}^{t}2\bar{c}^{\prime}(\bar{X}_{u})du} q( X^{\varepsilon}_s, Y^{\eta}_{s} )ds\right)-\int_{0}^{t} e^{\int_{s}^{t}2\bar{c}^{\prime}(\bar{X}_{u})du} \bar{q}( \bar X_s )ds\nonumber\\
 &\quad+\left(\frac{\eta}{\varepsilon}-\frac{1}{\gamma^{2}}\right)\mathbb{E}\left( \int_{0}^{t} e^{\int_{s}^{t}2\bar{c}^{\prime}(\bar{X}_{u})du} (\partial_{y}\Phi\tau)^{2}( X^{\varepsilon}_s, Y^{\eta}_{s} )ds\right)\nonumber\\
 &\quad + C \left( \frac{\eta^{2-\zeta}}{\varepsilon}+\varepsilon^{1-\zeta}\right).\nonumber
 \end{align}
 Now, by Lemma 10 in \cite{GS17} we have that
 \begin{align*}
 \mathbb{E}\left( \sup_{t\leq T}\left|\int_{0}^{t} e^{\int_{s}^{t}2\bar{c}^{\prime}(\bar{X}_{u})du} q( X^{\varepsilon}_s, Y^{\eta}_{s} )ds-\int_{0}^{t} e^{\int_{s}^{t}2\bar{c}^{\prime}(\bar{X}_{u})du} \bar{q}( \bar X_s )ds\right|\right)&\leq C(\sqrt{\eta}+\sqrt{\varepsilon}).
 \end{align*}
 Therefore, all in all, we get the estimate
 \begin{align*}
 \on{Var}\left(\theta^{\varepsilon}_t \right)-\sigma^{2}_{t}
 &\leq C \left( \sqrt{\eta}+\sqrt{\varepsilon}+\left(\frac{\eta}{\varepsilon}-\frac{1}{\gamma^{2}}\right)+ \frac{\eta^{2-\zeta}}{\varepsilon}+\varepsilon^{1-\zeta}\right).
 \end{align*}
 And, since we want to compare standard deviations, we have that for $\varepsilon,\eta$ small enough
 \begin{align}
 \sup_{t\in[0,T]}\left|\sqrt{\on{Var}\left(\theta^{\varepsilon}_t \right)}-\sqrt{\sigma^{2}_{t}}\right|
 &\leq C \sqrt{ \sqrt{\eta}+\sqrt{\varepsilon}+ \left(\frac{\eta}{\varepsilon}-\frac{1}{\gamma^{2}}\right)+ \frac{\eta^{2-\zeta}}{\varepsilon}+\varepsilon^{1-\zeta}}.\label{Eq:DiffereceLim_PreLimSD}
 \end{align}

\subsection{Compiling the bounds associated to the Malliavin derivatives}\label{SS:BoundsMalliavinDer}

We begin establishing explicit forms of the quantities $\mathbb{E}\left[ \norm{D
     X_t^{\varepsilon}}_{\frak{H}}^4 \right]$ and $\mathbb{E}\left[ \norm{D^2 X_t^{\varepsilon}
     \otimes_1 D^2 X_t^{\varepsilon}}_{\frak{H}^{\otimes 2}}^2
 \right]$ that we need to control. For the first term, we have
 \begin{equation}
     \label{firsttermtocontrol}
 \mathbb{E}\left[ \norm{D
     X_t^{\varepsilon}}_{\frak{H}}^4 \right] = \mathbb{E}\left[ \left(  \int_0^T
   \abs{D_u X_t^{\varepsilon}  }^2 du \right)^2 \right] =
   \int_{[0,T]^2}^{} \mathbb{E}\left[ \abs{D_u X_t^{\varepsilon}
   }^2 \abs{D_v X_t^{\varepsilon}  }^2 \right]du dv.
 \end{equation}
 Now, using the definition of the
 contraction operator $\otimes_1$ given in \eqref{deffirstcontraction},
 we have
 \begin{equation*}
 D^2 X_t^{\varepsilon}
     \otimes_1 D^2 X_t^{\varepsilon} = \int_0^T
     D_{u,v}^2 X_t^{\varepsilon} D_{u,w}^2 X_t^{\varepsilon} du,
 \end{equation*}
 where, by the fact that the second order Malliavin derivative is
 given by \eqref{secondorderMallder}, $D_{u,v}^2 X_t^{\varepsilon} D_{u,w}^2 X_t^{\varepsilon}$ is the
 $2 \times 2$ matrix given by
 \begin{align*}
 D_{u,v}^2 X_t^{\varepsilon} D_{u,w}^2 X_t^{\varepsilon} &= \left(
   \sum_{k=1}^2 D_{u,v}^{W^i,W^k}X_t^{\varepsilon}
   D_{u,w}^{W^k,W^j}X_t^{\varepsilon} \colon 1 \leq i,j \leq 2
                                                 \right)\\
   &= \begin{pmatrix}
 \displaystyle \sum_{k=1}^2 D_{u,v}^{W^1,W^k}X_t^{\varepsilon}
   D_{u,w}^{W^k,W^1}X_t^{\varepsilon} & \displaystyle \sum_{k=1}^2 D_{u,v}^{W^1,W^k}X_t^{\varepsilon}
   D_{u,w}^{W^k,W^2}X_t^{\varepsilon} \\
 \displaystyle \sum_{k=1}^2 D_{u,v}^{W^2,W^k}X_t^{\varepsilon}
   D_{u,w}^{W^k,W^1}X_t^{\varepsilon} & \displaystyle \sum_{k=1}^2 D_{u,v}^{W^2,W^k}X_t^{\varepsilon}
   D_{u,w}^{W^k,W^2}X_t^{\varepsilon}
 \end{pmatrix},
 \end{align*}
 so that, in total, $D^2 X_t^{\varepsilon}
     \otimes_1 D^2 X_t^{\varepsilon}$ is the
 $2 \times 2$ matrix given by
 \begin{align*}
 D^2 X_t^{\varepsilon}
     \otimes_1 D^2 X_t^{\varepsilon} &= \left(
   \sum_{k=1}^2 \int_0^T D_{u,v}^{W^i,W^k}X_t^{\varepsilon}
   D_{u,w}^{W^k,W^j}X_t^{\varepsilon} du \colon 1 \leq i,j \leq 2
                                                 \right)\\
   &= \begin{pmatrix}
 \displaystyle \sum_{k=1}^2 \int_0^T D_{u,v}^{W^1,W^k}X_t^{\varepsilon}
   D_{u,w}^{W^k,W^1}X_t^{\varepsilon}du & \displaystyle \sum_{k=1}^2 \int_0^T D_{u,v}^{W^1,W^k}X_t^{\varepsilon}
   D_{u,w}^{W^k,W^2}X_t^{\varepsilon}du \\
 \displaystyle \sum_{k=1}^2 \int_0^T D_{u,v}^{W^2,W^k}X_t^{\varepsilon}
   D_{u,w}^{W^k,W^1}X_t^{\varepsilon}du & \displaystyle \sum_{k=1}^2 \int_0^T D_{u,v}^{W^2,W^k}X_t^{\varepsilon}
   D_{u,w}^{W^k,W^2}X_t^{\varepsilon}du
 \end{pmatrix}.
 \end{align*}
 In view of this, we can finally write
 \begin{align}
   \label{secondtermtocontrol}
 & \mathbb{E}\left[ \norm{D^2 X_t^{\varepsilon}
     \otimes_1 D^2 X_t^{\varepsilon}}_{\frak{H}^{\otimes 2}}^2
 \right] = \mathbb{E}\left[   \int_{[0,T]^2}^{} \sum_{i,j=1}^2 \left(\sum_{k=1}^2 \int_0^T D_{u,v}^{W^i,W^k}X_t^{\varepsilon}
           D_{u,w}^{W^k,W^j}X_t^{\varepsilon} du  \right)^2dv dw\right]
   \nonumber \\
   &\qquad\qquad = \sum_{i,j,k,p=1}^2
     \int_{[0,T]^4}^{}\mathbb{E}\left(   D_{u,v}^{W^i,W^k}X_t^{\varepsilon}
           D_{u,w}^{W^k,W^j}X_t^{\varepsilon} D_{s,v}^{W^i,W^p}X_t^{\varepsilon}
           D_{s,w}^{W^p,W^j}X_t^{\varepsilon}\right)du ds dv dw.
 \end{align}
 We can now focus on computing $D X_t^{\varepsilon}$ and $D^2
 X_t^{\varepsilon}$, which are the quantities that appear in the two
 terms \eqref{firsttermtocontrol} and \eqref{secondtermtocontrol} that
 we need to control.
 \begin{remark}
 Under Assumption \ref{A:Assumption1}, the coefficients of the system \eqref{slowfastsystem}
 satisfy the assumptions of \cite[Corollary 3.5]{IRS19}, which
 guarantees that
 \begin{equation*}
 \sup_{0 \leq r \leq T} \mathbb{E}\left( \sup_{r \leq t \leq T}
   \abs{D_r^{W^j}X_t^{\varepsilon} }^p \right) +  \sup_{0 \leq r \leq T} \mathbb{E}\left( \sup_{r \leq t \leq T}
   \abs{D_r^{W^j}Y_t^{\eta} }^p \right) < \infty
 \end{equation*}
 and
 \begin{equation*}
 \sup_{0 \leq r_1,r_2 \leq T} \mathbb{E}\left( \sup_{r_1 \vee r_2 \leq t \leq T}
   \abs{D_{r_1,r_2}^{W^{j_1},W^{j_2}}X_t^{\varepsilon} }^p \right) +  \sup_{0 \leq r_1,r_2 \leq T} \mathbb{E}\left( \sup_{r_1 \vee r_2 \leq t \leq T}
   \abs{D_{r_1,r_2}^{W^{j_1},W^{j_2}}Y_t^{\eta} }^p \right) < \infty
 \end{equation*}
 for all $p \geq 1$.
\end{remark}
~\\
We can now state the required bounds in Lemmas \ref{L:FirstDerivativeMomentBound} and \ref{L:SecondDerivativeMomentBound}.
\begin{lemma}\label{L:FirstDerivativeMomentBound}
Let Assumptions \ref{A:Assumption1} and \ref{A:Assumption3} hold and let $T<\infty$ be given.  Then, we have that there is a finite constant $C<\infty$ (which depends on $T$) such that
\begin{align}
    \label{firsttermtocontrol1}
\sup_{t\in[0,T]}\mathbb{E}\left[ \norm{D
    X_t^{\varepsilon}}_{\frak{H}}^4 \right]  &
  \leq C\left(\varepsilon^{2}+\eta^{2}\right).
\end{align}

\end{lemma}
\begin{proof}[Proof of Lemma \ref{L:FirstDerivativeMomentBound}]

Let us define
\begin{align*}
m_{r}(t)&=\left|D^{W^{1}}_{r}X^{\varepsilon}_{t}\right|^{2}+\left|D^{W^{2}}_{r}X^{\varepsilon}_{t}\right|^{2}.
\end{align*}
By Lemmas \ref{boundonEDW1Xs2p} and \ref{boundonEDW2Xs2p} we have that
\begin{align*}
\mathbb{E}\left( \sup_{0\leq r\leq t\leq T} m_{r}(t)\right)&\leq C\left(\varepsilon+\eta\right).
\end{align*}
Finally, going back to (\ref{firsttermtocontrol}), we have by the
Cauchy-Schwarz inequality,
\begin{align*}
\mathbb{E}\left[ \norm{D
    X_t^{\varepsilon}}_{\frak{H}}^4 \right]  &=
  \int_{[0,T]^2}^{} \mathbb{E}\left[ m_u (t)m_{v}(t)   \right]du dv
  \leq C\left(\varepsilon^{2}+\eta^{2}\right),
\end{align*}
hence concluding the proof of the lemma.
\end{proof}
~\\
Let us now turn our attention to the bound for the second order Malliavin derivatives associated to the term $\mathbb{E}\left[  \norm{D^2 X_t^{\varepsilon}\otimes_1 D^2 X_t^{\varepsilon}}_{\frak{H}^{\otimes 2}}^2\right] $.

\begin{lemma}\label{L:SecondDerivativeMomentBound}
Let Assumptions \ref{A:Assumption1} and \ref{A:Assumption3} hold and let $T<\infty$ be given.  Then, we have that there is a finite constant $C<\infty$ (depending on $T$) such that
\begin{align}
    \label{secondtermtocontrol1}
\sup_{t\in[0,T]}\mathbb{E}\left[  \norm{D^2 X_t^{\varepsilon}
    \otimes_1 D^2 X_t^{\varepsilon}}_{\frak{H}^{\otimes 2}}^2  \right] &
  \leq C\left(\varepsilon^{4}+\eta^{3}+\varepsilon^{2}\eta+ (\varepsilon^{2}+\eta^{2})e^{-\frac{K}{4\eta}T} \right).
\end{align}
\end{lemma}
\begin{proof}[Proof of Lemma \ref{L:SecondDerivativeMomentBound}]
By the general expression \eqref{secondtermtocontrol} and Cauchy-Schwarz we have
\begin{align}
&\mathbb{E}\left[   \norm{D^2 X_t^{\varepsilon}
    \otimes_1 D^2 X_t^{\varepsilon}}_{\frak{H}^{\otimes 2}}^2\right]=\nonumber\\
& = \sum_{i,j,k,p=1}^2
    \int_{[0,T]^4}^{}\mathbb{E}\left(   D_{u,v}^{W^i,W^k}X_t^{\varepsilon}
          D_{u,w}^{W^k,W^j}X_t^{\varepsilon} D_{s,v}^{W^i,W^p}X_t^{\varepsilon}
          D_{s,w}^{W^p,W^j}X_t^{\varepsilon}\right)du ds dv dw\nonumber\\
&\leq
\sum_{i,j,k,p=1}^2
    \int_{[0,T]^4}\left(\mathbb{E}\left| D_{u,v}^{W^i,W^k}X_t^{\varepsilon}\right|^{4}\right)^{1/4}
          \left(\mathbb{E}\left|D_{u,w}^{W^k,W^j}X_t^{\varepsilon}\right|^{4}\right)^{1/4}\left(\mathbb{E}\left|D_{s,v}^{W^i,W^p}X_t^{\varepsilon}\right|^{4}\right)^{1/4}\nonumber\\
          &\qquad\qquad\qquad\qquad\qquad\qquad\qquad\qquad\qquad\qquad\qquad\qquad
          \left(\mathbb{E}\left|D_{s,w}^{W^p,W^j}X_t^{\varepsilon}\right|^{4}\right)^{1/4}du ds dv dw.\label{Eq:BoundSecondDer1}
\end{align}
Let us now investigate how the upper bounds of (\ref{Eq:BoundSecondDer1}) looks like in terms of $\varepsilon$ and $\eta$. By Proposition \ref{boundonEDW1W1Xs2p}, we have with $p=2$, for some
constant $C >0$,
\begin{equation*}
\mathbb{E}\left( \abs{D^{W^{1},W^{1}}_{r_1,r_2} X_t^{\varepsilon}}^{4}
\right) \leq C \varepsilon^{4}.
\end{equation*}
By Proposition \ref{P:BoundSecondMalliavinDW1DW2X}, we have with $p=2$, for some
constant $C >0$,
\begin{align*}
\mathbb{E}\left( \abs{D^{W^{1},W^{2}}_{r_1,r_2} X_t^{\varepsilon}}^{4}\right)&\leq C \left[\varepsilon^{4}+\varepsilon^{2}\eta^{2}+\left(\frac{\varepsilon}{\eta}\right)^{2}
 e^{-\frac{K}{\eta}(r_1-r_2)}\mathds{1}_{\{r_1\geq r_2\}}\right].
\end{align*}
By Proposition \ref{P:BoundSecondMalliavinDW2DW2X}, with $p=2$ we have for some constant $C >0$,
\begin{align*}
\mathbb{E}\left( \abs{D^{W^{2},W^{2}}_{r_1,r_2} X_t^{\varepsilon}}^{4}\right)
&\leq C \bigg[\varepsilon^{4}+\eta^{4}+\varepsilon^{2}\eta^{2}+
                \left(1+\left(\frac{\varepsilon}{\eta}\right)^{2}\right)e^{-\frac{K}{2\eta}(r_1
                 \vee r_2 - r_1 \wedge r_2)}\bigg].
\end{align*}
By inspecting these bounds for the second Malliavin derivatives, it
becomes clear that the worst term, i.e. the one that would yield the
largest bound in terms of $\varepsilon,\eta$ for
(\ref{Eq:BoundSecondDer1}) is the term that corresponds to the choice
$i=j=k=p=2$. That term is
\begin{align*}
  &\int_{[0,T]^4}\left(\mathbb{E}\left| D_{u,v}^{W^2,W^2}X_t^{\varepsilon}\right|^{4}\right)^{1/4}
          \left(\mathbb{E}\left|D_{u,w}^{W^2,W^2}X_t^{\varepsilon}\right|^{4}\right)^{1/4}
    \left(\mathbb{E}\left|D_{s,v}^{W^2,W^2}X_t^{\varepsilon}\right|^{4}\right)^{1/4}\\
  &\qquad\qquad\qquad\qquad\qquad\qquad\qquad\qquad\qquad\qquad\qquad\qquad
          \left(\mathbb{E}\left|D_{s,w}^{W^2,W^2}X_t^{\varepsilon}\right|^{4}\right)^{1/4}du ds dv dw.
\end{align*}
We use now Lemma \ref{L:BoundIntergalTermFinalSecondMalliavinTerm} with $k=\frac{K}{8\eta}$, which for $\eta$ small enough gives
\begin{align}
&\int_{[0,T]^4} e^{-\frac{K}{8\eta}(u\vee v - u \wedge v)}
                e^{-\frac{K}{8\eta}(u\vee w - u \wedge
                w)}e^{-\frac{K}{8\eta}(s\vee v - s \wedge
                v)}e^{-\frac{K}{8\eta}(s\vee w - s \wedge w)}dw ds dv
                du\nonumber\\
  &\qquad\qquad\qquad\qquad\qquad\qquad\qquad\qquad\qquad\qquad\qquad\qquad\qquad\leq C\left(\eta^{3}+\eta^{2}e^{-\frac{K}{4\eta}T}\right).
\end{align}
Thus, we get for $\varepsilon,\eta<1$
\begin{align}
&\int_{[0,T]^4}\left(\mathbb{E}\left| D_{u,v}^{W^2,W^2}X_t^{\varepsilon}\right|^{4}\right)^{1/4}
          \left(\mathbb{E}\left|D_{u,w}^{W^2,W^2}X_t^{\varepsilon}\right|^{4}\right)^{1/4}
                \left(\mathbb{E}\left|D_{s,v}^{W^2,W^2}X_t^{\varepsilon}\right|^{4}\right)^{1/4}\nonumber
\\
  &\qquad\qquad\qquad\qquad\qquad\qquad\qquad\qquad\qquad\qquad\qquad\qquad
          \left(\mathbb{E}\left|D_{s,w}^{W^2,W^2}X_t^{\varepsilon}\right|^{4}\right)^{1/4}du ds dv dw\nonumber\\
&\leq C\left[\varepsilon^{4}+\eta^{4}+\varepsilon^{2}\eta^{2}+ \left(1+\left(\frac{\varepsilon}{\eta}\right)^{2}\right)\left(\eta^{3}+\eta^{2}e^{-\frac{K}{4\eta}}\right)\right]\nonumber\\
&\leq C\left[\varepsilon^{4}+\eta^{3}+\varepsilon^{2}\eta+ (\varepsilon^{2}+\eta^{2})e^{-\frac{K}{4\eta}T} \right].
\end{align}
Finally, we have that for some $C<\infty$
\begin{align}
\sup_{t\in[0,T]}\mathbb{E}\left[   \norm{D^2 X_t^{\varepsilon}
    \otimes_1 D^2 X_t^{\varepsilon}}_{\frak{H}^{\otimes 2}}^2\right]
&  \leq C\left(\varepsilon^{4}+\eta^{3}+\varepsilon^{2}\eta+ (\varepsilon^{2}+\eta^{2})e^{-\frac{K}{4\eta}T} \right),\label{Eq:BoundSecondDer2}
\end{align}
hence concluding the proof of the lemma.
\end{proof}

\subsection{Calculation of the final bound and completion of the proof of Theorem \ref{T:MainQuantitative}}\label{SS:BoundsCompilation}
 Let us recall the bound
 \begin{align*}
d_W \left(\theta^{\varepsilon}_t, \mathscr{N} \left(\mu_t,\sigma_t^2
  \right)  \right) &\leq \mathbb{E}\left(\abs{
    \theta^{\varepsilon}_t} \right) \abs{1 - \frac{\sigma_t}{\sqrt{\on{Var}\left(
      \theta^{\varepsilon}_t \right)}}} + \frac{\sigma_t}{\sqrt{\on{Var}\left(
      \theta^{\varepsilon}_t \right)}} \abs{\mu_t - \mathbb{E}\left(
                     \theta^{\varepsilon}_t \right)}\\
  & + \frac{\sqrt{10}}{2 \varepsilon \on{Var}\left(
      \theta^{\varepsilon}_t \right)}\mathbb{E}\left[ \norm{D^2 X_t^{\varepsilon}
    \otimes_1 D^2 X_t^{\varepsilon}}_{\frak{H}^{\otimes 2}}^2 \right]^{\frac{1}{4}}\mathbb{E}\left[ \norm{D X_t^{\varepsilon}}_{\frak{H}}^4 \right]^{\frac{1}{4}}.
\end{align*}
By Lemma \ref{L:FirstDerivativeMomentBound} and \ref{L:SecondDerivativeMomentBound} we have for $\varepsilon,\eta<1$
\begin{align}
&\sup_{t\in[0,T]}\mathbb{E}\left[ \norm{D^2 X_t^{\varepsilon}
    \otimes_1 D^2 X_t^{\varepsilon}}_{\frak{H}^{\otimes 2}}^2 \right]^{\frac{1}{4}}\mathbb{E}\left[ \norm{D X_t^{\varepsilon}}_{\frak{H}}^4 \right]^{\frac{1}{4}}\leq\nonumber\\
    &\qquad\leq
C\left(\varepsilon+\eta^{3/4}+\varepsilon^{1/2}\eta^{1/4}+ (\sqrt{\varepsilon}+\sqrt{\eta})e^{-\frac{K}{16\eta}T}\right)\left(\sqrt{\varepsilon}+\sqrt{\eta}\right)\nonumber\\
    &\qquad\leq
C\left(\varepsilon^{3/2}+\varepsilon\eta^{1/4}+\varepsilon^{1/2}\eta^{3/4}+\eta^{5/4}+ (\varepsilon+\eta)e^{-\frac{K}{16\eta}T}\right).\nonumber
    \end{align}
The previous calculations give for the last term that
\begin{align}
&\frac{\sqrt{10}}{2 \varepsilon \on{Var}\left(
      \theta^{\varepsilon}_t \right)}\mathbb{E}\left[ \norm{D^2 X_t^{\varepsilon}
    \otimes_1 D^2 X_t^{\varepsilon}}_{\frak{H}^{\otimes 2}}^2 \right]^{\frac{1}{4}}\mathbb{E}\left[ \norm{D X_t^{\varepsilon}}_{\frak{H}}^4 \right]^{\frac{1}{4}}\nonumber\\
    &\qquad\qquad\qquad\qquad\leq\frac{C}{ \varepsilon \on{Var}\left(
      \theta^{\varepsilon}_t \right)}\left(\varepsilon^{3/2}+\varepsilon\eta^{1/4}+\varepsilon^{1/2}\eta^{3/4}+\eta^{5/4}+ (\varepsilon+\eta)e^{-\frac{K}{16\eta}T}\right)\nonumber\\
         &\qquad\qquad\qquad\qquad\leq\frac{C}{\on{Var}\left(
      \theta^{\varepsilon}_t \right)}\left(\varepsilon^{1/2}+\eta^{1/4}+\varepsilon^{-1/2}\eta^{3/4}+\varepsilon^{-1}\eta^{5/4}+ \left(1+\frac{\eta}{\varepsilon}\right)e^{-\frac{K}{16\eta}T}\right).\nonumber
\end{align}
In order to get the final bound, we need to study the terms $\abs{\mu_t - \mathbb{E}\left(
                     \theta^{\varepsilon}_t \right)}$ and $\abs{1 - \frac{\sigma_t}{\sqrt{\on{Var}\left(
      \theta^{\varepsilon}_t \right)}}}$. In order the obtain bounds
for these two terms, we appeal to \eqref{Eq:FirstMomentBound0} and
\eqref{Eq:DiffereceLim_PreLimSD}, respectively. Therefore, putting everything together, we finally obtain the error bound
\begin{align*}
d_W \left(\theta^{\varepsilon}_t, \mathscr{N} \left(\mu_t,\sigma_t^2
  \right)  \right) &\leq \mathbb{E}\left(\abs{
    \theta^{\varepsilon}_t} \right) \abs{1 - \frac{\sigma_t}{\sqrt{\on{Var}\left(
      \theta^{\varepsilon}_t \right)}}} + \frac{\sigma_t}{\sqrt{\on{Var}\left(
      \theta^{\varepsilon}_t \right)}} \abs{\mu_t - \mathbb{E}\left(
                     \theta^{\varepsilon}_t \right)}\\
  & + \frac{\sqrt{10}}{2 \varepsilon \on{Var}\left(
      \theta^{\varepsilon}_t \right)}\mathbb{E}\left[ \norm{D^2 X_t^{\varepsilon}
    \otimes_1 D^2 X_t^{\varepsilon}}_{\frak{H}^{\otimes 2}}^2 \right]^{\frac{1}{4}}\mathbb{E}\left[ \norm{D X_t^{\varepsilon}}_{\frak{H}}^4 \right]^{\frac{1}{4}}.\nonumber\\
&\leq  C \frac{\mathbb{E}\left(\abs{
    \theta^{\varepsilon}_t} \right)}{\sqrt{\on{Var}\left(
      \theta^{\varepsilon}_t \right)}} \sqrt{ \sqrt{\eta}+\sqrt{\varepsilon}+ \left(\frac{\eta}{\varepsilon}-\frac{1}{\gamma^{2}}\right)+ \frac{\eta^{2-\zeta}}{\varepsilon}+\varepsilon^{1-\zeta}}\nonumber\\
      &+C \frac{\sigma_t}{\sqrt{\on{Var}\left(
      \theta^{\varepsilon}_t \right)}}  \left( \frac{\eta^{1-\zeta}}{\sqrt{\varepsilon}}+\varepsilon^{1/2-\zeta}\right)\\
  & + \frac{C}{ \on{Var}\left(
      \theta^{\varepsilon}_t \right)}\left(\varepsilon^{1/2}+\eta^{1/2}+\varepsilon^{-1/2}\eta^{3/4}+\varepsilon^{-1}\eta^{5/4}+ \left(1+\frac{\eta}{\varepsilon}\right)e^{-\frac{K}{16\eta}T}\right).\nonumber
\end{align*}
Hence, there exists a constant $C$ such that for all $\zeta>0$ and for $\varepsilon,\eta$ sufficiently small one has
\begin{align}
\sup_{t\in(0,T]}d_W \left(\theta^{\varepsilon}_t, \mathscr{N} \left(\mu_t,\sigma_t^2
  \right)  \right) &\leq
                     C\left(\eta^{1/4}+\varepsilon^{1/4}+\left(\frac{\eta}{\varepsilon}-\frac{1}{\gamma^{2}}\right)^{1/2}+
                     \left( \frac{\eta}{\varepsilon} \right)^{1/2}\eta^{1/2-\zeta}+\varepsilon^{1/2-\zeta}\right)\nonumber\\
  & \quad+C\left(\left( \frac{\eta}{\varepsilon} \right)^{1/2}\eta^{1/4}+\left( \frac{\eta}{\varepsilon} \right)\eta^{1/4}+ \left(1+\frac{\eta}{\varepsilon}\right)e^{-\frac{K}{16\eta}T}\right).\nonumber
\end{align}
It is easy to see that under Assumption \ref{A:Assumption3}(i), the latter goes to zero as $\varepsilon,\eta$ at a rate that ultimately depends on Assumption \ref{A:Assumption3}(i). This completes the proof of Theorem \ref{T:MainQuantitative}.

\section{Preliminary results related to averaging and fluctuations}\label{S:PreliminaryCLTresults}

In this section, we present a quantitative result controlling the rate at which the error terms associated to the prelimit fluctuation process vanishes. This is a key preliminary result for the main result of this paper.

Let us re-express $\theta^{\varepsilon}_{t}=\frac{X^{\varepsilon}_{t}-\bar{X}_{t}}{\sqrt{\varepsilon}}$. We have by It\^{o}'s formula on the function $\Phi$ satisfying the Poisson equation (\ref{Eq:PoissonPDEFluctuations0})
 \begin{align}
 \theta^{\varepsilon}_{t}&=\frac{1}{\sqrt{\varepsilon}}\left(\int_{0}^{t}c(X^{\varepsilon}_s,Y^{\eta}_s)ds-\int_{0}^{t}\bar{c}(\bar{X}_{s})ds\right)+\int_{0}^{t}\sigma(X^{\varepsilon}_{s},Y^{\eta}_s)dW^{1}_{s}\nonumber\\
 &=\frac{1}{\sqrt{\varepsilon}}\left(\int_{0}^{t}c(X^{\varepsilon}_s,Y^{\eta}_s)ds-\int_{0}^{t}\bar{c}(X^{\varepsilon}_{s})ds\right)+\frac{1}{\sqrt{\varepsilon}}\int_{0}^{t}\left(\bar{c}(X^{\varepsilon}_s)ds-\bar{c}(\bar{X}_{s})\right)ds+\int_{0}^{t}\sigma(X^{\varepsilon}_{s},Y^{\eta}_s)dW^{1}_{s}\nonumber\\
 &=\frac{\eta}{\sqrt{\varepsilon}}\left(\Phi(X^{\varepsilon}_t,Y^{\eta}_t)-\Phi(X^{\varepsilon}_{0},Y^{\eta}_{0})\right)-\sqrt{\frac{\eta}{\varepsilon}}\int_0^{t}\partial_{y}\Phi(X^{\varepsilon}_s,Y^{\eta}_s)\tau(Y^{\eta}_{s})dW^{2}_{s}\nonumber\\
 &-\frac{\eta}{\sqrt{\varepsilon}}\int_{0}^{t}\left(c(X^{\varepsilon}_s,Y^{\eta}_s)\partial_{x}\Phi(X^{\varepsilon}_s,Y^{\eta}_s)+\frac{\varepsilon}{2}\left[\sigma^{2} \partial^{(2)}_{x}\Phi\right](X^{\varepsilon}_s,Y^{\eta}_s)\right)ds\nonumber\\
 &+\int_{0}^{t}\left(1-\eta\partial_{x}\Phi(X^{\varepsilon}_s,Y^{\eta}_s)\right)\sigma(X^{\varepsilon}_{s})dW^{1}_{s}+\frac{1}{\sqrt{\varepsilon}}\int_{0}^{t}\left(\bar{c}(X^{\varepsilon}_s)ds-\bar{c}(\bar{X}_{s})\right)ds\nonumber
 \end{align}

 Let us focus on the last term for a minute. Smoothness of $\bar{c}$ implies via Taylor's theorem that
 \[
 \bar{c}(x)=  \bar{c}(y) + \bar{c}^{\prime}(y) (x-y) + \Lambda[\bar{c}](x,y), \quad x,y \in \R^m,
 \]
 for some function $\Lambda[\bar{c}]$ such that $|x - y|^{-2}\Lambda[\bar{c}] (x,y)$ is locally bounded. Therefore, we obtain
 that $\theta^{\varepsilon}_{t}$ satisfies
 \begin{align}
 \theta^\varepsilon_{t} &=  \int_{0}^{t} \bar{c}^{\prime} ( \bar X_s) \theta^\varepsilon_{s} ds+ \int_0^t \frac{1}{\sqrt{\varepsilon}}\Lambda[\bar{c}]\parbar{ \bar X_s,X^{\varepsilon}_s}ds\nonumber\\
 &+\frac{\eta}{\sqrt{\varepsilon}}\left(\Phi(X^{\varepsilon}_t,Y^{\eta}_t)-\Phi(X^{\varepsilon}_{0},Y^{\eta}_{0})\right)\nonumber\\
 &-\frac{\eta}{\sqrt{\varepsilon}}\int_{0}^{t}\left(c(X^{\varepsilon}_s,Y^{\eta}_s)\partial_{x}\Phi(X^{\varepsilon}_s,Y^{\eta}_s)+\frac{\varepsilon}{2}\left[\sigma^{2} \partial^{(2)}_{y}\Phi\right](X^{\varepsilon}_s,Y^{\eta}_s)\right)ds\nonumber\\
 &+\int_{0}^{t}\left(1-\eta\partial_{x}\Phi(X^{\varepsilon}_s,Y^{\eta}_s)\right)\sigma(X^{\varepsilon}_{s})dW^{1}_{s}-\sqrt{\frac{\eta}{\varepsilon}}\int_0^{t}\partial_{y}\Phi(X^{\varepsilon}_s,Y^{\eta}_s)\tau(Y^{\eta}_{s})dW^{2}_{s}.\nonumber
   \end{align}

 By Duhamel's principle we can write
 \begin{align}
 \theta^\varepsilon_{t}&=\Psi_{x_{0}} (t) \int_{0}^{t} \left[\Psi_{x_{0}} (s)\right]^{-1} \sigma(X^{\varepsilon}_{s})  dW^1_{s}-\sqrt{\frac{\eta}{\varepsilon}}\Psi_{x_{0}} (t)\int_0^{t}\left[\Psi_{x_{0}} (s)\right]^{-1}\partial_{y}\Phi(X^{\varepsilon}_s,Y^{\eta}_s)\tau(Y^{\eta}_{s})dW^{2}_{s}\nonumber\\
 &\quad+\Psi_{x_{0}} (t)\int_0^t \left[\Psi_{x_{0}} (s)\right]^{-1}\frac{1}{\sqrt{\varepsilon}}\Lambda[\bar{c}]\parbar{ \bar X_s,X^{\varepsilon}_s}ds+R^{\varepsilon}(t;\Psi),\label{Eq:PrelimitFluctuationsSolution}
 \end{align}
 where
 \begin{align}
 R^{\varepsilon}(t;\Psi)&=\frac{\eta}{\sqrt{\varepsilon}}\Psi_{x_{0}} (t)\left(\Phi(X^{\varepsilon}_t,Y^{\eta}_t)-\Phi(X^{\varepsilon}_{0},Y^{\eta}_{0})\right)\nonumber\\
 &-\frac{\eta}{\sqrt{\varepsilon}}\Psi_{x_{0}} (t)\int_{0}^{t}\left[\Psi_{x_{0}} (s)\right]^{-1}\left(c(X^{\varepsilon}_s,Y^{\eta}_s)\partial_{x}\Phi(X^{\varepsilon}_s,Y^{\eta}_s)+\frac{\varepsilon}{2}\left[\sigma^{2} \partial^{(2)}_{y}\Phi\right](X^{\varepsilon}_s,Y^{\eta}_s)\right)ds\nonumber\\
 &-\eta \Psi_{x_{0}} (t)\int_{0}^{t}\left[\Psi_{x_{0}} (s)\right]^{-1}\left(\partial_{x}\Phi(X^{\varepsilon}_s,Y^{\eta}_s)\right)\sigma(X^{\varepsilon}_{s})dW^{1}_{s}.\nonumber
 \end{align}

Using Lemma 6 of \cite{BGS21}  to control the term
 $\mathbb{E}\left(\sup_{t\in[0,T]}|\Phi(X^{\varepsilon}_t,Y^{\eta}_t)|^{2}\right)$,  we get that for some constant $C<\infty$ and for $0<\zeta<1$
 \begin{align}
 \mathbb{E}\left(\sup_{t\in[0,T]}\left|R^{\varepsilon}(t;\Psi)\right|^{2}\right)&\leq C \frac{\eta^{2(1-\zeta)}}{\varepsilon}.\label{Eq:Rbound0}
 \end{align}

 Next,
 we need to control the term $Q^{\varepsilon}(t;\Psi)=\Psi_{x_{0}}
 (t)\int_0^t \left[\Psi_{x_{0}}
   (s)\right]^{-1}\frac{1}{\sqrt{\varepsilon}}\Lambda[\bar{c}]\parbar{ \bar
   X_s,X^{\varepsilon}_s}ds$. As it is shown in
 \cite{Spiliopoulos_CLT_Multiscale_2014}, one can prove that for any
 $T<\infty$, $\sup_{t\in[0,T]}\left|Q^{\varepsilon}(t;\Psi)\right|\rightarrow 0$ in probability as $\varepsilon\rightarrow 0$. However, in order to get quantitative rates of convergence, we need
 the rate at which this term goes to zero in $L^{2}\left( \Omega \right)$, which is the
 object of the following lemma.
 \begin{Lemma}
 Let $p \geq 1$ be a natural number and assume that
 \begin{equation*}
 \mathbb{E} \left( \sup_{t\in[0,T]}\abs{ \frac{X^{\varepsilon}_t -
     \bar{X}_t}{\sqrt{\varepsilon}}}^{8 \vee p}\right) < \infty.
 \end{equation*}
  Then, it holds that for any
 $\zeta >0$, we have
 \begin{align}
 \mathbb{E} \left( \sup_{t\in[0,T]}\left| Q^{\varepsilon}(t;\Psi) \right|^2 \right)&\leq C \left( \varepsilon^{1-\zeta} + \varepsilon^{1+\frac{p}{8}(1-\zeta)}  \right).\label{Eq:Qbound0}
 \end{align}
 \end{Lemma}
 \begin{proof}
 Recall that
 \begin{equation*}
 Q^{\varepsilon}(t;\Psi)=\Psi_{x_{0}} (t)\int_0^t \left[\Psi_{x_{0}} (s)\right]^{-1}\frac{1}{\sqrt{\varepsilon}}\Lambda[\bar{c}]\parbar{ \bar X_s,X^{\varepsilon}_s}ds.
 \end{equation*}
 The quadratic decay of $\Lambda[c]$ implies that
 \begin{align}
   \label{quaddecayoflambdaimplies0}
 \mathbb{E} \left( \sup_{t\in[0,T]}\left| Q^{\varepsilon}(t;\Psi) \right|^2 \right) \leq \mathbb{E}\left(\sup_{t\in[0,T]}\left|\Psi_{x_{0}} (t)\int_0^t
   \left[\Psi_{x_{0}}
   (s)\right]^{-1}\sqrt{\varepsilon}\abs{ \frac{X^{\varepsilon}_s - \bar{X}_s}{\sqrt{\varepsilon}}}^2ds\right|^2\right).
 \end{align}
 For any $1/4< \rho < 1/2$, let us introduce the stopping time $\tau^{\varepsilon}$ defined by
 \begin{equation*}
 \tau^{\varepsilon} = \inf \left\{ t > 0 \colon \abs{X^{\varepsilon}_t - \bar{X}_t}
 > \varepsilon^{\rho}\right\}.
 \end{equation*}
 The term on the right-hand side of
 \eqref{quaddecayoflambdaimplies0} can then be partitioned as
 \begin{align}
   \label{partitionedterm0}
 & \mathbb{E}\left(\sup_{t\in[0,T]}\left|\Psi_{x_{0}} (t)\int_0^t
   \left[\Psi_{x_{0}}
   (s)\right]^{-1}\sqrt{\varepsilon}\abs{ \frac{X^{\varepsilon}_s -
                  \bar{X}_s}{\sqrt{\varepsilon}}}^2ds\right|^2
                  \mathds{1}_{\left\{ \tau^{\varepsilon} > T
                 \right\}}\right) \nonumber \\
   &\qquad\quad  + \mathbb{E}\left(\sup_{t\in[0,T]}\left|\Psi_{x_{0}} (t)\int_0^t
   \left[\Psi_{x_{0}}
   (s)\right]^{-1}\sqrt{\varepsilon}\abs{ \frac{X^{\varepsilon}_s -
     \bar{X}_s}{\sqrt{\varepsilon}}}^2ds\right|^2 \mathds{1}_{\left\{
     \tau^{\varepsilon}<T \right\}}\right) = A^{\varepsilon}_1 + A^{\varepsilon}_2.
 \end{align}
 The condition $\tau^{\varepsilon}
 > T$ appearing in the term $A^{\varepsilon}_1$ implies that $\abs{X^{\varepsilon}_s -
     \bar{X}_s} \leq \varepsilon^{\rho}$, so that
   \begin{align}
       \label{firstpartofpartitionning0}
 A^{\varepsilon}_1 &\leq \sup_{t\in[0,T]}\left|\Psi_{x_{0}} (t)\int_0^t
   \left[\Psi_{x_{0}}
                     (s)\right]^{-1}ds\right|^2\varepsilon^{4\rho -
                     1}\nonumber \\
   &= \sup_{t\in[0,T]}\left|\Psi_{x_{0}} (t)\int_0^t
   \left[\Psi_{x_{0}}
   (s)\right]^{-1}ds\right|^2\varepsilon^{1-\zeta},
 \end{align}
 for $0<\zeta<1$. For the term $A^{\varepsilon}_2$ appearing in
 \eqref{partitionedterm0}, one can use Cauchy-Schwarz in order to get
 \begin{align*}
 A^{\varepsilon}_2 \leq \sqrt{\mathbb{E}\left( \left[ \sup_{t\in[0,T]}\left|\Psi_{x_{0}} (t)\int_0^t
   \left[\Psi_{x_{0}}
   (s)\right]^{-1}\sqrt{\varepsilon}\abs{ \frac{X^{\varepsilon}_s -
     \bar{X}_s}{\sqrt{\varepsilon}}}^2ds\right|^2
   \right]^2\right)}\sqrt{\mathbb{P}\left( \tau^{\varepsilon} < T
   \right)}.
 \end{align*}
 Note that by bounding $\abs{X^{\varepsilon}_s - \bar{X}_s}$ from above by its supremum over
 $t \in [0,T]$, one can write
 \begin{align}
   \label{estimateonothercspart0}
 & \sqrt{\mathbb{E}\left( \left[ \sup_{t\in[0,T]}\left|\Psi_{x_{0}} (t)\int_0^t
   \left[\Psi_{x_{0}}
   (s)\right]^{-1}\sqrt{\varepsilon}\abs{ \frac{X^{\varepsilon}_s -
     \bar{X}_s}{\sqrt{\varepsilon}}}^2ds\right|^2
   \right]^2\right)} \nonumber \\
   &\qquad\qquad\qquad\qquad\qquad\quad \leq \sup_{t\in[0,T]}\left|\Psi_{x_{0}} (t)\int_0^t
   \left[\Psi_{x_{0}}
   (s)\right]^{-1}ds\right|^2 \sqrt{\mathbb{E} \left( \sup_{t\in[0,T]}\abs{ \frac{X^{\varepsilon}_t -
     \bar{X}_t}{\sqrt{\varepsilon}}}^8\right)} \varepsilon.
 \end{align}
 Furthermore, using the definition of $\tau^{\varepsilon}$ and Markov's
 inequality yields
 \begin{align}
   \label{estimateonprobatau0}
 \mathbb{P}\left( \tau^{\varepsilon} < T
   \right) = \mathbb{P} \left( \sup_{t\in[0,T]}\abs{X^{\varepsilon}_t -
   \bar{X}_t} > \varepsilon^{\rho} \right)&\leq \frac{\mathbb{E} \left( \sup_{t\in[0,T]}\abs{X^{\varepsilon}_t -
                                            \bar{X}_t}^p
                                            \right)}{\varepsilon^{\rho
                                            p}} \nonumber \\
   &= \mathbb{E} \left( \sup_{t\in[0,T]}\abs{ \frac{X^{\varepsilon}_t -
     \bar{X}_t}{\sqrt{\varepsilon}}}^p\right)\varepsilon^{p \left( \frac{1}{2}-\rho \right)}.
 \end{align}
 Combining the estimates \eqref{estimateonothercspart0} and
 \eqref{estimateonprobatau0} yields
 \begin{align}
   \label{secondpartofpartitionning0}
 A^{\varepsilon}_2 & \leq \sup_{t\in[0,T]}\left|\Psi_{x_{0}} (t)\int_0^t
   \left[\Psi_{x_{0}}
   (s)\right]^{-1}ds\right|^2 \sqrt{\mathbb{E} \left( \sup_{t\in[0,T]}\abs{ \frac{X^{\varepsilon}_t -
   \bar{X}_t}{\sqrt{\varepsilon}}}^{8}\right)} \nonumber \\
   & \qquad\qquad\qquad\qquad\qquad\qquad\qquad\qquad\qquad\qquad \sqrt{\mathbb{E} \left( \sup_{t\in[0,T]}\abs{ \frac{X^{\varepsilon}_t -
     \bar{X}_t}{\sqrt{\varepsilon}}}^{p}\right)}
     \varepsilon^{1+\frac{p}{2} \left( \frac{1}{2}-\rho \right)}
     \nonumber \\
   & \leq \sup_{t\in[0,T]}\left|\Psi_{x_{0}} (t)\int_0^t
   \left[\Psi_{x_{0}}
   (s)\right]^{-1}ds\right|^2 \mathbb{E} \left( \sup_{t\in[0,T]}\abs{ \frac{X^{\varepsilon}_t -
   \bar{X}_t}{\sqrt{\varepsilon}}}^{8\vee p}\right)
     \varepsilon^{1+\frac{p}{2} \left( \frac{1}{2}-\rho \right)} \nonumber \\
   & \leq C \varepsilon^{1+\frac{p}{8}(1-\zeta)},
 \end{align}
 for $0 < \zeta < 1$. The combination of
 \eqref{firstpartofpartitionning0} and \eqref{secondpartofpartitionning0}
 finally yields \eqref{Eq:Qbound0}.
 \end{proof}
 \begin{remark}\label{R:BoundQterm}
   Note that when taking $p = 8$ in the above lemma (we are allowed to do so by the bounds of \cite{Spiliopoulos_CLT_Multiscale_2014}), one gets
   \begin{align*}
 \mathbb{E} \left( \sup_{t\in[0,T]}\left| Q^{\varepsilon}(t;\Psi) \right|^2
     \right)&\leq C \left( \varepsilon^{1-\zeta} +
              \varepsilon^{2-\zeta}  \right) \leq C \varepsilon^{1-\zeta},
   \end{align*}
   where the last inequality holds whenever $\varepsilon < 1$. 
 \end{remark}

\section{First order Malliavin derivatives}\label{S:FirstOrderMalliavinDer}

\noindent In this section, we derive the necessary bounds for the first order Malliavin derivatives. For $j=1,2$ we have that
\begin{align}
    \label{geneqforDX}
D_r^{W^j} X_t^{\varepsilon} &= \mathds{1}_{\left\{ j=1 \right\}}
  \sqrt{\varepsilon} \sigma\left(
      X_r^{\varepsilon},Y_r^{\eta} \right) + \int_r^t \left[
  \partial_1 c\left(X_s^{\varepsilon},Y_s^{\eta}
  \right)D_r^{W^j}X_s^{\varepsilon} + \partial_2
  c\left(X_s^{\varepsilon},Y_s^{\eta} \right)D_r^{W^j}Y_s^{\eta}
                              \right]ds \nonumber \\
  &\quad + \sqrt{\varepsilon}\int_r^t \left[ \partial_1 \sigma\left(X_s^{\varepsilon},Y_s^{\eta}
  \right)D_r^{W^j}X_s^{\varepsilon} + \partial_2
  \sigma\left(X_s^{\varepsilon},Y_s^{\eta} \right)D_r^{W^j}Y_s^{\eta} \right]dW_s^1,
\end{align}
and
\begin{align}
  \label{geneqforDY}
D_r^{W^j} Y_t^{\eta} &= \mathds{1}_{\left\{ j=2 \right\}}\frac{1}{\sqrt{\eta}}
   \tau\left(
      X_r^{\varepsilon},Y_r^{\eta} \right) + \frac{1}{\eta}\int_r^t \left[
  \partial_1 f\left(X_s^{\varepsilon},Y_s^{\eta}
  \right)D_r^{W^j}X_s^{\varepsilon} + \partial_2
  f\left(X_s^{\varepsilon},Y_s^{\eta} \right)D_r^{W^j}Y_s^{\eta}
                              \right]ds \nonumber \\
  &\quad + \frac{1}{\sqrt{\eta}}\int_r^t \left[ \partial_1 \tau\left(X_s^{\varepsilon},Y_s^{\eta}
  \right)D_r^{W^j}X_s^{\varepsilon} + \partial_2
  \tau\left(X_s^{\varepsilon},Y_s^{\eta} \right)D_r^{W^j}Y_s^{\eta} \right]dW_s^2,
\end{align}
Let us introduce some notation related to the term $D_r^{W^2}
Y_t^{\eta}$ that we will need to decompose in multiple locations in the sequel. Note that due to the affine structure of \eqref{geneqforDY},
one has
\begin{align*}
D_r^{W^2} Y_t^{\eta} &= Z_{r,2}(t) \Bigg[   \frac{1}{\sqrt{\eta}}
   \tau\left(
      X_r^{\varepsilon},Y_r^{\eta} \right) \\
  &\qquad\qquad\quad + \frac{1}{\eta}\int_r^t Z_{r,2}^{-1}(s) D_r^{W^2}X_s^{\varepsilon}\left[ \partial_1
  f\left(X_s^{\varepsilon},Y_s^{\eta} \right) -
                               \partial_1
  \tau\left(X_s^{\varepsilon},Y_s^{\eta} \right)\partial_2
    \tau\left(X_s^{\varepsilon},Y_s^{\eta} \right) \right]ds \\
  &\qquad\qquad\quad  +\frac{1}{\sqrt{\eta}}\int_r^t Z_{r,2}^{-1}(s) D_r^{W^2}X_s^{\varepsilon}\partial_1
    \tau\left(X_s^{\varepsilon},Y_s^{\eta} \right)dW_s^2 \Bigg],
\end{align*}
where
\begin{align}
  \label{expressionofZr2t}
  Z_{r,2}(t) =  e^{\frac{1}{\eta}\int_r^t \partial_2 f\left(X_s^{\varepsilon},Y_s^{\eta}
  \right)ds+ \frac{1}{\sqrt{\eta}} \int_r^t \partial_2 \tau\left(X_s^{\varepsilon},Y_s^{\eta}
  \right)dW_s^2 -\frac{1}{2\eta}\int_r^t \partial_2 \tau\left(X_s^{\varepsilon},Y_s^{\eta}
  \right)^2 ds}.
\end{align}
We can then decompose $D_r^{W^2} Y_t^{\eta}$ into
\begin{equation}
  \label{decompositionofDW2Yt}
D_r^{W^2}Y_t^{\eta}=Q^{\eta}_{r,1}(t)+Q^{\eta}_{r,2}(t),
\end{equation}
where
\begin{align}
  \label{defQ1}
Q^{\eta}_{r,1}(t)&= Z_{r,2}(t)\frac{1}{\sqrt{\eta}}
                   \tau\left(X_r^{\varepsilon},Y_r^{\eta} \right)
\end{align}
and
\begin{align}
    \label{defQ2}
 Q^{\eta}_{r,2}(t)&= Z_{r,2}(t) \Bigg[\frac{1}{\eta}\int_r^t Z_{r,2}^{-1}(s) D_r^{W^2}X_s^{\varepsilon}\left[ \partial_1
  f\left(X_s^{\varepsilon},Y_s^{\eta} \right) -
                               \partial_1
  \tau\left(X_s^{\varepsilon},Y_s^{\eta} \right)\partial_2
    \tau\left(X_s^{\varepsilon},Y_s^{\eta} \right) \right]ds\nonumber \\
  &\qquad\qquad\quad  +\frac{1}{\sqrt{\eta}}\int_r^t Z_{r,2}^{-1}(s) D_r^{W^2}X_s^{\varepsilon}\partial_1
    \tau\left(X_s^{\varepsilon},Y_s^{\eta} \right)dW_s^2 \Bigg].
\end{align}
Observe that $Q^{\eta}_{r,2}(t)$ can also be realized as the solution
to the affine stochastic differential equation
\begin{align}
  \label{Eq:SDE_Q2D2Y}
Q^{\eta}_{r,2}(t) &=  \frac{1}{\eta}\int_r^t \left[
  \partial_1 f\left(X_s^{\varepsilon},Y_s^{\eta}
  \right)D_r^{W^2}X_s^{\varepsilon} + \partial_2
  f\left(X_s^{\varepsilon},Y_s^{\eta} \right)Q^{\eta}_{r,2}(s)
                              \right]ds \nonumber\\
  &\quad + \frac{1}{\sqrt{\eta}}\int_r^t \left[ \partial_1 \tau\left(X_s^{\varepsilon},Y_s^{\eta}
  \right)D_r^{W^2}X_s^{\varepsilon} + \partial_2
  \tau\left(X_s^{\varepsilon},Y_s^{\eta} \right)Q^{\eta}_{r,2}(s) \right]dW_s^2.
\end{align}
We will, from now on, work with assumptions on the functions
$c,f,\sigma$ and $\tau$ that are applicable in wider setting than that of this
paper. We assume the following on the coefficients $c,f,\sigma$ and $\tau$.
\begin{assumption}
\label{A:Assumptiongeneral}
\begin{enumerate}[(i)]
\item The diffusion coefficient $\tau^{2}$ is uniformly nondegenerate.
\item  We assume that $c,\sigma\in C^{2,2}_{b}(\mathbb{R}\times\mathbb{R})$, i.e., they are uniformly bounded with bounded mixed derivatives up to order two.
\item{The function $f(x,y)$ has two bounded derivatives in $x$ and two derivatives in $y$. The function $\tau(x,y)$ has two bounded derivatives in $x$ and $y$. For both functions $f,\tau$, all partial derivatives are H\"{o}lder continuous, with exponent $\alpha$, with respect to $y$, uniformly in $x$. Given that the following combination appears many times later on, for a given natural number $p\geq 1$, we let $0<M<\infty$ be the constant such that
    \[
    \sup_{x,y}\left\{   |\partial_{1}f(x,y)|+ 2(2p-1)
    |\partial_{1}\tau(x,y)|^2 + |\partial_{2}\tau(x,y)|^2\right\}<M.
    \]}
\item{for a given natural number $p\geq 1$, there is a uniform constant $0<K<\infty$ such that
\[\sup_{x,y}\left\{   \left[ (2p-1)|\partial_1 f|+(2p-1)(2p-2)
  |\partial_1 \tau|^{2}+2p(2p-1)
  |\partial_2 \tau|^{2} +2p\partial_2 f\right]\left(x,y\right)\right\}\leq -K<0.    \]
}
\end{enumerate}
\end{assumption}
\begin{remark}
Note that Assumption \ref{A:Assumption1} is a particular case of
Assumption \ref{A:Assumptiongeneral} for $p=2$ in parts $(iii)$ and
$(iv)$. We derive all of our results in this increased generality as
we believe that they are of independent interest in terms of the
analysis of the Malliavin derivatives (and their moments) of systems of coupled SDEs.
\end{remark}
\begin{prop}
  \label{propDW1YleqDW1X}
Let $p \geq 1$ be a natural number. Then, it holds that for some
constant $C >0$,
\begin{equation*}
  \label{Eq:DW1Y_integralbound}
\mathbb{E}\left(\int_{r}^{t}\abs{D_r^{W^1}Y_s^{\eta}}^{2p}ds\right) \leq C\mathbb{E}\left(\int_{r}^{t}\abs{D_r^{W^1}X_s^{\varepsilon}}^{2p}ds\right).
\end{equation*}
\end{prop}
\begin{proof}
As $2p$ is an even number, we have $\abs{D_r^{W^1}Y_t^{\eta}}^{2p} =\left( D_r^{W^1}Y_s^{\eta} \right)^{2p}$ and It\^o's formula yields
\begin{align*}
&\mathbb{E}\left( \abs{D_r^{W^1}Y_t^{\eta}}^{2p} \right) =
                 \frac{2p}{\eta}\mathbb{E}\left(\int_r^t \left( D_r^{W^1}Y_s^{\eta} \right)^{2p-1}\left[ \partial_1 f\left(X_s^{\varepsilon},Y_s^{\eta}
  \right)D_r^{W^1}X_s^{\varepsilon} + \partial_2
  f\left(X_s^{\varepsilon},Y_s^{\eta} \right)D_r^{W^1}Y_s^{\eta}\right]ds\right)\\
  &\qquad+\frac{2p(2p-1)}{\eta}\mathbb{E}\left(\int_r^t \left( D_r^{W^1}Y_s^{\eta} \right)^{2p-2}\left[\partial_1 \tau\left(X_s^{\varepsilon},Y_s^{\eta}
  \right)D_r^{W^1}X_s^{\varepsilon} + \partial_2
  \tau\left(X_s^{\varepsilon},Y_s^{\eta}
    \right)D_r^{W^1}Y_s^{\eta}\right]^2ds\right).
\end{align*}
applying Young's inequality for products to
$\abs{D_r^{W^1}Y_s^{\eta}}^{2p-1}\abs{D_r^{W^1}X_s^{\varepsilon}}$ and
$\abs{D_r^{W^1}Y_s^{\eta}}^{2p-2}\abs{D_r^{W^1}X_s^{\varepsilon}}^2$
yields
\begin{align*}
&\mathbb{E}\left( \abs{D_r^{W^1}Y_t^{\eta}}^{2p} \right) \leq
                 \frac{1}{\eta}\mathbb{E}\Bigg(\int_r^t \Bigg[\abs{\partial_1
                 f\left(X_s^{\varepsilon},Y_s^{\eta} \right)} \\
  &\qquad\qquad\qquad\qquad\qquad\qquad\qquad + 2(2p-1)\abs{\partial_1 \tau\left(X_s^{\varepsilon},Y_s^{\eta}
  \right)}^2  \Bigg]\abs{D_r^{W^1}X_s^{\varepsilon}}^{2p}ds\Bigg)\\
  &\qquad+\frac{1}{\eta}\mathbb{E}\Bigg(\int_r^t\Bigg[ (2p-1)\abs{\partial_1
  f\left(X_s^{\varepsilon},Y_s^{\eta} \right)} + 2p\partial_2
    f\left(X_s^{\varepsilon},Y_s^{\eta} \right)\\
  &\qquad + (2p-1)(2p-2)
   \abs{\partial_1 \tau\left(X_s^{\varepsilon},Y_s^{\eta}
  \right)}^2 + 2p(2p-1)\abs{\partial_2 \tau\left(X_s^{\varepsilon},Y_s^{\eta}
  \right)}^2 \Bigg] \abs{D_r^{W^1}Y_s^{\eta}}^{2p}ds\Bigg).
\end{align*}
Assumption \ref{A:Assumptiongeneral} then implies that
\begin{align*}
\mathbb{E} \left( \abs{D_r^{W^1}Y_t^{\eta}}^{2p} \right)   &\leq
  \frac{M}{\eta}\mathbb{E}\int_r^t \abs{D_r^{W^1}X_s^{\varepsilon}}^{2p}ds  - \frac{K}{\eta}\mathbb{E}\int_r^t  \abs{D_r^{W^1}Y_s^{\eta}}^{2p}ds,
\end{align*}
resulting in
\begin{align*}
\mathbb{E}\left(\int_{r}^{t}\abs{D_r^{W^1}Y_s^{\eta}}^{2p}ds\right) \leq \frac{M}{K}\mathbb{E}\left(\int_{r}^{t}\abs{D_r^{W^1}X_s^{\varepsilon}}^{2p}ds\right),
\end{align*}
which concludes the proof.
\end{proof}
\begin{prop}
  \label{boundonEDW1Xs2p}
Let $p\geq 1$ be a natural number. Then, it holds that for some
constant $C >0$,
\begin{equation*}
\mathbb{E}\left( \sup_{r \leq s \leq t}
  \abs{D_r^{W^1}X_s^{\varepsilon}}^{2p} \right) \leq C \varepsilon^p.
\end{equation*}
\end{prop}
\begin{proof}
Using \eqref{geneqforDX}, we can write
\begin{align*}
\mathbb{E}\left( \sup_{r \leq s \leq t}
  \abs{D_r^{W^1}X_s^{\varepsilon}}^{2p} \right) &\leq C\Bigg(
  \varepsilon^p + (1+\varepsilon^p)\Bigg[\int_r^t \mathbb{E}\left( \sup_{r \leq u \leq s}
  \abs{D_r^{W^1}X_u^{\varepsilon}}^{2p} \right)ds \\
  &\qquad\qquad\qquad\qquad\qquad\qquad\qquad + \mathbb{E}\left(
  \int_r^t \abs{D_r^{W^1}Y_s^{\eta}}^{2p}ds\right)\Bigg]\Bigg).
\end{align*}
Applying Proposition \ref{propDW1YleqDW1X} and applying Gr\"onwall's
lemma concludes the proof.
\end{proof}
\begin{prop}
  \label{boundonEDW1Yt2p}
Let $p\geq 1$ be a natural number. Then, it holds that for some
constant $C >0$,
\begin{equation*}
\mathbb{E}\left(\int_r^t \abs{D_r^{W^1}Y_s^{\eta}}^{2p}ds \right) \leq C \varepsilon^p.
\end{equation*}
\end{prop}
\begin{proof}
The result is a consequence of Propositions \ref{propDW1YleqDW1X} and
\ref{boundonEDW1Xs2p}.
\end{proof}
\begin{prop}
  \label{boundonEDW1Yt2pb}
Let $p\geq 1$ be a natural number. Then, it holds that for some
constant $C >0$,
\begin{equation*}
\mathbb{E}\left( \abs{D_r^{W^1}Y_t^{\eta}}^{2p} \right) \leq C \varepsilon^p.
\end{equation*}
\end{prop}
\begin{proof}
Recall that per It\^o's formula, we have
\begin{align*}
&\mathbb{E}\left( \abs{D_r^{W^1}Y_t^{\eta}}^{2p} \right) =
                 \frac{2p}{\eta}\mathbb{E}\left(\int_r^t \left( D_r^{W^1}Y_s^{\eta} \right)^{2p-1}\left[ \partial_1 f\left(X_s^{\varepsilon},Y_s^{\eta}
  \right)D_r^{W^1}X_s^{\varepsilon} + \partial_2
  f\left(X_s^{\varepsilon},Y_s^{\eta} \right)D_r^{W^1}Y_s^{\eta}\right]ds\right)\\
  &\qquad+\frac{2p(2p-1)}{\eta}\mathbb{E}\left(\int_r^t \left( D_r^{W^1}Y_s^{\eta} \right)^{2p-2}\left[\partial_1 \tau\left(X_s^{\varepsilon},Y_s^{\eta}
  \right)D_r^{W^1}X_s^{\varepsilon} + \partial_2
  \tau\left(X_s^{\varepsilon},Y_s^{\eta}
    \right)D_r^{W^1}Y_s^{\eta}\right]^2ds\right).
\end{align*}
Differentiating this equality with respect to $t$, applying Young's inequality for products to
$\abs{D_r^{W^1}Y_s^{\eta}}^{2p-1}\abs{D_r^{W^1}X_s^{\varepsilon}}$ and
$\abs{D_r^{W^1}Y_s^{\eta}}^{2p-2}\abs{D_r^{W^1}X_s^{\varepsilon}}^2$ and using Assumption \ref{A:Assumptiongeneral} yields
\begin{align}
  \label{diffEDW1Yt}
\frac{d}{dt}\mathbb{E}\left( \abs{D_r^{W^1}Y_t^{\eta}}^{2p}
  \right) \leq \frac{M}{\eta}\mathbb{E}\left(
  \abs{D_r^{W^1}X_t^{\varepsilon}}^{2p} \right) - \frac{K}{\eta}\mathbb{E}\left( \abs{D_r^{W^1}Y_t^{\eta}}^{2p} \right).
\end{align}
Solving this differential inequality gives us
\begin{equation*}
\mathbb{E}\left( \abs{D_r^{W^1}Y_t^{\eta}}^{2p} \right) \leq
\frac{M}{K}\varepsilon^p \left( 1 - e^{-\frac{K}{\eta}(t-r)} \right)
\leq C \varepsilon^p
\end{equation*}
for some constant $c >0$.
\end{proof}
\begin{prop}
\label{boundonintDY2p}
Let $p\geq 1$ be a  natural number and $q \in \left\{ 1,2 \right\}$. Then, it holds that for some
constant $C >0$,
\begin{equation}
  \label{Eq:DY_intergalBound_2q}
\mathbb{E}\left(\left(\int_{r}^{t}|
    D_r^{W^2}Y_s^{\eta}|^{q}ds\right)^{\frac{2p}{q}}\right) \leq C \eta^{\frac{p(2-q)}{q}}+C\mathbb{E}\left(\int_{r}^{t}| D_r^{W^2}X_s^{\varepsilon}|^{2p}ds\right).
\end{equation}
\end{prop}
\begin{proof}
Using the decomposition \eqref{decompositionofDW2Yt}, one has
\begin{align}
  \label{decompositionofDW2Y}
\mathbb{E}\left(\left(\int_{r}^{t}\abs{D_r^{W^2}Y_s^{\eta}}^{q}ds\right)^{\frac{2p}{q}}\right)
&\leq C
\Bigg[\mathbb{E}\left(\left(\int_{r}^{t}\abs{Q^{\eta}_{r,1}(s)}^qds\right)^{\frac{2p}{q}}\right)
  \nonumber \\
  &\qquad\qquad\qquad\qquad + \mathbb{E}\left(\left(\int_{r}^{t}\abs{Q^{\eta}_{r,2}(s)}^qds\right)^{\frac{2p}{q}}\right) \Bigg].
\end{align}
We will estimate both these terms separately and we start with the
first one. Using the expression of $Q^{\eta}_{r,1}(t)$ given at
\eqref{defQ1} along with Assumption \ref{A:Assumptiongeneral}, we can write
\begin{align*}
\mathbb{E}\left(\left(\int_{r}^{t}\abs{Q^{\eta}_{r,1}(s)}^qds\right)^{\frac{2p}{q}}\right) &\leq  \frac{C}{\eta^{p}} \mathbb{E}\left( \left( \int_r^t Z_{r,2}(s)^q ds \right)^{\frac{2p}{q}} \right)
  \\
  &\leq  \frac{C}{\eta^{p}} \mathbb{E}\left(
    \int_r^t\cdots\int_r^t
    \prod_{i=1}^{2p/q}Z_{r,2}(s_i)^q ds_1\cdots ds_{2p/q} \right)\\
  & \leq \frac{C}{\eta^{p}}  \int_r^t\cdots\int_r^t \prod_{i=1}^{2p/q}\mathbb{E}\left(Z_{r,2}(s_i)^{2p}\right)^{q/2p} ds_1\cdots ds_{2p/q} .
\end{align*}
Using Lemma \ref{lemmaonExpofZr2tothe2pq}, we can further write
\begin{align}
  \label{Eq:Bound_D1Q}
\mathbb{E}\left(\left(\int_{r}^{t}\abs{Q^{\eta}_{r,1}(s)}^qds\right)^{\frac{2p}{q}}\right)  &\leq \frac{C}{\eta^{p}} \prod_{i=1}^{2p/q} \int_r^t
                     e^{-\frac{qK}{2p\eta}(s_i-r)}ds_i\nonumber
  \\
  &\leq C \eta^{\frac{p(2-q)}{q}}.
\end{align}
For the second term appearing on the right-hand side of
\eqref{decompositionofDW2Y}, recall that $Q^{\eta}_{r,2}(t)$ satisfies
the affine equation \eqref{Eq:SDE_Q2D2Y}. As $2p$ is an even number, we have $\abs{Q^{\eta}_{r,2}(t)}^{2p} =Q^{\eta}_{r,2}(t) ^{2p}$ and It\^o's formula yields
\begin{align*}
&\mathbb{E}\left( \abs{Q^{\eta}_{r,2}(t)}^{2p} \right) = \frac{2p}{\eta}\mathbb{E}\left(\int_r^tQ^{\eta}_{r,2}(s)^{2p-1}\left[ \partial_1 f\left(X_s^{\varepsilon},Y_s^{\eta}
  \right)D_r^{W^2}X_s^{\varepsilon} + \partial_2
  f\left(X_s^{\varepsilon},Y_s^{\eta} \right)Q^{\eta}_{r,2}(s)\right]ds\right)\\
  &\qquad+\frac{2p(2p-1)}{\eta}\mathbb{E}\left(\int_r^tQ^{\eta}_{r,2}(s)^{2p-2}\left[\partial_1 \tau\left(X_s^{\varepsilon},Y_s^{\eta}
  \right)D_r^{W^2}X_s^{\varepsilon} + \partial_2
  \tau\left(X_s^{\varepsilon},Y_s^{\eta}
    \right)Q^{\eta}_{r,2}(s)\right]^2ds\right).
\end{align*}
applying Young's inequality for products to
$\abs{Q^{\eta}_{r,2}(s)}^{2p-1}\abs{D_r^{W^2}X_s^{\varepsilon}}$ and
$\abs{Q^{\eta}_{r,2}(s)}^{2p-2}\abs{D_r^{W^2}X_s^{\varepsilon}}^2$
yields
\begin{align*}
&\mathbb{E}\left( \abs{Q^{\eta}_{r,2}(t)}^{2p} \right) \leq
                 \frac{1}{\eta}\mathbb{E}\Bigg(\int_r^t \Bigg[\abs{\partial_1
                 f\left(X_s^{\varepsilon},Y_s^{\eta} \right)} \\
  &\qquad\qquad\qquad\qquad\qquad\qquad\qquad + 2(2p-1)\abs{\partial_1 \tau\left(X_s^{\varepsilon},Y_s^{\eta}
  \right)}^2  \Bigg]\abs{D_r^{W^2}X_s^{\varepsilon}}^{2p}ds\Bigg)\\
  &\qquad+\frac{1}{\eta}\mathbb{E}\Bigg(\int_r^t\Bigg[ (2p-1)\abs{\partial_1
  f\left(X_s^{\varepsilon},Y_s^{\eta} \right)} + 2p\partial_2
    f\left(X_s^{\varepsilon},Y_s^{\eta} \right)\\
  &\qquad + (2p-1)(2p-2)
   \abs{\partial_1 \tau\left(X_s^{\varepsilon},Y_s^{\eta}
  \right)}^2 + 2p(2p-1)\abs{\partial_2 \tau\left(X_s^{\varepsilon},Y_s^{\eta}
  \right)}^2 \Bigg]Q^{\eta}_{r,2}(s)^{2p}ds\Bigg).
\end{align*}
Using Assumption \ref{A:Assumptiongeneral} then implies that
\begin{align}
  \label{ineqforEQ2}
\mathbb{E} \left( \abs{Q^{\eta}_{r,2}(t)}^{2p} \right)   &\leq
  \frac{M}{\eta}\mathbb{E} \left(\int_r^t \abs{D_r^{W^2}X_s^{\varepsilon}}^{2p}ds\right)  - \frac{K}{\eta}\mathbb{E} \left(\int_r^t  \abs{Q^{\eta}_{r,2}(s)}^{2p}ds\right),
\end{align}
which then immediately gives
\begin{align}
  \label{Eq:Bound_D2Q}
\mathbb{E} \left(\int_r^t  \abs{Q^{\eta}_{r,2}(s)}^{2p}ds\right)  \leq
  \frac{M}{K}\mathbb{E} \left(\int_r^t \abs{D_r^{W^2}X_s^{\varepsilon}}^{2p}ds\right).
\end{align}
Combining \eqref{Eq:Bound_D1Q} and \eqref{Eq:Bound_D2Q}, we finally
obtain that for all natural numbers $p \geq 1$ and $q \in \left\{ 1,2 \right\}$, \eqref{Eq:DY_intergalBound_2q} holds.
\end{proof}
\begin{prop}
  \label{boundonEDW2Xs2p}
Let $p\geq 1$ be a natural number. Then, it holds that for some
constant $C >0$,
\begin{equation*}
\mathbb{E}\left( \sup_{r \leq s \leq t}
  \abs{D_r^{W^2}X_s^{\varepsilon}}^{2p} \right) \leq C \left(\varepsilon^p
+ \eta^p\right).
\end{equation*}
\end{prop}
\begin{proof}
Using \eqref{geneqforDX} raised to the power $2p$, applying
expectation and using Doob's maximal inequality yields
\begin{align*}
\mathbb{E}\left( \sup_{r \leq s \leq t}
  \abs{D_r^{W^2}X_s^{\varepsilon}}^{2p} \right) &\leq C (1 + \varepsilon^p)\int_r^t \mathbb{E}\left( \sup_{r \leq u \leq s}
  \abs{D_r^{W^2}X_u^{\varepsilon}}^{2p} \right)ds \\
  &\quad + C\mathbb{E}\left(\left(
  \int_r^t \abs{D_r^{W^2}Y_s^{\eta}}ds\right)^{2p}\right) + C\varepsilon^p\mathbb{E}\left(\left(
  \int_r^t \abs{D_r^{W^2}Y_s^{\eta}}^2ds\right)^{p}\right).
\end{align*}
Using Proposition \ref{boundonintDY2p} twice yields
\begin{align*}
\mathbb{E}\left(   \sup_{r\leq s\leq t}\left|D_r^{W^2} X_t^{\varepsilon}\right|^{2p}\right) &\leq C (1+\varepsilon^{p}) \int_r^t \mathbb{E}\left(
\sup_{r\leq u\leq s}\left|D_r^{W^2} X_u^{\varepsilon}\right|^{2p}\right)ds+C\eta^{p}+C\varepsilon^{p},
\end{align*}
and applying
Gr\"onwall's lemma concludes the proof.
\end{proof}
\begin{prop}
\label{boundonintDY2pintermsofepsandeta}
Let $p\geq 1$ be a  natural number and $q \in \left\{ 1,2 \right\}$. Then, it holds that for some
constant $C >0$,
\begin{equation*}
\mathbb{E}\left(\left(\int_{r}^{t}\abs{D_r^{W^2}Y_s^{\eta}}^{q}ds\right)^{\frac{2p}{q}}\right) \leq C \left(\eta^{\frac{p(2-q)}{q}}+\varepsilon^{p}+\eta^{p}\right).
\end{equation*}
\end{prop}
\begin{proof}
The conclusion follows from applying Proposition \ref{boundonintDY2p}
together with Proposition \ref{boundonEDW2Xs2p}.
\end{proof}
\begin{prop}
\label{boundonDY2pintermsofepsandeta}
Let $p \geq 1$ be a natural number. Then, it holds that for some
constant $C >0$,
\begin{equation*}
\mathbb{E}\left(\abs{D_r^{W^2}Y_t^{\eta}}^{2p}\right) \leq C \left(
  \frac{1}{\eta^p}e^{-\frac{K}{\eta}(t-r)} +
  \varepsilon^p + \eta^p \right).
\end{equation*}
\end{prop}
\begin{proof}
  Recall from the proof of Proposition \ref{boundonintDY2p} that
  $D_r^{W^2}Y_t^{\eta}=Q^{\eta}_{r,1}(t)+Q^{\eta}_{r,2}(t)$, where
  $Q^{\eta}_{r,1}(t)$ and $Q^{\eta}_{r,2}(t)$ are defined in
  \eqref{defQ1} and \eqref{defQ2}, respectively. Hence,
\begin{equation*}
\mathbb{E}\left(\abs{D_r^{W^2}Y_t^{\eta}}^{2p}\right) \leq C \left(
  \mathbb{E}\left(\abs{Q^{\eta}_{r,1}(t)}^{2p}\right) + \mathbb{E}\left(\abs{Q^{\eta}_{r,2}(t)}^{2p}\right) \right).
\end{equation*}
As $Q^{\eta}_{r,1}(t)=
Z_{r,2}(t)\frac{1}{\sqrt{\eta}}\tau\left(X_r^{\varepsilon},Y_r^{\eta}
\right)$, Assumption \ref{A:Assumptiongeneral} and Lemma
\ref{lemmaonExpofZr2tothe2pq} imply
\begin{equation*}
\mathbb{E}\left(\abs{Q^{\eta}_{r,1}(t)}^{2p}\right) \leq  \frac{C}{\eta^p}e^{-\frac{K}{\eta}(t-r)}.
\end{equation*}
Applying the same procedure as in
\eqref{diffEDW1Yt} to the inequality \eqref{ineqforEQ2} yields
\begin{equation*}
\mathbb{E}\left(\abs{Q^{\eta}_{r,2}(t)}^{2p}\right) \leq C
(\varepsilon^p + \eta^p),
\end{equation*}
which concludes the proof.
\end{proof}

\section{Second order Malliavin derivatives}\label{S:SecondOrderMalliavinDer}
\noindent In this section, we derive the necessary bounds for the
second order Malliavin derivatives. We continue in this section to
work with Assumption \ref{A:Assumptiongeneral} as in the previous
section (for the same generality and independent interest reasons). We have, for $j_1,j_2 = 1,2$,
\begin{align}
  \label{SDEforDWj1Wj2Xt}
D_{r_1,r_2}^{W^{j_1},W^{j_2}}X_t^{\varepsilon} &= \sqrt{\varepsilon}\alpha_1
                                                 \left(X_{r_1}^{\varepsilon},X_{r_2}^{\varepsilon},Y_{r_1}^{\eta},Y_{r_2}^{\eta}
                                                 \right)\nonumber \\
  &\quad + \int_{r_1 \vee r_2}^t \left[ \beta_1[c]
    \left(X_s^{\varepsilon},Y_s^{\eta}
    \right)D_{r_1,r_2}^{W^{j_1},W^{j_2}}X_s^{\varepsilon} + b_1^{j_1,j_2}[c]\left(X_s^{\varepsilon},Y_s^{\eta}
    \right) \right]ds\nonumber \\
  &\quad + \sqrt{\varepsilon}\int_{r_1 \vee r_2}^t \left[ \beta_1[\sigma]
    \left(X_s^{\varepsilon},Y_s^{\eta}
    \right)D_{r_1,r_2}^{W^{j_1},W^{j_2}}X_s^{\varepsilon} + b_1^{j_1,j_2}[\sigma]\left(X_s^{\varepsilon},Y_s^{\eta}
    \right) \right]dW_s^1,
\end{align}
where we have defined
\begin{align*}
\alpha_1
\left(X_{r_1}^{\varepsilon},X_{r_2}^{\varepsilon},Y_{r_1}^{\eta},Y_{r_2}^{\eta}
                                                 \right) &=
  \mathds{1}_{\left\{ j_1 =1 \right\}} \left[
  \partial_1 \sigma \left( X_{r_1}^{\varepsilon},Y_{r_1}^{\eta}
  \right)D_{r_2}^{W^{j_2}}X_{r_1}^{\varepsilon} + \partial_2 \sigma \left( X_{r_1}^{\varepsilon},Y_{r_1}^{\eta}
  \right)D_{r_2}^{W^{j_2}}Y_{r_1}^{\eta} \right]\\
  &\quad + \mathds{1}_{\left\{ j_2 =1 \right\}} \left[
  \partial_1 \sigma \left( X_{r_2}^{\varepsilon},Y_{r_2}^{\eta}
  \right)D_{r_1}^{W^{j_1}}X_{r_2}^{\varepsilon} + \partial_2 \sigma \left( X_{r_2}^{\varepsilon},Y_{r_2}^{\eta}
  \right)D_{r_1}^{W^{j_1}}Y_{r_2}^{\eta} \right],
\end{align*}
\begin{align*}
\beta_1[c]
    \left(X_s^{\varepsilon},Y_s^{\eta}
    \right) = \partial_1 c \left( X_{s}^{\varepsilon},Y_{s}^{\eta}
  \right),
\end{align*}
\begin{align*}
b_1^{j_1,j_2}[c]\left(X_s^{\varepsilon},Y_s^{\eta}
    \right) &= \partial_1\partial_1c \left( X_{s}^{\varepsilon},Y_{s}^{\eta}
  \right)D_{r_1}^{W^{j_1}}X_{s}^{\varepsilon}D_{r_2}^{W^{j_2}}X_{s}^{\varepsilon}
  + \partial_1\partial_2c \left( X_{s}^{\varepsilon},Y_{s}^{\eta}
              \right)D_{r_1}^{W^{j_1}}X_{s}^{\varepsilon}D_{r_2}^{W^{j_2}}Y_{s}^{\eta}\\
  &\quad +  \partial_2\partial_1c \left( X_{s}^{\varepsilon},Y_{s}^{\eta}
  \right)D_{r_1}^{W^{j_1}}Y_{s}^{\eta}D_{r_2}^{W^{j_2}}X_{s}^{\varepsilon}
    +  \partial_2\partial_2c \left( X_{s}^{\varepsilon},Y_{s}^{\eta}
  \right)D_{r_1}^{W^{j_1}}Y_{s}^{\eta}D_{r_2}^{W^{j_2}}Y_{s}^{\eta}    \\
            &\quad + \partial_2 c \left( X_{s}^{\varepsilon},Y_{s}^{\eta}\right)D_{r_1,r_2}^{W^{j_1},W^{j_2}}Y_{s}^{\eta},
\end{align*}
with $\beta_1[\sigma]$, $b_1^{j_1,j_2}[\sigma]$ defined analogously. Similarly, we have, for $j_1,j_2 = 1,2$,
\begin{align}
  \label{equationsatisfiedbyDWWY}
D_{r_1,r_2}^{W^{j_1},W^{j_2}}Y_{t}^{\eta} &= \frac{1}{\sqrt{\eta}}\alpha_2
                                                 \left(X_{r_1}^{\varepsilon},X_{r_2}^{\varepsilon},Y_{r_1}^{\eta},Y_{r_2}^{\eta}
                                                 \right)\nonumber \\
  &\quad + \frac{1}{\eta}\int_{r_1 \vee r_2}^t \left[ \beta_2[f]
    \left(X_s^{\varepsilon},Y_s^{\eta}
    \right)D_{r_1,r_2}^{W^{j_1},W^{j_2}}Y_{s}^{\eta} + b_2^{j_1,j_2}[f]\left(X_s^{\varepsilon},Y_s^{\eta}
    \right) \right]ds\nonumber\\
  &\quad + \frac{1}{\sqrt{\eta}}\int_{r_1 \vee r_2}^t \left[ \beta_2[\tau]
    \left(X_s^{\varepsilon},Y_s^{\eta}
    \right)D_{r_1,r_2}^{W^{j_1},W^{j_2}}Y_{s}^{\eta} + b_2^{j_1,j_2}[\tau]\left(X_s^{\varepsilon},Y_s^{\eta}
    \right) \right]dW_s^2,
\end{align}
where
\begin{align*}
\alpha_2
                                                 \left(X_{r_1}^{\varepsilon},X_{r_2}^{\varepsilon},Y_{r_1}^{\eta},Y_{r_2}^{\eta}
                                                 \right) &=
  \mathds{1}_{\left\{ j_1 =2 \right\}} \left[
  \partial_1 \tau \left( X_{r_1}^{\varepsilon},Y_{r_1}^{\eta}
  \right)D_{r_2}^{W^{j_2}}X_{r_1}^{\varepsilon} + \partial_2 \tau \left( X_{r_1}^{\varepsilon},Y_{r_1}^{\eta}
  \right)D_{r_2}^{W^{j_2}}Y_{r_1}^{\eta} \right]\\
  &\quad + \mathds{1}_{\left\{ j_2 =2 \right\}} \left[
  \partial_1 \tau \left( X_{r_2}^{\varepsilon},Y_{r_2}^{\eta}
  \right)D_{r_1}^{W^{j_1}}X_{r_2}^{\varepsilon} + \partial_2 \tau \left( X_{r_2}^{\varepsilon},Y_{r_2}^{\eta}
  \right)D_{r_1}^{W^{j_1}}Y_{r_2}^{\eta} \right],
\end{align*}
\begin{align*}
\beta_2
    \left(X_s^{\varepsilon},Y_s^{\eta}
    \right) = \partial_2 f \left( X_{s}^{\varepsilon},Y_{s}^{\eta}
  \right),
\end{align*}
\begin{align*}
b_2^{j_1,j_2}[f]\left(X_s^{\varepsilon},Y_s^{\eta}
    \right) &= \Bigg[ \partial_1\partial_1f \left( X_{s}^{\varepsilon},Y_{s}^{\eta}
  \right)D_{r_1}^{W^{j_1}}X_{s}^{\varepsilon}D_{r_2}^{W^{j_2}}X_{s}^{\varepsilon}
  + \partial_1\partial_2f \left( X_{s}^{\varepsilon},Y_{s}^{\eta}
              \right)D_{r_1}^{W^{j_1}}X_{s}^{\varepsilon}D_{r_2}^{W^{j_2}}Y_{s}^{\eta}\\
  &\quad +  \partial_2\partial_1f \left( X_{s}^{\varepsilon},Y_{s}^{\eta}
  \right)D_{r_1}^{W^{j_1}}Y_{s}^{\eta}D_{r_2}^{W^{j_2}}X_{s}^{\varepsilon}
    +  \partial_2\partial_2f \left( X_{s}^{\varepsilon},Y_{s}^{\eta}
  \right)D_{r_1}^{W^{j_1}}Y_{s}^{\eta}D_{r_2}^{W^{j_2}}Y_{s}^{\eta}    \\
            &\quad + \partial_1 f \left( X_{s}^{\varepsilon},Y_{s}^{\eta}\right)D_{r_1,r_2}^{W^{j_1},W^{j_2}}X_{s}^{\varepsilon}\Bigg],
\end{align*}
with $\beta_2[\sigma]$ and $b_2^{j_1,j_2}[\sigma]$ defined analogously.

 \subsection{Bound for $\mathbb{E}\left( \abs{D^{W^{1},W^{1}}_{r_1,r_2} X_t^{\varepsilon}}^{2p}\right)$}

 \begin{prop}
  \label{boundonEDW1W1Xs2p}
Let $p\geq 1$ be a natural number. Then, it holds that for some
constant $C >0$,
\begin{equation*}
\mathbb{E}\left( \abs{D^{W^{1},W^{1}}_{r_1,r_2} X_t^{\varepsilon}}^{2p}
\right) \leq C \varepsilon^{2p}.
\end{equation*}
\end{prop}
\begin{proof}
Recall that $D^{W^{1},W^{1}}_{r_1,r_2} X_t^{\varepsilon}$ satisfies
equation \eqref{SDEforDWj1Wj2Xt}. Raising \eqref{SDEforDWj1Wj2Xt} to
the power $2p$, taking expectation, using Doob's maximal inequality
for martingales along with Assumption \ref{A:Assumptiongeneral} as
well as Propositions \ref{boundonEDW1Xs2p} and
\ref{boundonEDW1Yt2p} gives us that, for some constant $C>0$,
\begin{align*}
\mathbb{E}\left(\abs{D^{W^{1},W^{1}}_{r_1,r_2} X_t^{\varepsilon}}^{2p}\right)&\leq C \varepsilon^{p}\left(\mathbb{E}\left(\abs{D_{r_1\wedge r_2}^{W^{1}}X_{r_1\vee r_2}^{\varepsilon}}^{2p}\right) + \mathbb{E}\left(\abs{D_{r_1\wedge r_2}^{W^{1}}Y_{r_1\vee r_2}^{\eta}}^{2p}\right)\right)\\
&\quad
                                                                                                                             +C\bigg(\varepsilon^{2p}+\varepsilon^{3p}+(1+\varepsilon^{p})\bigg(\mathbb{E}\left(\int_{r_1\vee r_2}^{t}  \abs{D^{W^{1},W^{1}}_{r_1,r_2} Y_s^{\eta}}^{2p}ds\right)\\
  &\qquad\qquad\qquad\qquad\qquad\qquad\qquad +\mathbb{E}\left( \int_{r_1\vee r_2}^{t} \abs{D^{W^{1},W^{1}}_{r_1,r_2} X_s^{\varepsilon}}^{2p}ds\right)\bigg)
\bigg).
\end{align*}
Using Propositions \ref{boundonEDW1Xs2p}, \ref{boundonEDW1Yt2p}, \ref{boundonEDW1Yt2pb} and
\ref{proponEintDW1W1YslessEintDW1W1Xs} then yields, for $\varepsilon$ small enough,
\begin{align*}
\mathbb{E}\left(\abs{D^{W^{1},W^{1}}_{r_1,r_2} X_t^{\varepsilon}}^{2p}\right)&\leq C \left[\varepsilon^{2p}+\mathbb{E}\left( \int_{r_1\vee r_2}^{t} \abs{D^{W^{1},W^{1}}_{r_1,r_2} X_s^{\varepsilon}}^{2p}ds\right)\right].
\end{align*}
Applying Gr\"onwall's lemma finally shows that, for some finite constant $C>0$,
\begin{align*}
\mathbb{E}\left(\abs{D^{W^{1},W^{1}}_{r_1,r_2} X_t^{\varepsilon}}^{2p}\right)&\leq C \varepsilon^{2p},
\end{align*}
which is the desired conclusion.
\end{proof}

 \begin{prop}
   \label{proponEintDW1W1YslessEintDW1W1Xs}
Let $0 \leq r_1,r_2 \leq t$ and $p \geq 1$. Then, it holds that for some
constant $C >0$,
\begin{align*}
 \mathbb{E} \left( \int_{r_1\vee r_2}^{t}  \abs{D^{W^{1},W^{1}}_{r_1,r_2} Y_s^{\eta}}^{2p}ds \right)
&\leq C \left[\varepsilon^{2p}+\mathbb{E} \left( \int_{r_1\vee r_2}^{t} \abs{D^{W^{1},W^{1}}_{r_1,r_2} X_s^{\varepsilon}}^{2p}ds \right)\right]
\end{align*}
\end{prop}

\begin{proof}
As $2p$ is even, It\^{o}'s formula applied to
$\abs{D^{W^{1},W^{1}}_{r_1,r_2} Y_t^{\eta}}^{2p}$ and taking
expectation yields
\begin{align*}
\mathbb{E}\left(\abs{D^{W^{1},W^{1}}_{r_1,r_2} Y_t^{\eta}}^{2p}\right)
  &= \frac{1}{\eta}\mathbb{E}\bigg(\int_{r_1 \vee r_2}^t\bigg[ 2p \partial_2
    f\left( X_{s}^{\varepsilon},Y_{s}^{\eta}\right)
    \abs{D^{W^{1},W^{1}}_{r_1,r_2} Y_s^{\eta}}^{2p} \\
  &\qquad\qquad\qquad\qquad\qquad\qquad +
    2p b_2^{1,1}[f]\left( X_{s}^{\varepsilon},Y_{s}^{\eta}\right)
    \left(D^{W^{1},W^{1}}_{r_1,r_2} Y_s^{\eta}  \right)^{2p-1}\bigg]ds
\bigg)\\
  &\quad +\frac{1}{\eta}\mathbb{E}\bigg(\int_{r_1 \vee r_2}^t
    p(2p-1)\left(D^{W^{1},W^{1}}_{r_1,r_2}
    Y_s^{\eta}\right)^{2p-2}\\
  &\qquad\qquad\qquad\qquad\qquad\qquad +\bigg[ \partial_2 \tau\left(
    X_{s}^{\varepsilon},Y_{s}^{\eta}\right)D^{W^{1},W^{1}}_{r_1,r_2}
    Y_s^{\eta} + b_2^{1,1}[\tau] \bigg]^2ds \bigg)\\
  &\leq \frac{1}{\eta}\mathbb{E}\bigg(\int_{r_1 \vee r_2}^t\Big[ 2p \partial_2
    f\left( X_{s}^{\varepsilon},Y_{s}^{\eta}\right) + 2p(2p-1)\abs{\partial_2 \tau\left(
    X_{s}^{\varepsilon},Y_{s}^{\eta}\right)}^2\Big]\\
  & \qquad\qquad\qquad\qquad\qquad\qquad\qquad\qquad\qquad\qquad\qquad \abs{D^{W^{1},W^{1}}_{r_1,r_2} Y_s^{\eta}}^{2p}ds\bigg) \\
  &\quad + \frac{1}{\eta}\mathbb{E}\bigg(\int_{r_1 \vee r_2}^t
    2p \abs{b_2^{1,1}[f]\left( X_{s}^{\varepsilon},Y_{s}^{\eta}\right)}
    \abs{D^{W^{1},W^{1}}_{r_1,r_2} Y_s^{\eta}}^{2p-1}ds
\bigg)\\
  &\quad +\frac{1}{\eta}\mathbb{E}\bigg(\int_{r_1 \vee r_2}^t
    2p(2p-1)\abs{b_2^{1,1}[\tau]\left( X_{s}^{\varepsilon},Y_{s}^{\eta}\right)}^2
    \abs{D^{W^{1},W^{1}}_{r_1,r_2} Y_s^{\eta}}^{2p-2}ds
\bigg).
\end{align*}
Note that, using Young's inequality for products, we can write
\begin{align*}
\abs{D_{r_1,r_2}^{W^1,W^1}X_{s}^{\varepsilon}}\abs{D^{W^{1},W^{1}}_{r_1,r_2}
  Y_s^{\eta}}^{2p-1} \leq
  \frac{1}{2p}\abs{D_{r_1,r_2}^{W^1,W^1}X_{s}^{\varepsilon}}^{2p} + \frac{2p-1}{2p}\abs{D^{W^{1},W^{1}}_{r_1,r_2}
  Y_s^{\eta}}^{2p}
\end{align*}
and
\begin{align*}
\abs{D_{r_1,r_2}^{W^1,W^1}X_{s}^{\varepsilon}}^2\abs{D^{W^{1},W^{1}}_{r_1,r_2}
  Y_s^{\eta}}^{2p-2} \leq
  \frac{1}{p}\abs{D_{r_1,r_2}^{W^1,W^1}X_{s}^{\varepsilon}}^{2p} + \frac{2p-2}{2p}\abs{D^{W^{1},W^{1}}_{r_1,r_2}
  Y_s^{\eta}}^{2p}.
\end{align*}
As the terms $b_2^{1,1}[f]\left( X_{s}^{\varepsilon},Y_{s}^{\eta}\right)$
and $b_2^{1,1}[\tau]\left( X_{s}^{\varepsilon},Y_{s}^{\eta}\right)$
contain the terms $\partial_1 f \left(
  X_{s}^{\varepsilon},Y_{s}^{\eta}\right)D_{r_1,r_2}^{W^1,W^1}X_{s}^{\varepsilon}$
and $\partial_1 \tau \left(
  X_{s}^{\varepsilon},Y_{s}^{\eta}\right)D_{r_1,r_2}^{W^1,W^1}X_{s}^{\varepsilon}$,
respectively, we get, using the previous bounds,
\begin{align*}
\mathbb{E}\left(\abs{D^{W^{1},W^{1}}_{r_1,r_2} Y_t^{\eta}}^{2p}\right)
  &\leq \frac{1}{\eta}\mathbb{E}\bigg(\int_{r_1 \vee r_2}^t\Big[ 2p \partial_2
    f\left( X_{s}^{\varepsilon},Y_{s}^{\eta}\right) + 2p(2p-1)\abs{\partial_2 \tau\left(
    X_{s}^{\varepsilon},Y_{s}^{\eta}\right)}^2 \\
  & \quad + (2p-1)\abs{\partial_1 f\left(
    X_{s}^{\varepsilon},Y_{s}^{\eta}\right)} + 2(2p-1)(2p-2)\abs{\partial_1 \tau\left(
    X_{s}^{\varepsilon},Y_{s}^{\eta}\right)}^2\Big] \\
  & \qquad\qquad\qquad\qquad\qquad\qquad\qquad\qquad\qquad\qquad\qquad
    \abs{D^{W^{1},W^{1}}_{r_1,r_2} Y_s^{\eta}}^{2p}ds\bigg) \\
  &\quad + \frac{C}{\eta}\mathbb{E}\bigg(\int_{r_1 \vee r_2}^t
    \abs{D_{r_1}^{W^{j_1}}X_{s}^{\varepsilon}D_{r_2}^{W^{j_2}}X_{s}^{\varepsilon}}
    \abs{D^{W^{1},W^{1}}_{r_1,r_2} Y_s^{\eta}}^{2p-1}ds
    \bigg)\\
  &\quad + \frac{C}{\eta}\mathbb{E}\bigg(\int_{r_1 \vee r_2}^t
    \abs{
    D_{r_1}^{W^{j_1}}X_{s}^{\varepsilon}D_{r_2}^{W^{j_2}}Y_{s}^{\eta}}
    \abs{D^{W^{1},W^{1}}_{r_1,r_2} Y_s^{\eta}}^{2p-1}ds
    \bigg)\\
  &\quad + \frac{C}{\eta}\mathbb{E}\bigg(\int_{r_1 \vee r_2}^t
    \abs{D_{r_1}^{W^{j_1}}Y_{s}^{\eta}D_{r_2}^{W^{j_2}}X_{s}^{\varepsilon}}
    \abs{D^{W^{1},W^{1}}_{r_1,r_2} Y_s^{\eta}}^{2p-1}ds
    \bigg)\\
    &\quad + \frac{C}{\eta}\mathbb{E}\bigg(\int_{r_1 \vee r_2}^t
    \abs{D_{r_1}^{W^{j_1}}Y_{s}^{\eta}D_{r_2}^{W^{j_2}}Y_{s}^{\eta}}
    \abs{D^{W^{1},W^{1}}_{r_1,r_2} Y_s^{\eta}}^{2p-1}ds
      \bigg)\\
  &\quad + \frac{C}{\eta}\mathbb{E}\bigg(\int_{r_1 \vee r_2}^t
    \abs{D_{r_1}^{W^{j_1}}X_{s}^{\varepsilon}D_{r_2}^{W^{j_2}}X_{s}^{\varepsilon}}^2
    \abs{D^{W^{1},W^{1}}_{r_1,r_2} Y_s^{\eta}}^{2p-2}ds
    \bigg)\\
  &\quad + \frac{C}{\eta}\mathbb{E}\bigg(\int_{r_1 \vee r_2}^t
    \abs{
    D_{r_1}^{W^{j_1}}X_{s}^{\varepsilon}D_{r_2}^{W^{j_2}}Y_{s}^{\eta}}^2
    \abs{D^{W^{1},W^{1}}_{r_1,r_2} Y_s^{\eta}}^{2p-2}ds
    \bigg)\\
  &\quad + \frac{C}{\eta}\mathbb{E}\bigg(\int_{r_1 \vee r_2}^t
    \abs{D_{r_1}^{W^{j_1}}Y_{s}^{\eta}D_{r_2}^{W^{j_2}}X_{s}^{\varepsilon}}^2
    \abs{D^{W^{1},W^{1}}_{r_1,r_2} Y_s^{\eta}}^{2p-2}ds
    \bigg)\\
    &\quad + \frac{C}{\eta}\mathbb{E}\bigg(\int_{r_1 \vee r_2}^t
    \abs{D_{r_1}^{W^{j_1}}Y_{s}^{\eta}D_{r_2}^{W^{j_2}}Y_{s}^{\eta}}^2
    \abs{D^{W^{1},W^{1}}_{r_1,r_2} Y_s^{\eta}}^{2p-2}ds
      \bigg)\\
    &\quad + \frac{C}{\eta}\mathbb{E}\bigg(\int_{r_1 \vee r_2}^t
    \abs{D^{W^{1},W^{1}}_{r_1,r_2} X_s^{\varepsilon}}^{2p}ds
\bigg).
\end{align*}
Using Assumption \ref{A:Assumptiongeneral} in the first summand
finally gets us
\begin{align}
  \label{upperboundonD2W1W1YtwithDXDXD2Y}
\mathbb{E}\left(\abs{D^{W^{1},W^{1}}_{r_1,r_2} Y_t^{\eta}}^{2p}\right)
  &\leq -\frac{K}{\eta}\mathbb{E}\bigg(\int_{r_1 \vee r_2}^t\abs{D^{W^{1},W^{1}}_{r_1,r_2} Y_s^{\eta}}^{2p}ds\bigg) \nonumber\\
  &\quad + \frac{C}{\eta}\mathbb{E}\bigg(\int_{r_1 \vee r_2}^t
    \abs{D_{r_1}^{W^{j_1}}X_{s}^{\varepsilon}D_{r_2}^{W^{j_2}}X_{s}^{\varepsilon}}
    \abs{D^{W^{1},W^{1}}_{r_1,r_2} Y_s^{\eta}}^{2p-1}ds
    \bigg)\nonumber\\
  &\quad + \frac{C}{\eta}\mathbb{E}\bigg(\int_{r_1 \vee r_2}^t
    \abs{
    D_{r_1}^{W^{j_1}}X_{s}^{\varepsilon}D_{r_2}^{W^{j_2}}Y_{s}^{\eta}}
    \abs{D^{W^{1},W^{1}}_{r_1,r_2} Y_s^{\eta}}^{2p-1}ds
    \bigg)\nonumber\\
  &\quad + \frac{C}{\eta}\mathbb{E}\bigg(\int_{r_1 \vee r_2}^t
    \abs{D_{r_1}^{W^{j_1}}Y_{s}^{\eta}D_{r_2}^{W^{j_2}}X_{s}^{\varepsilon}}
    \abs{D^{W^{1},W^{1}}_{r_1,r_2} Y_s^{\eta}}^{2p-1}ds
    \bigg)\nonumber\\
    &\quad + \frac{C}{\eta}\mathbb{E}\bigg(\int_{r_1 \vee r_2}^t
    \abs{D_{r_1}^{W^{j_1}}Y_{s}^{\eta}D_{r_2}^{W^{j_2}}Y_{s}^{\eta}}
    \abs{D^{W^{1},W^{1}}_{r_1,r_2} Y_s^{\eta}}^{2p-1}ds
      \bigg)\nonumber\\
  &\quad + \frac{C}{\eta}\mathbb{E}\bigg(\int_{r_1 \vee r_2}^t
    \abs{D_{r_1}^{W^{j_1}}X_{s}^{\varepsilon}D_{r_2}^{W^{j_2}}X_{s}^{\varepsilon}}^2
    \abs{D^{W^{1},W^{1}}_{r_1,r_2} Y_s^{\eta}}^{2p-2}ds
    \bigg)\nonumber\\
  &\quad + \frac{C}{\eta}\mathbb{E}\bigg(\int_{r_1 \vee r_2}^t
    \abs{
    D_{r_1}^{W^{j_1}}X_{s}^{\varepsilon}D_{r_2}^{W^{j_2}}Y_{s}^{\eta}}^2
    \abs{D^{W^{1},W^{1}}_{r_1,r_2} Y_s^{\eta}}^{2p-2}ds
    \bigg)\nonumber\\
  &\quad + \frac{C}{\eta}\mathbb{E}\bigg(\int_{r_1 \vee r_2}^t
    \abs{D_{r_1}^{W^{j_1}}Y_{s}^{\eta}D_{r_2}^{W^{j_2}}X_{s}^{\varepsilon}}^2
    \abs{D^{W^{1},W^{1}}_{r_1,r_2} Y_s^{\eta}}^{2p-2}ds
    \bigg)\nonumber\\
    &\quad + \frac{C}{\eta}\mathbb{E}\bigg(\int_{r_1 \vee r_2}^t
    \abs{D_{r_1}^{W^{j_1}}Y_{s}^{\eta}D_{r_2}^{W^{j_2}}Y_{s}^{\eta}}^2
    \abs{D^{W^{1},W^{1}}_{r_1,r_2} Y_s^{\eta}}^{2p-2}ds
      \bigg)\nonumber\\
    &\quad + \frac{C}{\eta}\mathbb{E}\bigg(\int_{r_1 \vee r_2}^t
    \abs{D^{W^{1},W^{1}}_{r_1,r_2} X_s^{\varepsilon}}^{2p}ds
\bigg).
\end{align}
We now need to estimate each term in the expression above, aside from
the first and last one. Starting with the first of the remaining
terms, we have
 \begin{align*}
 &\mathbb{E}\left(\int_{r_1\vee r_2}^{t} \left| D^{W^{1}}_{r_1} X_s^{\varepsilon} D^{W^{1}}_{r_2} X_s^{\varepsilon} \right| \abs{D^{W^{1},W^{1}}_{r_1,r_2} Y_s^{\eta}}^{2p-1}ds\right)\\
 &\leq
 \mathbb{E}\left(\int_{r_1\vee r_2}^{t}  \left(D^{W^{1}}_{r_1} X_s^{\varepsilon} D^{W^{1}}_{r_2} X_s^{\varepsilon}\right)^{2p}ds\right)^{1/2p} \mathbb{E}\left(\int_{r_1\vee r_2}^{t}\abs{D^{W^{1},W^{1}}_{r_1,r_2} Y_s^{\eta}}^{2p}ds\right)^{\frac{2p-1}{2p}}\\
 &\leq
\mathbb{E}\left(\int_{r_1\vee r_2}^{t}  \left(D^{W^{1}}_{r_1} X_s^{\varepsilon}\right)^{4p}ds\right)^{1/4p} \mathbb{E}\left(\int_{r_1\vee r_2}^{t}\left(  D^{W^{1}}_{r_2} X_s^{\varepsilon}\right)^{4p}ds \right)^{1/4p} \mathbb{E}\left(\int_{r_1\vee r_2}^{t}\abs{D^{W^{1},W^{1}}_{r_1,r_2} Y_s^{\eta}}^{2p}ds\right)^{\frac{2p-1}{2p}}\\
 &\leq
 C \varepsilon \mathbb{E}\left(\int_{r_1\vee r_2}^{t}\abs{D^{W^{1},W^{1}}_{r_1,r_2} Y_s^{\eta}}^{2p}ds\right)^{\frac{2p-1}{2p}},
 \end{align*}
 where Proposition \ref{boundonEDW1Xs2p} was used for the latter
 bound. In a similar manner, using Propositions
 \ref{boundonEDW1Xs2p} and \ref{boundonEDW1Yt2p}, we obtain
 \begin{align*}
 &\mathbb{E}\left(\int_{r_1\vee r_2}^{t}  \left|D^{W^{1}}_{r_1} Y_s^{\eta} D^{W^{1}}_{r_2} X_s^{\varepsilon} \right| \abs{D^{W^{1},W^{1}}_{r_1,r_2} Y_s^{\eta}}^{2p-1}ds\right)\leq C
 \varepsilon \mathbb{E}\left(\int_{r_1\vee r_2}^{t}\abs{D^{W^{1},W^{1}}_{r_1,r_2} Y_s^{\eta}}^{2p}ds\right)^{\frac{2p-1}{2p}}\\
  &\mathbb{E}\left(\int_{r_1\vee r_2}^{t}  \left|D^{W^{1}}_{r_1} X_s^{\varepsilon} D^{W^{1}}_{r_2} Y_s^{\eta} \right| \abs{D^{W^{1},W^{1}}_{r_1,r_2} Y_s^{\eta}}^{2p-1}ds\right)\leq C
 \varepsilon \mathbb{E}\left(\int_{r_1\vee r_2}^{t}\abs{D^{W^{1},W^{1}}_{r_1,r_2} Y_s^{\eta}}^{2p}ds\right)^{\frac{2p-1}{2p}}\\
 &\mathbb{E}\left(\int_{r_1\vee r_2}^{t} \left| D^{W^{1}}_{r_1} Y_s^{\eta} D^{W^{1}}_{r_2} Y_s^{\eta} \right| \abs{D^{W^{1},W^{1}}_{r_1,r_2} Y_s^{\eta}}^{2p-1}ds\right)\leq C
 \varepsilon \mathbb{E}\left(\int_{r_1\vee r_2}^{t}\abs{D^{W^{1},W^{1}}_{r_1,r_2} Y_s^{\eta}}^{2p}ds\right)^{\frac{2p-1}{2p}}\\
 &\mathbb{E}\left(\int_{r_1\vee r_2}^{t}  \abs{D^{W^{1}}_{r_1} Y_s^{\eta}D^{W^{1}}_{r_2} X_s^{\varepsilon}}^{2} \abs{D^{W^{1},W^{1}}_{r_1,r_2} Y_s^{\eta}}^{2p-2}ds\right)\leq C
 \varepsilon^{2} \mathbb{E}\left(\int_{r_1\vee r_2}^{t}\abs{D^{W^{1},W^{1}}_{r_1,r_2} Y_s^{\eta}}^{2p}ds\right)^{\frac{p-1}{p}}\\
 &\mathbb{E}\left(\int_{r_1\vee r_2}^{t}  \abs{D^{W^{1}}_{r_1} X_s^{\varepsilon}D^{W^{1}}_{r_2} Y_s^{\eta}}^{2} \abs{D^{W^{1},W^{1}}_{r_1,r_2} Y_s^{\eta}}^{2p-2}ds\right)\leq C
 \varepsilon^{2} \mathbb{E}\left(\int_{r_1\vee r_2}^{t}\abs{D^{W^{1},W^{1}}_{r_1,r_2} Y_s^{\eta}}^{2p}ds\right)^{\frac{p-1}{p}}\\
 &\mathbb{E}\left(\int_{r_1\vee r_2}^{t}  \abs{D^{W^{1}}_{r_1} Y_s^{\eta}D^{W^{1}}_{r_2} Y_s^{\eta}}^{2} \abs{D^{W^{1},W^{1}}_{r_1,r_2} Y_s^{\eta}}^{2p-2}ds\right)\leq C
 \varepsilon^{2} \mathbb{E}\left(\int_{r_1\vee r_2}^{t}\abs{D^{W^{1},W^{1}}_{r_1,r_2} Y_s^{\eta}}^{2p}ds\right)^{\frac{p-1}{p}}.
 \end{align*}
Hence, using the above estimates in
\eqref{upperboundonD2W1W1YtwithDXDXD2Y} yields
\begin{align*}
\mathbb{E}\left(\abs{D^{W^{1},W^{1}}_{r_1,r_2} Y_t^{\eta}}^{2p}\right)
&\leq -\frac{K}{\eta}\mathbb{E}\left(\int_{r_1\vee r_2}^{t} \abs{D^{W^{1},W^{1}}_{r_1,r_2} Y_s^{\eta}}^{2p}ds\right)\\
 &\quad+\frac{C}{\eta} \varepsilon \mathbb{E}\left(\int_{r_1\vee r_2}^{t}  \abs{D^{W^{1},W^{1}}_{r_1,r_2} Y_s^{\eta}}^{2p}ds\right)^{\frac{2p-1}{2p}}\\
 &\quad+\frac{C}{\eta} \varepsilon^{2}\mathbb{E}\left(\int_{r_1\vee r_2}^{t}  \abs{D^{W^{1},W^{1}}_{r_1,r_2} Y_s^{\eta}}^{2p}ds\right)^{\frac{p-1}{p}}\\
 &\quad +\frac{C}{\eta} \mathbb{E}\left(\int_{r_1\vee r_2}^{t} \abs{D^{W^{1},W^{1}}_{r_1,r_2} X_s^{\varepsilon}}^{2p}ds\right).
\end{align*}
In order to solve this inequality, let us set
$A_t=\mathbb{E}\left(\int_{r_1\vee r_2}^{t}
  \abs{D^{W^{1},W^{1}}_{r_1,r_2} Y_s^{\eta}}^{2p}ds\right)$ and
$B_t=\mathbb{E}\left(\int_{r_1\vee r_2}^{t}
  \abs{D^{W^{1},W^{1}}_{r_1,r_2}
    X_s^{\varepsilon}}^{2p}ds\right)$. Using Young's inequality for
products with H\"older conjugates $p_i,q_i > 1$, $i=1,2$, one can
write, for some constants $C<\infty$ and $D<\infty$ to be chosen later,
\begin{align*}
A_t&\leq C\left(\varepsilon A_t^{\frac{2p-1}{2p}}+\varepsilon^{2}A_t^{\frac{p-1}{p}}+B_t\right)\\
&\leq \frac{1}{p_1} (C\varepsilon)^{p_1}+\frac{1}{q_1}A_t^{\frac{2p-1}{2p}q_1}+\frac{1}{p_2}(C\varepsilon^{2}D^{-1})^{p_2}+\frac{1}{q_2}D^{q_2}A_t^{\frac{p-1}{p}q_2}+C B_t\\
&= \left(\frac{1}{p_1} (C\varepsilon)^{p_1}+\frac{1}{p_2}(C\varepsilon^{2}D^{-1})^{p_2}\right)+\left(\frac{1}{q_1}A_t^{\frac{2p-1}{2p}q_1}+\frac{1}{q_2}D^{q_2}A_t^{\frac{p-1}{p}q_2}\right)+C B_t
\end{align*}
Now, we can choose the H\"older conjugates to be $(p_1,q_1)=
\left(2p,\frac{2p}{2p-1}\right)$ and
$(p_2,q_2)=\left(p,\frac{p}{p-1}\right)$, which yields
\begin{align*}
A_t
&\leq \left(\frac{1}{2p} C^{2p}+\frac{1}{p}(CD^{-1})^{p}\right)\varepsilon^{2p}+\left(\frac{2p-1}{2p}+\frac{p-1}{p}D^{\frac{p}{p-1}}\right)A_t+C B_t
\end{align*}
At this point, let us choose $D^{\frac{p}{p-1}}<\frac{1}{2(p-1)}$,
which results in
$\left(\frac{2p-1}{2p}+\frac{p-1}{p}D^{\frac{p}{p-1}}\right)<1$. Note
that $\left(\frac{2p-1}{2p}+\frac{p-1}{p}D^{\frac{p}{p-1}}\right)$
cannot become arbitrarily small, but the fact that for any finite $p
\geq 1$,
there is a $D>0$ so that
$\left(\frac{2p-1}{2p}+\frac{p-1}{p}D^{\frac{p}{p-1}}\right)<1$ is
enough for our purposes. We then obtain that, for some constant
$C>0$ that might change from line to line,
\begin{align*}
A_t
&\leq C\left(\varepsilon^{2p} +B_t\right),
\end{align*}
or alternatively
\begin{align*}
 \mathbb{E}\left(\int_{r_1\vee r_2}^{t}
  \abs{D^{W^{1},W^{1}}_{r_1,r_2} Y_s^{\eta}}^{2p}ds\right)
&\leq C \left[\varepsilon^{2p}+\mathbb{E} \left( \int_{r_1\vee r_2}^{t} \abs{D^{W^{1},W^{1}}_{r_1,r_2} X_s^{\varepsilon}}^{2p}ds \right)\right],
\end{align*}
which concludes the proof.
\end{proof}

 \subsection{Bound for $\mathbb{E}\left( \abs{D^{W^{1},W^{2}}_{r_1,r_2} X_t^{\varepsilon}}^{2p}\right)$}

\begin{prop}\label{P:BoundSecondMalliavinDW1DW2X}
Let $p\geq 1$ be a  natural number. Then, it holds that for some
constant $C >0$,
\begin{align*}
\mathbb{E}\left( \abs{D^{W^{1},W^{2}}_{r_1,r_2} X_t^{\varepsilon}}^{2p}\right)&\leq C \left[\varepsilon^{2p}+\varepsilon^{p}\eta^{p}+\left(\frac{\varepsilon}{\eta}\right)^pe^{-\frac{K}{\eta}(r_1-r_2)}\mathds{1}_{\{r_1\geq r_2\}}\right].
\end{align*}
\end{prop}
\begin{proof}
Using \eqref{SDEforDWj1Wj2Xt}, the Burkholder-Davis-Gundy inequality, and Assumption \ref{A:Assumptiongeneral}, we can write
  \begin{align*}
&\mathbb{E}\left( \abs{D^{W^{1},W^{2}}_{r_1,r_2} X_t^{\varepsilon}}^{2p}\right)\leq C \varepsilon^{p}\left[\mathbb{E}\left(\abs{D_{ r_2}^{W^{2}}X_{r_1}^{\varepsilon}}^{2p}\right) + \mathbb{E}\left(\abs{D_{r_2}^{W^{2}}Y_{r_1}^{\eta}}^{2p}\right)\right]\mathds{1}_{\{r_1\geq r_2\}}\nonumber\\
&\qquad\qquad  +C\left(1+\varepsilon^{p}\right)\mathbb{E}\left( \int_{r_1\vee r_2}^{t}  \abs{D^{W^{1}}_{r_1} X_s^{\varepsilon}D^{W^{2}}_{r_2} X_s^{\varepsilon}}^{2p}ds\right)\nonumber\\
&\qquad\qquad  +C\mathbb{E}\left(\left|\int_{r_1\vee r_2}^{t}  \abs{D^{W^{1}}_{r_1} X_s^{\varepsilon}D^{W^{2}}_{r_2} Y_s^{\eta}}ds\right|^{2p}\right) +C\varepsilon^{p}\mathbb{E}\left(\left|\int_{r_1\vee r_2}^{t}  \abs{D^{W^{1}}_{r_1} X_s^{\varepsilon}D^{W^{2}}_{r_2} Y_s^{\eta}}^{2}ds\right|^{p}
\right)\nonumber\\
&\qquad\qquad +C\mathbb{E}\left(\left|\int_{r_1\vee r_2}^{t}  \abs{D^{W^{1}}_{r_1} Y_s^{\eta}D^{W^{2}}_{r_2} X_s^{\varepsilon}}ds\right|^{2p}\right) +C\varepsilon^{p}\mathbb{E}\left(\left|\int_{r_1\vee r_2}^{t}  \abs{D^{W^{1}}_{r_1} Y_s^{\eta}D^{W^{2}}_{r_2}  X_s^{\varepsilon}}^{2}ds\right|^{p}\right)\nonumber\\
&\qquad\qquad  +C\mathbb{E}\left(\left|\int_{r_1\vee r_2}^{t}  \abs{D^{W^{1}}_{r_1} Y_s^{\eta}D^{W^{2}}_{r_2} Y_s^{\eta}}ds\right|^{2p}\right) +C\varepsilon^{p}\mathbb{E}\left(\left|\int_{r_1\vee r_2}^{t}  \abs{D^{W^{1}}_{r_1} Y_s^{\eta}D^{W^{2}}_{r_2}  Y_s^{\eta}}^{2}ds\right|^{p}\right)\nonumber\\
&\qquad\qquad  +C\mathbb{E}\left(\left|\int_{r_1\vee r_2}^{t}  \abs{D^{W^{1},W^{2}}_{r_1,r_2} Y_s^{\eta}}ds\right|^{2p}\right)+C\varepsilon^{p}\mathbb{E}\left(\left|\int_{r_1\vee r_2}^{t}  \abs{D^{W^{1},W^{2}}_{r_1,r_2} Y_s^{\eta}}^{2}ds\right|^{p}\right)\nonumber\\
&\qquad\qquad  +C\left(1+\varepsilon^{p}\right)\mathbb{E}\left(\int_{r_1\vee r_2}^{t} \abs{D^{W^{1},W^{2}}_{r_1,r_2} X_s^{\varepsilon}}^{2p}ds\right).
  \end{align*}
Let us bound these terms individually. Propositions \ref{boundonEDW2Xs2p} and \ref{boundonDY2pintermsofepsandeta}
ensure that
\begin{align*}
\mathbb{E}\left(\abs{D_{ r_2}^{W^{2}}X_{r_1}^{\varepsilon}}^{2p}\right) + \mathbb{E}\left(\abs{D_{r_2}^{W^{2}}Y_{r_1}^{\eta}}^{2p}\right)&\leq C\left(\varepsilon^{p}+\eta^{p}+\frac{1}{\eta^p}e^{-\frac{K}{\eta}(r_1-r_2)}\right).
\end{align*}
Next, applying H\"older's inequality together with Propositions \ref{boundonEDW1Xs2p} and \ref{boundonEDW2Xs2p} yields
\begin{align*}
\mathbb{E}\left( \int_{r_1\vee r_2}^{t}  \abs{D^{W^{1}}_{r_1} X_s^{\varepsilon}D^{W^{2}}_{r_2} X_s^{\varepsilon}}^{2p}ds\right)
&\leq C\varepsilon^{p}\left( \varepsilon^{p}+\eta^{p}\right).
\end{align*}
Similarly, H\"older's inequality combined with Propositions
\ref{boundonEDW1Yt2p} and \ref{boundonEDW2Xs2p} imply that
\begin{align*}
\mathbb{E}\left(\int_{r_1\vee r_2}^{t}  \abs{D^{W^{1}}_{r_1}
  Y_s^{\eta}D^{W^{2}}_{r_2} X_s^{\varepsilon}}^{2p}ds\right)
&\leq C\varepsilon^{p}\left(\varepsilon^{p}+\eta^{p}\right).
\end{align*}
For any $q \in \left\{ 1,2 \right\}$, H\"older's inequality,
Proposition \ref{boundonEDW1Xs2p} and Proposition
\ref{boundonintDY2pintermsofepsandeta} show that
\begin{align*}
\mathbb{E}\left(\left(\int_{r_1\vee r_2}^{t}  \abs{D^{W^{1}}_{r_1} X_s^{\varepsilon}D^{W^{2}}_{r_2} Y_s^{\eta}}^{q}ds\right)^{\frac{2p}{q}}\right)
&\leq C\mathbb{E}\left(\sup_{r_1\vee r_2\leq s\leq t} \abs{
                                                                                  D^{W^{1}}_{r_1} X_s^{\varepsilon}}^{2p p_1}\right)^{\frac{1}{p_1}}\\
  &\qquad\qquad\qquad \mathbb{E}\left(\left(\int_{r_1\vee r_2}^{t}  \abs{D^{W^{2}}_{r_2} Y_s^{\eta}}^{q}ds\right)^{\frac{2pq_1}{q}}\right)^{\frac{1}{q_1}}\nonumber\\
&\leq C \varepsilon^{p}\left(\eta^{\frac{p(2-q)}{q}}+\varepsilon^{p}+\eta^{p}\right).
\end{align*}
Now, Proposition \ref{proponEintDW1YDW2Y} guarantees that, for any $q \in \left\{ 1,2
\right\}$,
\begin{equation*}
\mathbb{E}\left(\left(\int_{r_1 \vee r_2}^{t}\abs{D_{r_1}^{W^1}Y_s^{\eta}D_{r_2}^{W^2}Y_s^{\eta}}^{q}ds\right)^{\frac{2p}{q}}\right) \leq C \varepsilon^p \eta^{\frac{p(2-q)}{q}}.
\end{equation*}
Finally, Proposition \ref{proponEintDW1W2Ystotheq} implies that
\begin{equation*}
\mathbb{E}\left(\left(\int_{r_1 \vee r_2}^{t}\abs{D_{r_1,r_2}^{W^1,W^2}Y_s^{\eta}}^{q}ds\right)^{\frac{2p}{q}}\right) \leq C\varepsilon^{p}\left(\eta^{\frac{p(2-q)}{q}}+\varepsilon^{p}+\eta^{p}\right)+
C \mathbb{E}\left(\int_{r_1\vee r_2}^{t}  \abs{D^{W_1,W_2}_{r_1,r_2}X^{\varepsilon}_{s}}^{2p}ds\right).
\end{equation*}
Putting all these estimates together, we obtain
\begin{align*}
\mathbb{E}\left( \abs{D^{W^{1},W^{2}}_{r_1,r_2} X_t^{\varepsilon}}^{2p}\right)&\leq C \varepsilon^{p}\left(\varepsilon^{p}+\eta^{p}+\frac{1}{\eta^p}e^{-\frac{K}{\eta}(r_1-r_2)}\right)\mathds{1}_{\{r_1\geq r_2\}}+C\left(1+\varepsilon^{p}\right)\varepsilon^{p}(\varepsilon^{p}+\eta^{p})\\
&\quad+C\varepsilon^{p}\left(\eta^{p}+\varepsilon^{p}+\eta^{p}\right)
+C\varepsilon^{p}\varepsilon^{p}\left(1+\varepsilon^{p}+\eta^{p}\right)+C \varepsilon^{p}\left(\varepsilon^{p}+\eta^{p}+\varepsilon^{p}\eta^{p}\right)
  \\
  &\quad +C
\varepsilon^{p}\left(\varepsilon^{p}+\eta^{p}+\varepsilon^{p}\right)+C\varepsilon^{p}\left(\eta^{p}+\varepsilon^{p}+\eta^{p}\right)+
    C\varepsilon^{2p}\left(1+\varepsilon^{p}+\eta^{p}\right)\\
  &\quad +
C\left(1+\varepsilon^{p}\right) \mathbb{E}\left(\int_{r_1\vee r_2}^{t}  \abs{D^{W_1,W_2}_{r_1,r_2}X^{\varepsilon}_{s}}^{2p}ds\right)
  \\
  &\quad +C\left(1+\varepsilon^{p}\right)\mathbb{E}\left(\int_{r_1\vee r_2}^{t}  \abs{D^{W_1,W_2}_{r_1,r_2}X^{\varepsilon}_{s}}^{2p}ds\right)\\
&\leq
                                                                                                                         C\bigg[\left(\frac{\varepsilon}{\eta}\right)^{p}e^{-\frac{K}{\eta}(r_1-r_2)}\mathds{1}_{\{r_1\geq r_2\}}+\varepsilon^{2p}+\varepsilon^{p}\eta^{p}\\
  &\qquad\qquad\qquad\qquad\qquad\qquad\qquad\qquad\quad +\mathbb{E}\left(\int_{r_1\vee r_2}^{t}  \abs{D^{W_1,W_2}_{r_1,r_2}X^{\varepsilon}_{s}}^{2p}ds\right)\bigg].
\end{align*}
Gr\"onwall's lemma finally yields
\begin{align*}
\mathbb{E}\left( \abs{D^{W^{1},W^{2}}_{r_1,r_2} X_t^{\varepsilon}}^{2p}\right)&\leq C \left[\varepsilon^{2p}+\varepsilon^{p}\eta^{p}+\left(\frac{\varepsilon}{\eta}\right)^pe^{-\frac{K}{\eta}(r_1-r_2)}\mathds{1}_{\{r_1\geq r_2\}}\right].
\end{align*}
\end{proof}

\begin{prop}
  \label{proponEintDW1W2Ystotheq}
Let $p\geq 1$ be a  natural number and $q \in \left\{ 1,2 \right\}$. Then, it holds that for some
constant $C >0$,
\begin{equation*}
\mathbb{E}\left(\left(\int_{r_1 \vee r_2}^{t}\abs{D_{r_1,r_2}^{W^1,W^2}Y_s^{\eta}}^{q}ds\right)^{\frac{2p}{q}}\right) \leq C\varepsilon^{p}\left(\eta^{\frac{p(2-q)}{q}}+\varepsilon^{p}+\eta^{p}\right)+
C \mathbb{E}\left(\int_{r_1\vee r_2}^{t}  \abs{D^{W_1,W_2}_{r_1,r_2}X^{\varepsilon}_{s}}^{2p}ds\right).
\end{equation*}
\end{prop}
\begin{proof}
Recall that $D^{W^{1}W^{2}}_{r_1,r_{2}} Y_t^{\eta}$ satisfies Equation
\eqref{equationsatisfiedbyDWWY}, which has an affine structure, so
that we have
\begin{align}
  \label{decompin3ofDW1W2Yt}
D_{r_1,r_2}^{W^{1},W^{2}}Y_t^{\eta} &= \frac{1}{\sqrt{\eta}}Z_{r_1
  \vee
  r_2,2}(t)\left(\partial_1\tau (X_{r_1}^{\varepsilon},Y_{r_2}^{\eta})
                 D_{r_1}^{W_1}X^{\varepsilon}_{r_2}+\partial_2\tau
                 (X_{r_2}^{\varepsilon},Y_{r_2}^{\eta})
                 D_{r_1}^{W_1}Y^{\eta}_{r_2}\right)\mathds{1}_{\left\{ r_2 \geq r_1 \right\}}
\nonumber \\
  &\quad + \frac{1}{\eta}Z_{r_1
  \vee
  r_2,2}(t) \int_{r_1 \vee r_2}^t Z_{r_1
  \vee
  r_2,2}^{-1}(s) \Big[b_2^{1,2}[f]\left(X_s^{\varepsilon},Y_s^{\eta}
    \right)\nonumber\\
  &\qquad\qquad\qquad\qquad\qquad\qquad\qquad\qquad\qquad - \partial_2\tau
    \left(X_s^{\varepsilon},Y_s^{\eta}
    \right)b_2^{1,2}[\tau]\left(X_s^{\varepsilon},Y_s^{\eta}
    \right)  \Big]ds\nonumber\\
  &\quad + \frac{1}{\sqrt{\eta}}Z_{r_1
  \vee
  r_2,2}(t) \int_{r_1 \vee r_2}^t Z_{r_1
  \vee
  r_2,2}^{-1}(s)b_2^{1,2}[\tau]\left(X_s^{\varepsilon},Y_s^{\eta}
    \right)dW_s^2\nonumber\\
  &= \hat{Q}_{r_1\vee r_2,1}^{\eta}(t)+\hat{Q}_{r_1\vee r_2,2}^{\eta}(t)+\hat{Q}_{r_1\vee r_2,3}^{\eta}(t),
\end{align}
where
\begin{align*}
\hat{Q}_{r_1\vee r_2,1}^{\eta}(t)&=\frac{1}{\sqrt{\eta}}Z_{r_1
  \vee
  r_2,2}(t)\left(\partial_1\tau (X_{r_1}^{\varepsilon},Y_{r_2}^{\eta})
                                   D_{r_1}^{W_1}X^{\varepsilon}_{r_2}+\partial_2\tau
                                   (X_{r_2}^{\varepsilon},Y_{r_2}^{\eta})
                            D_{r_1}^{W_1}Y^{\eta}_{r_2}\right)\mathds{1}_{\left\{
                                   r_2 \geq r_1 \right\}},
\end{align*}
\begin{align}
  \label{structureofQhat2}
 \hat{Q}_{r_1\vee r_2,2}^{\eta}(t)&=\frac{1}{\eta}Z_{r_1
  \vee
  r_2,2}(t) \int_{r_1 \vee r_2}^t Z_{r_1
  \vee
  r_2,2}^{-1}(s) \Big[b_2^{1,2}[f]\left(X_s^{\varepsilon},Y_s^{\eta}
    \right)-(\partial_{1}f\left(X_s^{\varepsilon},Y_s^{\eta}\right)\nonumber
  \\
  &\qquad\qquad -\partial_{2}\tau\left(X_s^{\varepsilon},Y_s^{\eta}
    \right)\partial_{1}\tau\left(X_s^{\varepsilon},Y_s^{\eta}
    \right)) D_{r_1,r_2}^{W^{1},W^{2}}X^{\varepsilon}_{s} - \partial_2\tau\left(X_s^{\varepsilon},Y_s^{\eta}
    \right)
    b_2^{1,2}[\tau]\left(X_s^{\varepsilon},Y_s^{\eta}
    \right)  \Big]ds\nonumber \\
  &\quad + \frac{1}{\sqrt{\eta}}Z_{r_1
  \vee
  r_2,2}(t) \int_{r_1 \vee r_2}^t Z_{r_1
  \vee
  r_2,2}^{-1}(s)\left(b_2^{1,2}[\tau]\left(X_s^{\varepsilon},Y_s^{\eta}
    \right)-\partial_{1}f\left(X_s^{\varepsilon},Y_s^{\eta}
    \right)
    D_{r_1,r_2}^{W^{1},W^{2}}X^{\varepsilon}_{s}\right)dW_s^2
\end{align}
and
\begin{align*}
  \hat{Q}_{r_1\vee r_2,3}^{\eta}(t)&=\frac{1}{\eta}Z_{r_1
  \vee
  r_2,2}(t) \int_{r_1 \vee r_2}^t Z_{r_1
  \vee
  r_2,2}^{-1}(s) (\partial_{1}f\left(X_s^{\varepsilon},Y_s^{\eta}
                                     \right)\\
  &\qquad\qquad\qquad\qquad\qquad\qquad\qquad\qquad -\partial_{2}\tau\left(X_s^{\varepsilon},Y_s^{\eta}
    \right)\partial_{1}\tau\left(X_s^{\varepsilon},Y_s^{\eta}
    \right)) D_{r_1,r_2}^{W^{1},W^{2}}X^{\varepsilon}_{s}  ds\\
  &\quad + \frac{1}{\sqrt{\eta}}Z_{r_1
  \vee
  r_2,2}(t) \int_{r_1 \vee r_2}^t Z_{r_1
  \vee
  r_2,2}^{-1}(s)\partial_{1}f \left(X_s^{\varepsilon},Y_s^{\eta}
    \right) D_{r_1,r_2}^{W^{1},W^{2}}X^{\varepsilon}_{s}dW_s^2.
\end{align*}
We can hence write
\begin{align}
  \label{tripledecompofEintDW1W2Y}
\mathbb{E}\left(\left(\int_{r_1 \vee
  r_2}^{t}\abs{D_{r_1,r_2}^{W^1,W^2}Y_s^{\eta}}^{q}ds\right)^{\frac{2p}{q}}\right)
  \leq C \sum_{i=1}^{3}\mathbb{E}\left(\left(\int_{r_1 \vee r_2}^{t}\abs{\hat{Q}_{r_1\vee r_2,i}^{\eta}(s)}^{q}ds\right)^{\frac{2p}{q}}\right)
\end{align}
and estimate these three terms separately. For the term corresponding
to $i=1$ in \eqref{tripledecompofEintDW1W2Y}, applying Proposition \ref{propinitialconditionDY} with $W = \left(\partial_1\tau (X_{r_1}^{\varepsilon},Y_{r_2}^{\eta})
                                   D_{r_1}^{W_1}X^{\varepsilon}_{r_2}+\partial_2\tau
                                   (X_{r_2}^{\varepsilon},Y_{r_2}^{\eta})
                            D_{r_1}^{W_1}Y^{\eta}_{r_2}\right)\mathds{1}_{\left\{
                                   r_2 \geq r_1 \right\}}$ together
                               with Propositions \ref{boundonEDW1Xs2p} and \ref{boundonEDW1Yt2p}
                               immediately yields
\begin{align}
\label{finalestimateforQhat1}
\mathbb{E}\left(\left(\int_{r_1\vee r_2}^{t}  \abs{\hat{Q}_{r_1\vee r_2,1}^{\eta}(s)}^{q}ds\right)^{\frac{2p}{q}}\right)&\leq C\varepsilon^p\eta^{\frac{p(2-q)}{q}}.
\end{align}
For the term corresponding to $i=3$ in \eqref{tripledecompofEintDW1W2Y}, note that the term $\hat{Q}_{r_1\vee r_2,3}^{\eta}(t)$ satisfies a
linear stochastic differential equation of the form
\begin{align*}
\hat{Q}_{r_1\vee r_2,3}^{\eta}(t)&=\frac{1}{\eta}\int_{r_1\vee r_2}^{t}\left[\partial_{2}f(X^{\varepsilon}_{s},Y^{\eta}_{s})\hat{Q}_{r_1\vee r_2,3}^{\eta}(s)+\partial_{1}f(X^{\varepsilon}_{s},Y^{\eta}_{s}) D^{W_1,W_2}_{r_1,r_2}X^{\varepsilon}_{s}\right]ds\\
&\quad +\frac{1}{\sqrt{\eta}}\int_{r_1\vee r_2}^{t}\left[\partial_{2}\tau(X^{\varepsilon}_{s},Y^{\eta}_{s})\hat{Q}_{r_1\vee r_2,3}^{\eta}(s)+\partial_{1}\tau(X^{\varepsilon}_{s},Y^{\eta}_{s}) D^{W_1,W_2}_{r_1,r_2}X^{\varepsilon}_{s}\right]dW_{s}^{2}.
\end{align*}
As the structure of this equation is the same as the one satisfied by
$D^{W_1,W_1}_{r_1,r_2}Y^{\eta}_{t}$ (with
$D^{W_1,W_2}_{r_1,r_2}X^{\varepsilon}_{s}$ instead of
$D^{W_1,W_1}_{r_1,r_2}X^{\varepsilon}_{s}$), the same methodology used for bounding
$D^{W_1,W_1}_{r_1,r_2}Y^{\eta}_{t}$ in terms of
$D^{W_1,W_1}_{r_1,r_2}X^{\varepsilon}_{t}$ in Proposition \ref{proponEintDW1W1YslessEintDW1W1Xs} yields
\begin{align}
  \label{finalestimateforQhat3}
\mathbb{E}\left(\left(\int_{r_1\vee r_2}^{t}  \abs{\hat{Q}_{r_1\vee r_2,3}^{\eta}(s)}^{q}ds\right)^{\frac{2p}{q}}\right)&\leq C \mathbb{E}\left(\int_{r_1\vee r_2}^{t}  \abs{\hat{Q}_{r_1\vee r_2,3}^{\eta}(s)}^{2p}ds\right)\nonumber\\
&\leq C \mathbb{E}\left(\int_{r_1\vee r_2}^{t}  \abs{D^{W_1,W_2}_{r_1,r_2}X^{\varepsilon}_{s}}^{2p}ds\right).
\end{align}
For the term corresponding to $i=2$ in
\eqref{tripledecompofEintDW1W2Y}, inspecting the structure of
$\hat{Q}_{r_1\vee r_2,2}^{\eta}(t)$ given in \eqref{structureofQhat2} shows that it is enough to estimate terms of the form
\begin{align*}
&\mathbb{E}\left(\left(\int_{r_1 \vee r_2}^t
  \abs{\frac{1}{\eta}Z_{r_1
  \vee
  r_2,2}(s)\int_{r_1 \vee r_2}^s Z_{r_1
  \vee
  r_2,2}^{-1}(u)F_u G_{u} du }^q ds \right)^{\frac{2p}{q}}
                 \right)
\end{align*}
and
\begin{align*}
&\mathbb{E}\left(\left(\int_{r_1 \vee r_2}^t
  \abs{\frac{1}{\sqrt{\eta}}Z_{r_1
  \vee
  r_2,2}(s)\int_{r_1 \vee r_2}^s Z_{r_1
  \vee
  r_2,2}^{-1}(u)F_u G_{u} dW^{2}_{u} }^q ds \right)^{\frac{2p}{q}}
                 \right),
\end{align*}
where the product $F_u G_{u}$ is any of the terms $D_{r_1}^{W^{1}}X_{u}^{\varepsilon}D_{r_2}^{W^{2}}X_{u}^{\varepsilon}$,
$D_{r_1}^{W^{1}}X_{u}^{\varepsilon}D_{r_2}^{W^{2}}Y_{u}^{\eta}$,
$D_{r_1}^{W^{1}}Y_{u}^{\eta}D_{r_2}^{W^{2}}X_{u}^{\varepsilon}$ or
$D_{r_1}^{W^{1}}Y_{u}^{\eta}D_{r_2}^{W^{2}}Y_{u}^{\eta}$.
Whenever $F_u G_{u} =
D_{r_1}^{W^{1}}X_{u}^{\varepsilon}D_{r_2}^{W^{2}}X_{u}^{\varepsilon}$ and $F_u G_{u} =
D_{r_1}^{W^{1}}Y_{u}^{\eta}D_{r_2}^{W^{2}}X_{u}^{\varepsilon}$,
Propositions \ref{expintegratedFuGudu} and \ref{expintegratedFuGudWu}
respectively imply
\begin{align*}
&\mathbb{E}\left(\left(\int_{r_1 \vee r_2}^t
  \abs{\frac{1}{\eta}Z_{r_1
  \vee
  r_2,2}(s)\int_{r_1 \vee r_2}^s Z_{r_1
  \vee
  r_2,2}^{-1}(u)D_{r_1}^{W^{1}}X_{u}^{\varepsilon}D_{r_2}^{W^{2}}X_{u}^{\varepsilon} du}^q ds \right)^{\frac{2p}{q}}
                 \right)\leq C\varepsilon^{p}(\varepsilon^{p}+\eta^{p}),\nonumber\\
 &\mathbb{E}\left(\left(\int_{r_1 \vee r_2}^t
  \abs{\frac{1}{\eta}Z_{r_1
  \vee
  r_2,2}(s)\int_{r_1 \vee r_2}^s Z_{r_1
  \vee
  r_2,2}^{-1}(u)D_{r_1}^{W^{1}}Y_{u}^{\eta}D_{r_2}^{W^{2}}X_{u}^{\varepsilon} du}^q ds \right)^{\frac{2p}{q}}
                 \right)\leq C\varepsilon^{p}(\varepsilon^{p}+\eta^{p})
\end{align*}
and
\begin{align*}
&\mathbb{E}\left(\left(\int_{r_1 \vee r_2}^t
  \abs{\frac{1}{\sqrt{\eta}}Z_{r_1
  \vee
  r_2,2}(s)\int_{r_1 \vee r_2}^s Z_{r_1
  \vee
  r_2,2}^{-1}(u)D_{r_1}^{W^{1}}X_{u}^{\varepsilon}D_{r_2}^{W^{2}}X_{u}^{\varepsilon} dW^{2}_u}^q ds \right)^{\frac{2p}{q}}
                 \right)\leq C\varepsilon^{p}(\varepsilon^{p}+\eta^{p}),\nonumber\\
 &\mathbb{E}\left(\left(\int_{r_1 \vee r_2}^t
  \abs{\frac{1}{\sqrt{\eta}}Z_{r_1
  \vee
  r_2,2}(s)\int_{r_1 \vee r_2}^s Z_{r_1
  \vee
  r_2,2}^{-1}(u)D_{r_1}^{W^{1}}Y_{u}^{\eta}D_{r_2}^{W^{2}}X_{u}^{\varepsilon} dW^{2}_u}^q ds \right)^{\frac{2p}{q}}
                 \right)\leq C\varepsilon^{p}(\varepsilon^{p}+\eta^{p}).
\end{align*}
It remains to deal with the cases where the product term $F_u G_u$ is
either $F_u
G_u=D_{r_1}^{W^{1}}X_{u}^{\varepsilon}D_{r_2}^{W^{2}}Y_{u}^{\eta}$ or
$F_u G_u = D_{r_1}^{W^{1}}Y_{u}^{\eta}D_{r_2}^{W^{2}}Y_{u}^{\eta}$. It
is enough to treat the second case where
$F_u=D_{r_1}^{W^{1}}Y_{u}^{\eta}$ as the first case where
$F_u=D_{r_1}^{W^{1}}X_{u}^{\varepsilon}$ works the same and is even
simpler. To proceed, we recall the decomposition $D_r^{W^2}
Y_t^{\eta}=Q^{\eta}_{r,1}(t)+Q^{\eta}_{r,2}(t)$, so that we can write
\begin{align*}
D_{r_1}^{W^{1}}Y_{u}^{\eta}D_{r_2}^{W^{2}}Y_{u}^{\eta}&=D_{r_1}^{W^{1}}Y_{u}^{\eta}Q^{\eta}_{r_2,1}(u)+D_{r_1}^{W^{1}}Y_{u}^{\eta}Q^{\eta}_{r_2,2}(u).
\end{align*}
As was already pointed out, $Q^{\eta}_{r_2,2}(u)$ satisfies the same
equation as $D_{r_2}^{W^{1}}Y_{u}^{\eta}$, but with
$D^{W^{1}}_{r_2}X^{\varepsilon}_v$ replaced by
$D^{W^{2}}_{r_2}X^{\varepsilon}_v$. Hence, using Propositions
\ref{expintegratedFuGudu} and \ref{expintegratedFuGudWu} again, we have
\begin{align*}
&\mathbb{E}\left(\left(\int_{r_1 \vee r_2}^t
  \abs{\frac{1}{\eta}Z_{r_1
  \vee
  r_2,2}(s)\int_{r_1 \vee r_2}^s Z_{r_1
  \vee
  r_2,2}^{-1}(u)D_{r_1}^{W^{1}}Y_{u}^{\eta}Q^{\eta}_{r_2,2}(u)du}^q ds \right)^{\frac{2p}{q}}
                 \right)\leq C\varepsilon^{p}(\varepsilon^{p}+\eta^{p}),
\end{align*}
and
\begin{align*}
&\mathbb{E}\left(\left(\int_{r_1 \vee r_2}^t
  \abs{\frac{1}{\sqrt{\eta}}Z_{r_1
  \vee
  r_2,2}(s)\int_{r_1 \vee r_2}^s Z_{r_1
  \vee
  r_2,2}^{-1}(u)D_{r_1}^{W^{1}}Y_{u}^{\eta}Q^{\eta}_{r_2,2}(u)dW^{2}_u}^q ds \right)^{\frac{2p}{q}}
                 \right)\leq C\varepsilon^{p}(\varepsilon^{p}+\eta^{p}).
\end{align*}
We are now left with the case where $F_u
G_u=D_{r_1}^{W^{1}}Y_{u}^{\eta}Q_{r_2,1}^{\eta}(u)$. Define
\begin{equation*}
V_{r_{1}\vee r_{2}}(s)=\int_{r_1 \vee r_2}^s Z_{r_1
  \vee
  r_2,2}^{-1}(u)D_{r_1}^{W^{1}}Y_{u}^{\eta}Q^{\eta}_{r_2,1}(u)du.
\end{equation*}
Then, we can write
\begin{align*}
&\mathbb{E}\left(\left(\int_{r_1 \vee r_2}^t
  \abs{\frac{1}{\eta}Z_{r_1
  \vee
  r_2,2}(s)\int_{r_1 \vee r_2}^s Z_{r_1
  \vee
  r_2,2}^{-1}(u)D_{r_1}^{W^{1}}Y_{u}^{\eta}Q^{\eta}_{r_2,1}(u)du}^q ds \right)^{\frac{2p}{q}}
                 \right)\\
  &\qquad\qquad\qquad = \mathbb{E}\left(\left(\int_{r_1\vee r_2}^{t}  \abs{\frac{1}{\eta}Z_{r_1
  \vee
    r_2,2}(s)V_{r_{1}\vee r_{2}}(s)}^{q}ds\right)^{\frac{2p}{q}}\right)\\
  &\qquad\qquad\qquad \leq C \int_{r_1\vee r_2}^{t}\cdots \int_{r_1\vee r_2}^{t}\prod_{i=1}^{2p/q}\left(\mathbb{E}\left(  \abs{\frac{1}{\eta}Z_{r_1
  \vee
  r_2,2}(s_i)V_{r_{1}\vee r_{2}}(s_i)}^{2p}\right)\right)^{
\frac{q}{2p}}ds_1
\cdots ds_{2p/q}.
\end{align*}
Applying Lemma \ref{lemmaonEVr1veer2sZr22p} and integrating yields
\begin{align*}
&\mathbb{E}\left(\left(\int_{r_1 \vee r_2}^t
  \abs{\frac{1}{\eta}Z_{r_1
  \vee
  r_2,2}(s)\int_{r_1 \vee r_2}^s Z_{r_1
  \vee
  r_2,2}^{-1}(u)D_{r_1}^{W^{1}}Y_{u}^{\eta}Q^{\eta}_{r_2,1}(u)du}^q ds \right)^{\frac{2p}{q}}
                 \right) \leq C \varepsilon^p \eta^{\frac{p(2-q)}{q}}.
\end{align*}
The same analysis applied to the stochastic integral counterpart of
the above expression gives
\begin{align*}
&\mathbb{E}\left(\left(\int_{r_1 \vee r_2}^t
  \abs{\frac{1}{\sqrt{\eta}}Z_{r_1
  \vee
  r_2,2}(s)\int_{r_1 \vee r_2}^s Z_{r_1
  \vee
  r_2,2}^{-1}(u)D_{r_1}^{W^{1}}Y_{u}^{\eta}Q^{\eta}_{r_2,1}(u)dW^2_u}^q ds \right)^{\frac{2p}{q}}
                 \right) \leq C \varepsilon^p \eta^{\frac{p(2-q)}{q}}.
\end{align*}
Putting together all the estimates we just obtained, we finally get
\begin{equation}
  \label{finalestimateforQhat2}
\mathbb{E}\left(\left(\int_{r_1\vee r_2}^{t}  \abs{\hat{Q}_{r_1\vee
        r_2,2}^{\eta}(s)}^{q}ds\right)^{\frac{2p}{q}}\right) \leq C
\varepsilon^p \left(\eta^{\frac{p(2-q)}{q}} +  \varepsilon^p + \eta^p \right).
\end{equation}
Recalling the decomposition \eqref{decompin3ofDW1W2Yt} and using the estimates \eqref{finalestimateforQhat1},
\eqref{finalestimateforQhat3} and \eqref{finalestimateforQhat2} allows
us to write
\begin{align*}
\mathbb{E}\left(\left(\int_{r_1\vee r_2}^{t}  \abs{D^{W^{1}W^{2}}_{r_1,r_{2}} Y_s^{\eta}}^{q}ds\right)^{\frac{2p}{q}}\right)&\leq C\varepsilon^{p}\left(\eta^{\frac{p(2-q)}{q}}+\varepsilon^{p}+\eta^{p}\right)+
C \mathbb{E}\left(\int_{r_1\vee r_2}^{t}  \abs{D^{W_1,W_2}_{r_1,r_2}X^{\varepsilon}_{s}}^{2p}ds\right),
\end{align*}
which concludes the proof.
\end{proof}

\begin{prop}
  \label{proponEintDW1YDW2Y}
Let $p\geq 1$ be a  natural number and $q \in \left\{ 1,2 \right\}$. Then, it holds that for some
constant $C >0$,
\begin{equation*}
\mathbb{E}\left(\left(\int_{r_1 \vee r_2}^{t}\abs{D_{r_1}^{W^1}Y_s^{\eta}D_{r_2}^{W^2}Y_s^{\eta}}^{q}ds\right)^{\frac{2p}{q}}\right) \leq C \varepsilon^p \eta^{\frac{p(2-q)}{q}}.
\end{equation*}
\end{prop}
\begin{proof}
Recall from \eqref{decompositionofDW2Yt} that
$D_{r_2}^{W^2}Y_s^{\eta}$ can be decomposed as
$D_{r_2}^{W^2}Y_s^{\eta} = Q^{\eta}_{r_2,1}(s)+Q^{\eta}_{r_2,2}(s)$,
where $Q^{\eta}_{r_2,1}(s)$ and $Q^{\eta}_{r_2,2}(s)$ are defined in
\eqref{defQ1} and \eqref{defQ2}, respectively. This allows us to write
\begin{align}
  \label{decompEintDW1YDW2Y}
\mathbb{E}\left(\left(\int_{r_1\vee r_2}^{t}  \abs{D^{W^{1}}_{r_1}
  Y_s^{\eta}D^{W^{2}}_{r_2}
  Y_s^{\eta}}^{q}ds\right)^{\frac{2p}{q}}\right)&\leq  C
                                                  \mathbb{E}\left(\left(\int_{r_1\vee
                                                  r_2}^{t}
                                                  \abs{D^{W^{1}}_{r_1}
                                                  Y_s^{\eta}Q^{\eta}_{r_2,1}(s)}^{q}ds\right)^{\frac{2p}{q}}\right)\nonumber
  \\
&\quad +C \mathbb{E}\left(\left(\int_{r_1\vee r_2}^{t}  \abs{D^{W^{1}}_{r_1} Y_s^{\eta}Q^{\eta}_{r_2,2}(s)}^{q}ds\right)^{\frac{2p}{q}}\right).
\end{align}
Let us start by estimating the second term. We know that $Q^{\eta}_{r_2,2}(s)$ satisfies Equation
\eqref{Eq:SDE_Q2D2Y}, which is the same equation satisfied by
$D^{W^{1}}_{r_2} Y_s^{\eta}$ (see Equation \eqref{geneqforDY}), only with
$D^{W^{1}}_{r_2} X_u^{\varepsilon}$ replaced by $D^{W^{2}}_{r_2}
X_u^{\varepsilon}$. The same analysis as the one performed in Proposition
\ref{propDW1YleqDW1X} and its proof will hence yield
\begin{equation*}
\mathbb{E}\left(\int_{r_1 \vee
    r_2}^{t}\abs{Q^{\eta}_{r_2,2}(s)}^{2p}ds\right) \leq
C\mathbb{E}\left(\int_{r_1 \vee r_2}^{t}\abs{D_r^{W^2}X_s^{\varepsilon}}^{2p}ds\right).
\end{equation*}
Then, H\"older's inequality and Proposition \ref{boundonEDW2Xs2p} immediately imply that
\begin{align*}
\mathbb{E}\left(\left(\int_{r_1\vee r_2}^{t}  \abs{D^{W^{1}}_{r_1}
  Y_s^{\eta}Q^{\eta}_{r_2,2}(s)}^{q}ds\right)^{\frac{2p}{q}}\right)&\leq
                                                                     C
                                                                     \varepsilon^{p}\left(\varepsilon^{p}+\eta^{p}\right).
\end{align*}
For the first term on the right-hand side of
\eqref{decompEintDW1YDW2Y}, recall that as $\tau$ is assumed to be
uniformly bounded, $\abs{Q^{\eta}_{r_2,1}(s)} \leq
\frac{C}{\sqrt{\eta}}Z_{r_2,2}(s)$. We can hence write
\begin{align*}
&\mathbb{E}\left(\left(\int_{r_1\vee r_2}^{t}  \abs{D^{W^{1}}_{r_1}
                 Y_s^{\eta}Q^{\eta}_{r_2,1}(s)}^{q}ds\right)^{\frac{2p}{q}}\right)
                 \leq C \eta^{-p}\mathbb{E}\left(\left(\int_{r_1\vee r_2}^{t}  \abs{D^{W^{1}}_{r_1} Y_s^{\eta}Z_{r_2,2}(s)}^{q}ds\right)^{\frac{2p}{q}}\right)\nonumber\\
&\qquad\qquad\qquad\qquad \leq C \eta^{-p}\int_{r_1\vee r_2}^{t}\cdots \int_{r_1\vee r_2}^{t}\mathbb{E}\left(\prod_{i=1}^{2p/q}  \abs{D^{W^{1}}_{r_1} Y_{s_i}^{\eta}Z_{r_2,2}(s_i)}^{q}\right)ds_1
\cdots ds_{2p/q}\nonumber\\
&\qquad\qquad\qquad\qquad \leq C \eta^{-p}\int_{r_1\vee r_2}^{t}\cdots \int_{r_1\vee r_2}^{t}\prod_{i=1}^{2p/q}\mathbb{E}\left(  \abs{D^{W^{1}}_{r_1} Y_{s_i}^{\eta}Z_{r_2,2}(s_i)}^{2p}\right)^{
\frac{q}{2p}}ds_1
\cdots ds_{2p/q}.
\end{align*}
Finally, Lemma \ref{lemmaonEDW1YZr22p} allows us to write
\begin{align}
\label{analogouscalcusedintheW2W2case1}
\mathbb{E}\left(\left(\int_{r_1\vee r_2}^{t}  \abs{D^{W^{1}}_{r_1}
                 Y_s^{\eta}Q^{\eta}_{r_2,1}(s)}^{q}ds\right)^{\frac{2p}{q}}\right)
                 &\leq C \eta^{-p}\varepsilon^p \left( \int_{r_1\vee
                 r_2}^{t} e^{-\frac{qD}{2p\eta}(s-r_1 \vee
                 r_2)}ds\right)^{\frac{2p}{q}}\nonumber\\
  &\leq C \varepsilon^p \eta^{\frac{p(2-q)}{q}}.
\end{align}
\end{proof}

 \subsection{Bound for $\mathbb{E}\left( \abs{D^{W^{2},W^{2}}_{r_1,r_2} X_t^{\varepsilon}}^{2p}\right)$}

 \begin{prop}\label{P:BoundSecondMalliavinDW2DW2X}
Let $p\geq 1$ be a  natural number. Then, it holds that for some
constant $C >0$,
\begin{align*}
\mathbb{E}\left( \abs{D^{W^{2},W^{2}}_{r_1,r_2} X_t^{\varepsilon}}^{2p}\right)
&\leq C \bigg[\varepsilon^{2p}+\eta^{2p}+\varepsilon^{p}\eta^{p}+
                \left(1+\left(\frac{\varepsilon}{\eta}\right)^{p}\right)e^{-\frac{K}{2\eta}(r_1
                 \vee r_2 - r_1 \wedge r_2)}\bigg].
\end{align*}
\end{prop}
\begin{proof}
Using \eqref{SDEforDWj1Wj2Xt}, the Burkholder-Davis-Gundy inequality, and Assumption \ref{A:Assumptiongeneral}, we can write
  \begin{align*}
&\mathbb{E}\left( \abs{D^{W^{2},W^{2}}_{r_1,r_2} X_t^{\varepsilon}}^{2p}\right)\leq C\left(1+\varepsilon^{p}\right)\mathbb{E}\left( \int_{r_1\vee r_2}^{t}  \abs{D^{W^{2}}_{r_1} X_s^{\varepsilon}D^{W^{2}}_{r_2} X_s^{\varepsilon}}^{2p}ds\right)\nonumber\\
&\qquad\qquad  +C\mathbb{E}\left(\left|\int_{r_1\vee r_2}^{t}  \abs{D^{W^{2}}_{r_1} X_s^{\varepsilon}D^{W^{2}}_{r_2} Y_s^{\eta}}ds\right|^{2p}\right) +C\varepsilon^{p}\mathbb{E}\left(\left|\int_{r_1\vee r_2}^{t}  \abs{D^{W^{2}}_{r_1} X_s^{\varepsilon}D^{W^{2}}_{r_2} Y_s^{\eta}}^{2}ds\right|^{p}
\right)\nonumber\\
&\qquad\qquad +C\mathbb{E}\left(\left|\int_{r_1\vee r_2}^{t}  \abs{D^{W^{2}}_{r_1} Y_s^{\eta}D^{W^{2}}_{r_2} X_s^{\varepsilon}}ds\right|^{2p}\right) +C\varepsilon^{p}\mathbb{E}\left(\left|\int_{r_1\vee r_2}^{t}  \abs{D^{W^{2}}_{r_1} Y_s^{\eta}D^{W^{2}}_{r_2}  X_s^{\varepsilon}}^{2}ds\right|^{p}\right)\nonumber\\
&\qquad\qquad  +C\mathbb{E}\left(\left|\int_{r_1\vee r_2}^{t}  \abs{D^{W^{2}}_{r_1} Y_s^{\eta}D^{W^{2}}_{r_2} Y_s^{\eta}}ds\right|^{2p}\right) +C\varepsilon^{p}\mathbb{E}\left(\left|\int_{r_1\vee r_2}^{t}  \abs{D^{W^{2}}_{r_1} Y_s^{\eta}D^{W^{2}}_{r_2}  Y_s^{\eta}}^{2}ds\right|^{p}\right)\nonumber\\
&\qquad\qquad  +C\mathbb{E}\left(\left|\int_{r_1\vee r_2}^{t}  \abs{D^{W^{2},W^{2}}_{r_1,r_2} Y_s^{\eta}}ds\right|^{2p}\right)+C\varepsilon^{p}\mathbb{E}\left(\left|\int_{r_1\vee r_2}^{t}  \abs{D^{W^{2},W^{2}}_{r_1,r_2} Y_s^{\eta}}^{2}ds\right|^{p}\right)\nonumber\\
&\qquad\qquad  +C\left(1+\varepsilon^{p}\right)\mathbb{E}\left(\int_{r_1\vee r_2}^{t} \abs{D^{W^{2},W^{2}}_{r_1,r_2} X_s^{\varepsilon}}^{2p}ds\right).
  \end{align*}
Let us bound these terms individually. Applying H\"older's inequality together with Proposition \ref{boundonEDW2Xs2p} yields
\begin{align*}
\mathbb{E}\left( \int_{r_1\vee r_2}^{t}  \abs{D^{W^{2}}_{r_1} X_s^{\varepsilon}D^{W^{2}}_{r_2} X_s^{\varepsilon}}^{2p}ds\right)
&\leq C\left( \varepsilon^{2p}+\eta^{2p}\right).
\end{align*}
Similarly, for $q \in \left\{ 1,2 \right\}$, H\"older's inequality, Proposition \ref{boundonEDW2Xs2p}
and Proposition \ref{boundonintDY2pintermsofepsandeta} imply that
\begin{align*}
&\mathbb{E}\left(\left(\int_{r_1\vee r_2}^{t}  \abs{D^{W^{2}}_{r_1} X_s^{\varepsilon}D^{W^{2}}_{r_2} Y_s^{\eta}}^{q}ds\right)^{\frac{2p}{q}}\right)\\
&\qquad\qquad\qquad\qquad \leq C\mathbb{E}\left(\sup_{r_1\vee r_2\leq s\leq t} \abs{ D^{W^{2}}_{r_1} X_s^{\varepsilon}}^{2p p_1}\right)^{\frac{1}{p_1}}\mathbb{E}\left(\left(\int_{r_1\vee r_2}^{t}  \abs{D^{W^{2}}_{r_2} Y_s^{\eta}}^{q}ds\right)^{\frac{2pp_1}{q}}\right)^{\frac{1}{q_1}}\\
&\qquad\qquad\qquad\qquad \leq C \left(\varepsilon^{p}+\eta^{p}\right)\left(\eta^{\frac{p(2-q)}{q}}+\varepsilon^{p}+\eta^{p}\right).
\end{align*}
Combining the two last estimates with Propositions \ref{proponEintDW2YDW2Ytoqthento2poverq} and \ref{proponEintDW2W2Ystotheq}
gives us the global estimate
  \begin{align*}
\mathbb{E}\left( \abs{D^{W^{2},W^{2}}_{r_1,r_2} X_t^{\varepsilon}}^{2p}\right)&\leq C\left(1+\varepsilon^{p}\right)\left( \varepsilon^{2p}+\eta^{2p}\right)\nonumber\\
&\quad  +C\left(\varepsilon^{2p}+\eta^{2p}\right)+C\varepsilon^{p}\left(\varepsilon^{p}+\eta^{p}\right)\left(1+\varepsilon^{p}+\eta^{p}\right)\nonumber\\
&\quad +C\left(\varepsilon^{2p}+\eta^{2p}\right)+C\varepsilon^{p}\left(\varepsilon^{p}+\eta^{p}\right)\left(1+\varepsilon^{p}+\eta^{p}\right)\nonumber\\
&\quad +C\left(\varepsilon^{2p} + \eta^{2p} +
\eta^{p}\left(\varepsilon^{p} + \eta^{p}  \right) +
  e^{-\frac{K}{2\eta}(r_1 \vee r_2 - r_1
                                                                                                                                                      \wedge r_2)}\right)\\
    &\quad +C\varepsilon^{p}\left(\varepsilon^{2p} + \eta^{2p} +
\left(\varepsilon^{p} + \eta^{p}  \right) +
  \frac{1}{\eta^p}e^{-\frac{K}{2\eta}(r_1 \vee r_2 - r_1
    \wedge r_2)}\right)\nonumber\\
&\quad  +C \bigg( \eta^{p}\left(\varepsilon^{p}+\eta^{p}+\frac{1}{\eta^{p}}e^{-\frac{K}{\eta}(r_1\vee
                                                          r_2-r_1\wedge
                                     r_2)}  \right)\\
    &\qquad\qquad\qquad\qquad\qquad  +  e^{-\frac{K}{2\eta}(r_1
    \vee r_2 - r_1 \wedge r_2)} +\left( \varepsilon^{2p} + \eta^{2p}
    \right) \bigg)\\
    &\quad +C\varepsilon^{p}\bigg(\left(\varepsilon^{p}+\eta^{p}+\frac{1}{\eta^{p}}e^{-\frac{K}{\eta}(r_1\vee
                                                          r_2-r_1\wedge
      r_2)}  \right)\\
    &\qquad\qquad\qquad\qquad\qquad +  \frac{1}{\eta^p}e^{-\frac{K}{2\eta}(r_1
    \vee r_2 - r_1 \wedge r_2)} +\left( \varepsilon^p + \eta^p
    \right)\left(1 +  \varepsilon^p + \eta^p
    \right)\bigg)\nonumber\\
&\quad  +C\left(1+\varepsilon^{p}\right)\mathbb{E}\left(\int_{r_1\vee r_2}^{t} \abs{D^{W^{2},W^{2}}_{r_1,r_2} X_s^{\varepsilon}}^{2p}ds\right).
  \end{align*}
For $\varepsilon,\eta<1$, simplifying this bound yields
  \begin{align*}
\mathbb{E}\left( \abs{D^{W^{2},W^{2}}_{r_1,r_2} X_t^{\varepsilon}}^{2p}\right)
&\leq C\bigg[(\varepsilon^{2p}+\eta^{2p}+\varepsilon^{p}\eta^{p}) + e^{-\frac{K}{\eta}(r_1
                                                                              \vee r_2 - r_1 \wedge r_2)}\\
    &\quad +
                 \left(\frac{\varepsilon}{\eta}\right)^{p}e^{-\frac{K}{2\eta}(r_1
                 \vee r_2 - r_1 \wedge r_2)} + \mathbb{E}\left(\int_{r_1\vee r_2}^{t}  \abs{D^{W_2,W_2}_{r_1,r_2}X^{\varepsilon}_{s}}^{2p}ds\right)\bigg].
  \end{align*}
Finally, applying Gr\"onwall's lemma, we obtain
\begin{align*}
\mathbb{E}\left( \abs{D^{W^{2},W^{2}}_{r_1,r_2} X_t^{\varepsilon}}^{2p}\right)
&\leq C \bigg[\varepsilon^{2p}+\eta^{2p}+\varepsilon^{p}\eta^{p}+
                \left(1+\left(\frac{\varepsilon}{\eta}\right)^{p}\right)e^{-\frac{K}{2\eta}(r_1
                 \vee r_2 - r_1 \wedge r_2)}\bigg],
\end{align*}
as desired.
\end{proof}

\begin{prop}
  \label{proponEintDW2YDW2Ytoqthento2poverq}
Let $p\geq 1$ be a  natural number and $q \in \left\{ 1,2 \right\}$. Then, it holds that for some
constant $C >0$,
\begin{align*}
\mathbb{E}\left(\left(\int_{r_1 \vee
      r_2}^{t}\abs{D_{r_1}^{W^2}Y_s^{\eta}D_{r_2}^{W^2}Y_s^{\eta}}^{q}ds\right)^{\frac{2p}{q}}\right)
&\leq C \bigg[ \varepsilon^{2p} + \eta^{2p} +
\eta^{\frac{p(2-q)}{q}}\left(\varepsilon^{p} + \eta^{p}  \right)\\
&\qquad\qquad\qquad\qquad\qquad +
  \eta^{\frac{2p(1-q)}{q}}e^{-\frac{K}{2\eta}(r_1 \vee r_2 - r_1
    \wedge r_2)} \bigg].
\end{align*}
\end{prop}
\begin{proof}
The decomposition
$D_r^{W^2}Y_t^{\eta}=Q^{\eta}_{r,1}(t)+Q^{\eta}_{r,2}(t)$ given in
\eqref{decompositionofDW2Yt} allows us to write
\begin{align}
  \label{4termdecompofEintDr1W2YDr2W2Ytotheqdstothe2poverq}
\mathbb{E}\left(\left(\int_{r_1\vee r_2}^{t}  \abs{D^{W^{2}}_{r_1} Y_s^{\eta}D^{W^{2}}_{r_2} Y_s^{\eta}}^{q}ds\right)^{\frac{2p}{q}}\right)&\leq  C \mathbb{E}\left(\left(\int_{r_1\vee r_2}^{t}  \abs{Q^{\eta}_{r_1,1}(s)Q^{\eta}_{r_2,1}(s)}^{q}ds\right)^{\frac{2p}{q}}\right)\nonumber\\
&\quad +C \mathbb{E}\left(\left(\int_{r_1\vee r_2}^{t}  \abs{Q^{\eta}_{r_1,1}(s)Q^{\eta}_{r_2,2}(s)}^{q}ds\right)^{\frac{2p}{q}}\right)\nonumber\\
&\quad+C \mathbb{E}\left(\left(\int_{r_1\vee r_2}^{t}  \abs{Q^{\eta}_{r_1,2}(s)Q^{\eta}_{r_2,1}(s)}^{q}ds\right)^{\frac{2p}{q}}\right)\nonumber\\
&\quad+C \mathbb{E}\left(\left(\int_{r_1\vee r_2}^{t}  \abs{Q^{\eta}_{r_1,2}(s)Q^{\eta}_{r_2,2}(s)}^{q}ds\right)^{\frac{2p}{q}}\right).
\end{align}
We will now estimate these four terms separately. Recall that $Q^{\eta}_{r,2}(t)$ satisfies the stochastic differential
equation \eqref{Eq:SDE_Q2D2Y}, which has the same structure as the one satisfied
by $D^{W^{1}}_{r} Y_t^{\eta}$, but with $D^{W^{1}}_{r}
X_{s}^{\varepsilon}$ replaced by $D^{W^{2}}_{r} X_{s}^{\varepsilon}$. A
calculation analogous to
\eqref{analogouscalcusedintheW2W2case1} in the proof of
Proposition \ref{proponEintDW1YDW2Y} will yield
\begin{align*}
& \mathbb{E}\left(\left(\int_{r_1\vee r_2}^{t}  \abs{Q^{\eta}_{r_1,2}(s) Q^{\eta}_{r_2,1}(s)}^{q}ds\right)^{\frac{2p}{q}}\right)\leq C \left(\varepsilon^{p}+\eta^{p}\right)\eta^{\frac{p(2-q)}{q}},
\end{align*}
\begin{align*}
& \mathbb{E}\left(\left(\int_{r_1\vee r_2}^{t}  \abs{Q^{\eta}_{r_1,1}(s) Q^{\eta}_{r_2,2}(s)}^{q}ds\right)^{\frac{2p}{q}}\right)\leq C \left(\varepsilon^{p}+\eta^{p}\right)\eta^{\frac{p(2-q)}{q}},
\end{align*}
as well as
\begin{align*}
&\mathbb{E}\left(\left(\int_{r_1\vee r_2}^{t}  \abs{Q^{\eta}_{r_1,2}(s)Q^{\eta}_{r_2,2}(s)}^{q}ds\right)^{\frac{2p}{q}}\right)
\leq C \left(\varepsilon^{2p}+\eta^{2p}\right).
\end{align*}
Hence, it remains to address the first term on the right-hand side of
\eqref{4termdecompofEintDr1W2YDr2W2Ytotheqdstothe2poverq}. For this
purpose, we recall that  $Q^{\eta}_{r,1}(t)= \frac{1}{\sqrt{\eta}}
Z_{r,2}(t)  \tau\left(X_{r}^{\varepsilon},Y_{r}^{\eta} \right)$, with
$\tau$ assumed to be uniformly bounded. Hence, we can write
\begin{align*}
\mathbb{E}\left(\left(\int_{r_1\vee r_2}^{t}  \abs{Q^{\eta}_{r_1,1}(s)Q^{\eta}_{r_2,1}(s)}^{q}ds\right)^{\frac{2p}{q}}\right) &\leq \frac{C}{\eta^{2p}} \mathbb{E}\left(\left(   \int_{r_1 \vee r_2}^t
                 Z_{r_1,2}^q(s)Z_{r_2,2}^q(s)
                                       ds \right)^{\frac{2p}{q}}\right).
\end{align*}
Using Proposition \ref{lemmaonEintZr12Zr22q} then implies that
\begin{align*}
\mathbb{E}\left(\left(\int_{r_1\vee r_2}^{t}
  \abs{Q^{\eta}_{r_1,1}(s)Q^{\eta}_{r_2,1}(s)}^{q}ds\right)^{\frac{2p}{q}}\right)
  &\leq C \eta^{\frac{2p(1-q)}{q}}e^{-\frac{K}{2\eta}(r_1 \vee r_2 - r_1
    \wedge r_2)}.
\end{align*}
Combining these four estimates concludes the proof.
\end{proof}

\begin{prop}
  \label{proponEintDW2W2Ystotheq}
Let $p\geq 1$ be a  natural number and $q \in \left\{ 1,2 \right\}$. Then, it holds that for some
constant $C >0$,
\begin{align*}
\mathbb{E}\left(\left(\int_{r_1 \vee r_2}^{t}\abs{D_{r_1,r_2}^{W^2,W^2}Y_s^{\eta}}^{q}ds\right)^{\frac{2p}{q}}\right) &\leq C\Bigg[\left(\varepsilon^{p}+\eta^{p}+\frac{1}{\eta^{p}}e^{-\frac{K}{\eta}(r_1\vee
                                                          r_2-r_1\wedge r_2)}  \right)\eta^{\frac{p(2-q)}{q}}\\
  &\quad +  \eta^{\frac{2p(1-q)}{q}}e^{-\frac{K}{2\eta}(r_1
    \vee r_2 - r_1 \wedge r_2)} +\left( \varepsilon^p + \eta^p
    \right)\left(\eta^{\frac{p(2-q)}{q}} +  \varepsilon^p + \eta^p
    \right)\\
  &\quad +  \mathbb{E}\left(\int_{r_1\vee r_2}^{t}
    \abs{D^{W_2,W_2}_{r_1,r_2}X^{\varepsilon}_{s}}^{2p}ds\right)\Bigg].
\end{align*}

\end{prop}
\begin{proof}
Recall from \eqref{equationsatisfiedbyDWWY} that the process
$D^{W^{2},W^{2}}_{r_1,r_{2}} Y_t^{\eta}$ satisfies an affine
stochastic differential equation. We can hence write
\begin{align*}
D_{r_1,r_2}^{W^{2},W^{2}}Y_t^{\eta} &= \frac{1}{\sqrt{\eta}}Z_{r_1
  \vee
  r_2,2}(t)\Big(\partial_1\tau (X_{r_1\vee r_2}^{\varepsilon},Y_{r_1\vee
                 r_2}^{\eta}) D_{r_1\wedge
                 r_2}^{W_2}X^{\varepsilon}_{r_1\vee r_2}\\
  &\qquad\qquad\qquad\qquad\qquad\qquad\qquad\qquad\qquad\qquad +\partial_2\tau (X_{r_1\vee r_2}^{\varepsilon},Y_{r_1\vee r_2}^{\eta}) D_{r_1\wedge r_2}^{W_2}Y^{\eta}_{r_1\vee r_2}\Big)
\\
  &\quad + \frac{1}{\eta}Z_{r_1
  \vee
  r_2,2}(t) \int_{r_1 \vee r_2}^t Z_{r_1
  \vee
  r_2,2}^{-1}(s) \Big[b_2^{2,2}[f]\left(X_s^{\varepsilon},Y_s^{\eta}
    \right)\\
  &\qquad\qquad\qquad\qquad\qquad\qquad\qquad\qquad\qquad\qquad - \partial_2\tau
    \left(X_s^{\varepsilon},Y_s^{\eta}
    \right)b_2^{2,2}[\tau]\left(X_s^{\varepsilon},Y_s^{\eta}
    \right)  \Big]ds\\
  &\quad + \frac{1}{\sqrt{\eta}}Z_{r_1
  \vee
  r_2,2}(t) \int_{r_1 \vee r_2}^t Z_{r_1
  \vee
  r_2,2}^{-1}(s)b_2^{2,2}[\tau]\left(X_s^{\varepsilon},Y_s^{\eta}
    \right)dW_s^2\nonumber\\
  &= \tilde{Q}_{r_1\vee r_2,1}^{\eta}(t)+\tilde{Q}_{r_1\vee r_2,2}^{\eta}(t)+\tilde{Q}_{r_1\vee r_2,3}^{\eta}(t),
\end{align*}
where
\begin{align*}
\tilde{Q}_{r_1\vee r_2,1}^{\eta}(t)&=\frac{1}{\sqrt{\eta}}Z_{r_1
  \vee
  r_2,2}(t)\Big(\partial_1\tau (X_{r_1\vee r_2}^{\varepsilon},Y_{r_1\vee
                 r_2}^{\eta}) D_{r_1\wedge
                 r_2}^{W_2}X^{\varepsilon}_{r_1\vee r_2}\\
  &\qquad\qquad\qquad\qquad\qquad\qquad\qquad\qquad\qquad\qquad +\partial_2\tau (X_{r_1\vee r_2}^{\varepsilon},Y_{r_1\vee r_2}^{\eta}) D_{r_1\wedge r_2}^{W_2}Y^{\eta}_{r_1\vee r_2}\Big),
\end{align*}
\begin{align}
  \label{defQtilde2}
 \tilde{Q}_{r_1\vee r_2,2}^{\eta}(t)&=\frac{1}{\eta}Z_{r_1
  \vee
  r_2,2}(t) \int_{r_1 \vee r_2}^t Z_{r_1
  \vee
  r_2,2}^{-1}(s) \Big[b_2^{2,2}[f]\left(X_s^{\varepsilon},Y_s^{\eta}
    \right)-(\partial_{1}f\left(X_s^{\varepsilon},Y_s^{\eta}
                                      \right)\nonumber\\
  &\qquad\qquad -\partial_{2}\tau\left(X_s^{\varepsilon},Y_s^{\eta}
    \right)\partial_{1}\tau\left(X_s^{\varepsilon},Y_s^{\eta}
    \right)) D_{r_1,r_2}^{W^{2},W^{2}}X^{\varepsilon}_{s} - \partial_2\tau\left(X_s^{\varepsilon},Y_s^{\eta}
    \right)
    b_2^{2,2}[\tau]\left(X_s^{\varepsilon},Y_s^{\eta}
    \right)  \Big]ds\nonumber\\
  &\quad + \frac{1}{\sqrt{\eta}}Z_{r_1
  \vee
  r_2,2}(t) \int_{r_1 \vee r_2}^t Z_{r_1
  \vee
  r_2,2}^{-1}(s)\Big(b_2^{2,2}[\tau]\left(X_s^{\varepsilon},Y_s^{\eta}
    \right)\nonumber\\
  &\qquad\qquad\qquad\qquad\qquad\qquad\qquad\qquad\qquad\qquad -\partial_{1}f\left(X_s^{\varepsilon},Y_s^{\eta}
    \Big)
    D_{r_1,r_2}^{W^{2},W^{2}}X^{\varepsilon}_{s}\right)dW_s^2
\end{align}
and
\begin{align*}
\tilde{Q}_{r_1\vee r_2,3}^{\eta}(t)&=\frac{1}{\eta}Z_{r_1
  \vee
  r_2,2}(t) \int_{r_1 \vee r_2}^t Z_{r_1
  \vee
  r_2,2}^{-1}(s) (\partial_{1}f\left(X_s^{\varepsilon},Y_s^{\eta}
                                     \right)\\
  &\qquad\qquad\qquad\qquad\qquad\qquad\qquad\qquad -\partial_{2}\tau\left(X_s^{\varepsilon},Y_s^{\eta}
    \right)\partial_{1}\tau\left(X_s^{\varepsilon},Y_s^{\eta}
    \right)) D_{r_1,r_2}^{W^{2},W^{2}}X^{\varepsilon}_{s}  ds\nonumber\\
  &\quad + \frac{1}{\sqrt{\eta}}Z_{r_1
  \vee
  r_2,2}(t) \int_{r_1 \vee r_2}^t Z_{r_1
  \vee
  r_2,2}^{-1}(s)\partial_{1}f \left(X_s^{\varepsilon},Y_s^{\eta}
    \right) D_{r_1,r_2}^{W^{2},W^{2}}X^{\varepsilon}_{s}dW_s^2.
\end{align*}
We can hence write
\begin{align}
  \label{tripledecompofEintDW2W2Y}
\mathbb{E}\left(\left(\int_{r_1 \vee
  r_2}^{t}\abs{D_{r_1,r_2}^{W^2,W^2}Y_s^{\eta}}^{q}ds\right)^{\frac{2p}{q}}\right)
  \leq C \sum_{i=1}^{3}\mathbb{E}\left(\left(\int_{r_1 \vee r_2}^{t}\abs{\tilde{Q}_{r_1\vee r_2,i}^{\eta}(s)}^{q}ds\right)^{\frac{2p}{q}}\right)
\end{align}
and estimate these three terms separately. For the term corresponding
to $i=1$ in \eqref{tripledecompofEintDW2W2Y}, we can apply Proposition \ref{propinitialconditionDY} together with Propositions \ref{boundonEDW2Xs2p} and
\ref{boundonDY2pintermsofepsandeta} in order to get
\begin{align}
  \label{Eq:Bound_hatQ1_W2W2}
\mathbb{E}\left(\left(\int_{r_1\vee r_2}^{t}  \abs{\tilde{Q}_{r_1\vee r_2,1}^{\eta}(s)}^{q}ds\right)^{\frac{2p}{q}}\right)&\leq C\left(\varepsilon^{p}+\eta^{p}+\frac{1}{\eta^{p}}e^{-\frac{K}{\eta}(r_1\vee r_2-r_1\wedge r_2)}  \right)\eta^{\frac{p(2-q)}{q}}.
\end{align}
For the term corresponding
to $i=3$ in \eqref{tripledecompofEintDW2W2Y}, note that $\tilde{Q}_{r_1\vee r_2,3}^{\eta}(t)$ satisfies a
stochastic differential equation of the form
\begin{align*}
\tilde{Q}_{r_1\vee r_2,3}^{\eta}(t)&=\frac{1}{\eta}\int_{r_1\vee r_2}^{t}\left[\partial_{2}f(X^{\varepsilon}_{s},Y^{\eta}_{s})\tilde{Q}_{r_1\vee r_2,3}^{\eta}(s)+\partial_{1}f(X^{\varepsilon}_{s},Y^{\eta}_{s}) D^{W_2,W_2}_{r_1,r_2}X^{\varepsilon}_{s}\right]ds\nonumber\\
&\quad +\frac{1}{\sqrt{\eta}}\int_{r_1\vee r_2}^{t}\left[\partial_{2}\tau(X^{\varepsilon}_{s},Y^{\eta}_{s})\tilde{Q}_{r_1\vee r_2,3}^{\eta}(s)+\partial_{1}\tau(X^{\varepsilon}_{s},Y^{\eta}_{s}) D^{W_2,W_2}_{r_1,r_2}X^{\varepsilon}_{s}\right]dW_{s}^{2}.
\end{align*}
The same methodology that was used for bounding
$D^{W_1,W_1}_{r_1,r_2}Y^{\eta}_{t}$ in terms of
$D^{W_1,W_1}_{r_1,r_2}X^{\varepsilon}_{t}$ in Proposition \ref{proponEintDW1W1YslessEintDW1W1Xs} yields
\begin{align}
  \label{Eq:Bound_hatQ3_W2W2}
\mathbb{E}\left(\left(\int_{r_1\vee r_2}^{t}  \abs{\tilde{Q}_{r_1\vee r_2,3}^{\eta}(s)}^{q}ds\right)^{\frac{2p}{q}}\right)&\leq C \mathbb{E}\left(\int_{r_1\vee r_2}^{t}  \abs{\tilde{Q}_{r_1\vee r_2,3}^{\eta}(s)}^{2p}ds\right)\nonumber\\
&\leq C \mathbb{E}\left(\int_{r_1\vee r_2}^{t}  \abs{D^{W_2,W_2}_{r_1,r_2}X^{\varepsilon}_{s}}^{2p}ds\right).
\end{align}
For the term corresponding to $i=2$ in
\eqref{tripledecompofEintDW2W2Y}, inspecting the structure of
$\tilde{Q}_{r_1\vee r_2,2}^{\eta}(t)$ given in \eqref{defQtilde2}
shows that it is enough to estimate terms of the form
\begin{align*}
&\mathbb{E}\left(\left(\int_{r_1 \vee r_2}^t
  \abs{\frac{1}{\eta}Z_{r_1
  \vee
  r_2,2}(s)\int_{r_1 \vee r_2}^s Z_{r_1
  \vee
  r_2,2}^{-1}(u)F_u G_{u} du }^q ds \right)^{\frac{2p}{q}}
                 \right)
\end{align*}
and
\begin{align*}
&\mathbb{E}\left(\left(\int_{r_1 \vee r_2}^t
  \abs{\frac{1}{\sqrt{\eta}}Z_{r_1
  \vee
  r_2,2}(s)\int_{r_1 \vee r_2}^s Z_{r_1
  \vee
  r_2,2}^{-1}(u)F_u G_{u} dW^{2}_{u} }^q ds \right)^{\frac{2p}{q}}
                 \right),
\end{align*}
where the product $F_u G_{u}$ is any of the terms $D_{r_1}^{W^{2}}X_{u}^{\varepsilon}D_{r_2}^{W^{2}}X_{u}^{\varepsilon}$,
$D_{r_1}^{W^{2}}X_{u}^{\varepsilon}D_{r_2}^{W^{2}}Y_{u}^{\eta}$,
$D_{r_1}^{W^{2}}Y_{u}^{\eta}D_{r_2}^{W^{2}}X_{u}^{\varepsilon}$ or
$D_{r_1}^{W^{2}}Y_{u}^{\eta}D_{r_2}^{W^{2}}Y_{u}^{\eta}$. The analysis
in the cases where $F_u G_u$ is equal to either $D_{r_1}^{W^{2}}X_{u}^{\varepsilon}D_{r_2}^{W^{2}}X_{u}^{\varepsilon}$,
$D_{r_1}^{W^{2}}X_{u}^{\varepsilon}D_{r_2}^{W^{2}}Y_{u}^{\eta}$ or $D_{r_1}^{W^{2}}Y_{u}^{\eta}D_{r_2}^{W^{2}}X_{u}^{\varepsilon}$ is the
same as in Proposition \ref{proponEintDW1W2Ystotheq} and its proof and
yields
\begin{align*}
&\mathbb{E}\left(\left(\int_{r_1 \vee r_2}^t
  \abs{\frac{1}{\eta}Z_{r_1
  \vee
  r_2,2}(s)\int_{r_1 \vee r_2}^s Z_{r_1
  \vee
  r_2,2}^{-1}(u)D_{r_1}^{W^{2}}X_{u}^{\varepsilon}D_{r_2}^{W^{2}}X_{u}^{\varepsilon} du}^q ds \right)^{\frac{2p}{q}}
                 \right)\leq C \left(\varepsilon^{2p}+\eta^{2p}\right),\\
&\mathbb{E}\left(\left(\int_{r_1 \vee r_2}^t
  \abs{\frac{1}{\sqrt{\eta}}Z_{r_1
  \vee
  r_2,2}(s)\int_{r_1 \vee r_2}^s Z_{r_1
  \vee
  r_2,2}^{-1}(u)D_{r_1}^{W^{2}}X_{u}^{\varepsilon}D_{r_2}^{W^{2}}X_{u}^{\varepsilon} dW^{2}_u}^q ds \right)^{\frac{2p}{q}}
                 \right)\leq C\left(\varepsilon^{2p}+\eta^{2p}\right)
\end{align*}
for the case where $F_u G_u =
D_{r_1}^{W^{2}}X_{u}^{\varepsilon}D_{r_2}^{W^{2}}X_{u}^{\varepsilon}$,
\begin{align*}
&\mathbb{E}\left(\left(\int_{r_1 \vee r_2}^t
  \abs{\frac{1}{\eta}Z_{r_1
  \vee
  r_2,2}(s)\int_{r_1 \vee r_2}^s Z_{r_1
  \vee
  r_2,2}^{-1}(u)D_{r_1}^{W^{2}}X_{u}^{\varepsilon}D_{r_2}^{W^{2}}Y_{u}^{\eta}du}^q ds \right)^{\frac{2p}{q}}
                 \right)\\
  & \qquad\qquad\qquad\qquad\qquad\qquad\qquad\qquad\qquad\qquad\qquad\qquad \leq C \left( \varepsilon^p + \eta^p \right)\left(\eta^{\frac{p(2-q)}{q}} +  \varepsilon^p + \eta^p \right)\\
&\mathbb{E}\left(\left(\int_{r_1 \vee r_2}^t
  \abs{\frac{1}{\sqrt{\eta}}Z_{r_1
  \vee
  r_2,2}(s)\int_{r_1 \vee r_2}^s Z_{r_1
  \vee
  r_2,2}^{-1}(u)D_{r_1}^{W^{2}}X_{u}^{\varepsilon}D_{r_2}^{W^{2}}Y_{u}^{\eta}dW^2_u}^q ds \right)^{\frac{2p}{q}}
                 \right)\\
  & \qquad\qquad\qquad\qquad\qquad\qquad\qquad\qquad\qquad\qquad\qquad\qquad \leq C \left( \varepsilon^p + \eta^p \right)\left(\eta^{\frac{p(2-q)}{q}} +  \varepsilon^p + \eta^p \right)
\end{align*}
for the case where $F_u G_u = D_{r_1}^{W^{2}}X_{u}^{\varepsilon}D_{r_2}^{W^{2}}Y_{u}^{\eta}$, and
\begin{align*}
&\mathbb{E}\left(\left(\int_{r_1 \vee r_2}^t
  \abs{\frac{1}{\eta}Z_{r_1
  \vee
  r_2,2}(s)\int_{r_1 \vee r_2}^s Z_{r_1
  \vee
  r_2,2}^{-1}(u)D_{r_1}^{W^{2}}Y_{u}^{\eta}D_{r_2}^{W^{2}}X_{u}^{\varepsilon}du}^q ds \right)^{\frac{2p}{q}}
                 \right)\\
  & \qquad\qquad\qquad\qquad\qquad\qquad\qquad\qquad\qquad\qquad\qquad\qquad \leq C \left( \varepsilon^p + \eta^p \right)\left(\eta^{\frac{p(2-q)}{q}} +  \varepsilon^p + \eta^p \right)\\
&\mathbb{E}\left(\left(\int_{r_1 \vee r_2}^t
  \abs{\frac{1}{\sqrt{\eta}}Z_{r_1
  \vee
  r_2,2}(s)\int_{r_1 \vee r_2}^s Z_{r_1
  \vee
  r_2,2}^{-1}(u)D_{r_1}^{W^{2}}Y_{u}^{\eta}D_{r_2}^{W^{2}}X_{u}^{\varepsilon}dW^2_u}^q ds \right)^{\frac{2p}{q}}
                 \right)\\
  & \qquad\qquad\qquad\qquad\qquad\qquad\qquad\qquad\qquad\qquad\qquad\qquad \leq C \left( \varepsilon^p + \eta^p \right)\left(\eta^{\frac{p(2-q)}{q}} +  \varepsilon^p + \eta^p \right)
\end{align*}
for the case where $F_u G_u = D_{r_1}^{W^{2}}Y_{u}^{\eta}D_{r_2}^{W^{2}}X_{u}^{\varepsilon}$.
It remains to examine the case where $F_uG_u=D_{r_1}^{W^{2}}Y_{u}^{\eta}D_{r_2}^{W^{2}}Y_{u}^{\eta}$.
We recall that we have the decomposition \eqref{decompositionofDW2Yt} given by
$D_r^{W^2} Y_u^{\eta}=Q^{\eta}_{r,1}(u)+Q^{\eta}_{r,2}(u)$. Hence, we write
\begin{align*}
D_{r_1}^{W^{2}}Y_{u}^{\eta}D_{r_2}^{W^{2}}Y_{u}^{\eta}&=Q^{\eta}_{r_1,1}(u)Q^{\eta}_{r_2,1}(u)
+Q^{\eta}_{r_1,1}(u)Q^{\eta}_{r_2,2}(u)+Q^{\eta}_{r_1,2}(u)Q^{\eta}_{r_2,1}(u)+
Q^{\eta}_{r_1,2}(u)Q^{\eta}_{r_2,2}(u).
\end{align*}
As pointed out in \eqref{Eq:SDE_Q2D2Y}, $Q^{\eta}_{r,2}(t)$ satisfies
a stochastic differential equation which has the same structure as the one satisfied
by $D^{W^{1}}_{r} Y_t^{\eta}$, but with
$D_{r}^{W^{1}}X_{s}^{\varepsilon}$ replaced by
$D_{r}^{W^{2}}X_{s}^{\varepsilon}$, so that the same analysis as in
Proposition \ref{proponEintDW1W2Ystotheq} and its proof gives us
\begin{align*}
&\mathbb{E}\left(\left(\int_{r_1 \vee r_2}^t
  \abs{\frac{1}{\eta}Z_{r_1
  \vee
  r_2,2}(s)\int_{r_1 \vee r_2}^s Z_{r_1
  \vee
  r_2,2}^{-1}(u)Q^{\eta}_{r_1,2}(u)Q^{\eta}_{r_2,2}(u)du}^q ds \right)^{\frac{2p}{q}}
                 \right)\leq C \left(\varepsilon^{2p}+\eta^{2p}\right),\\
&\mathbb{E}\left(\left(\int_{r_1 \vee r_2}^t
  \abs{\frac{1}{\sqrt{\eta}}Z_{r_1
  \vee
  r_2,2}(s)\int_{r_1 \vee r_2}^s Z_{r_1
  \vee
  r_2,2}^{-1}(u)Q^{\eta}_{r_1,2}(u)Q^{\eta}_{r_2,2}(u) dW^{2}_u}^q ds \right)^{\frac{2p}{q}}
                 \right)\leq C\left(\varepsilon^{2p}+\eta^{2p}\right)
\end{align*}
for the case where $F_u G_u =
Q^{\eta}_{r_1,2}(u)Q^{\eta}_{r_2,2}(u)$, and
\begin{align*}
&\mathbb{E}\left(\left(\int_{r_1 \vee r_2}^t
  \abs{\frac{1}{\eta}Z_{r_1
  \vee
  r_2,2}(s)\int_{r_1 \vee r_2}^s Z_{r_1
  \vee
  r_2,2}^{-1}(u)Q^{\eta}_{r_1,1}(u)Q^{\eta}_{r_2,2}(u)du}^q ds \right)^{\frac{2p}{q}}
                 \right)\leq C(\varepsilon^{p}+\eta^{p})\eta^{\frac{p(2-q)}{q}},\\
&\mathbb{E}\left(\left(\int_{r_1 \vee r_2}^t
  \abs{\frac{1}{\sqrt{\eta}}Z_{r_1
  \vee
  r_2,2}(s)\int_{r_1 \vee r_2}^s Z_{r_1
  \vee
  r_2,2}^{-1}(u)Q^{\eta}_{r_1,1}(u)Q^{\eta}_{r_2,2}(u) dW^{2}_u}^q ds \right)^{\frac{2p}{q}}
                 \right)\leq C(\varepsilon^{p}+\eta^{p})\eta^{\frac{p(2-q)}{q}}
\end{align*}
for the cases where $F_u G_u =
Q^{\eta}_{r_1,1}(u)Q^{\eta}_{r_2,2}(u)$ and $F_u G_u =
Q^{\eta}_{r_1,2}(u)Q^{\eta}_{r_2,1}(u)$. Hence, the only new case that
we still have to handle is the case where $F_u
G_u=Q^{\eta}_{r_1,1}(u)Q^{\eta}_{r_2,1}(u)$. Recall that
$Q^{\eta}_{r,1}(t)= \frac{1}{\sqrt{\eta}} Z_{r,2}(t)
\tau\left(X_{r}^{\varepsilon},Y_{r}^{\eta} \right)$ with $\tau$ assumed
to be uniformly bounded, so that
\begin{align*}
&\mathbb{E}\left(\left(\int_{r_1 \vee r_2}^t
  \abs{\frac{1}{\eta}Z_{r_1
  \vee
  r_2,2}(s)\int_{r_1 \vee r_2}^s Z_{r_1
  \vee
  r_2,2}^{-1}(u)Q^{\eta}_{r_1,1}(u)Q^{\eta}_{r_2,1}(u)du}^q ds \right)^{\frac{2p}{q}}
                 \right)\nonumber\\
  & \qquad\qquad\qquad\qquad \leq \frac{1}{\eta^{4p}} \mathbb{E}\left(\left(   \int_{r_1 \vee r_2}^t
                Z_{r_1 \vee r_2,2}^q(s)\abs{ \int_{r_1 \vee r_2}^s
    Z_{r_1 \wedge r_2,2}(u) du}^q
                 ds \right)^{\frac{2p}{q}}\right)\\
    & \qquad\qquad\qquad\qquad \leq \frac{1}{\eta^{4p}} \mathbb{E}\left(\abs{\int_{r_1 \vee r_2}^t
                Z_{r_1 \vee r_2,2}^q(s)ds}^{\frac{2p}{q}}\abs{ \int_{r_1 \vee r_2}^t
    Z_{r_1 \wedge r_2,2}(s) ds}^{2p}
                  \right)\\
   & \qquad\qquad\qquad\qquad \leq \frac{1}{\eta^{4p}} \mathbb{E}\left(\abs{\int_{r_1 \vee r_2}^t
                Z_{r_1 \vee r_2,2}^q(s)ds}^{\frac{4p}{q}}\right)^{1/2}\mathbb{E}\left(\abs{ \int_{r_1 \vee r_2}^t
    Z_{r_1 \wedge r_2,2}(s) ds}^{4p}
                  \right)^{1/2}\\
   & \qquad\qquad\qquad\qquad \leq C \eta^{\frac{2p(1 -q)}{q}}e^{-\frac{K}{2\eta}(r_1
    \vee r_2 - r_1 \wedge r_2)},
\end{align*}
where the last bound was obtained by applying Propositions
\ref{propinitialconditionDY} and
\ref{propintZr1minr2tosomepower}. The same arguments (with the
addition of the Burkholder-Davis-Gundy inequality) also yield
\begin{align*}
&\mathbb{E}\left(\left(\int_{r_1 \vee r_2}^t
  \abs{\frac{1}{\sqrt{\eta}}Z_{r_1
  \vee
  r_2,2}(s)\int_{r_1 \vee r_2}^s Z_{r_1
  \vee
  r_2,2}^{-1}(u)Q^{\eta}_{r_1,1}(u)Q^{\eta}_{r_2,1}(u)du}^q ds \right)^{\frac{2p}{q}}
                 \right)\nonumber\\
   &
     \qquad\qquad\qquad\qquad\qquad\qquad\qquad\qquad\qquad\qquad\qquad\qquad\qquad \leq C \eta^{\frac{2p(1 -q)}{q}}e^{-\frac{K}{2\eta}(r_1
    \vee r_2 - r_1 \wedge r_2)}.
\end{align*}
Combining the previous estimates, we have established
\begin{align}
  \label{Eq:Bound_hatQ2_W2W2}
\mathbb{E}\left(\left(\int_{r_1\vee r_2}^{t}  \abs{\tilde{Q}_{r_1\vee
  r_2,2}^{\eta}(s)}^{q}ds\right)^{\frac{2p}{q}}\right)&\leq C \bigg[   \eta^{\frac{2p(1-q)}{q}}e^{-\frac{K}{2\eta}(r_1
    \vee r_2 - r_1 \wedge r_2)}\nonumber\\
    &\qquad\qquad\qquad\qquad\qquad +\left( \varepsilon^p + \eta^p \right)\left(\eta^{\frac{p(2-q)}{q}} +  \varepsilon^p + \eta^p \right)\bigg].
\end{align}
Finally putting together the estimates \eqref{Eq:Bound_hatQ1_W2W2},
\eqref{Eq:Bound_hatQ3_W2W2} and \eqref{Eq:Bound_hatQ2_W2W2}, we can write
\begin{align*}
\mathbb{E}\left(\left(\int_{r_1 \vee r_2}^{t}\abs{D_{r_1,r_2}^{W^2,W^2}Y_s^{\eta}}^{q}ds\right)^{\frac{2p}{q}}\right) &\leq C\Bigg[\left(\varepsilon^{p}+\eta^{p}+\frac{1}{\eta^{p}}e^{-\frac{K}{\eta}(r_1\vee
                                                          r_2-r_1\wedge r_2)}  \right)\eta^{\frac{p(2-q)}{q}}\\
  &\quad +  \eta^{\frac{2p(1-q)}{q}}e^{-\frac{K}{2\eta}(r_1
    \vee r_2 - r_1 \wedge r_2)} +\left( \varepsilon^p + \eta^p
    \right)\left(\eta^{\frac{p(2-q)}{q}} +  \varepsilon^p + \eta^p
    \right)\\
  &\quad +  \mathbb{E}\left(\int_{r_1\vee r_2}^{t}
    \abs{D^{W_2,W_2}_{r_1,r_2}X^{\varepsilon}_{s}}^{2p}ds\right)\Bigg],
\end{align*}
which concludes the proof.
\end{proof}

\section{Ancillary results}\label{S:AuxiliaryBounds}
We continue in this section to
work with Assumption \ref{A:Assumptiongeneral} as in the previous
section (for the same generality and independent interest reasons).
\begin{lemma}
  \label{lemmaonExpofZr2tothe2pq}
Let $0 \leq r \leq t \leq 1$ and $p \geq 1$. Under Assumption \ref{A:Assumptiongeneral}, it
holds that
\begin{equation*}
\mathbb{E}\left(\abs{Z_{r,2}(t)}^{2p} \right) \leq e^{-\frac{K}{\eta}(t-r)}.
\end{equation*}
\end{lemma}
\begin{proof}
Using the expression for $Z_{r,2}(t)$ given in \eqref{expressionofZr2t}, one can write
\begin{align*}
\mathbb{E}\left(\abs{Z_{r,2}(t)}^{2p} \right) &= \mathbb{E}\left( e^{\frac{2p}{\eta}\int_r^t \partial_2 f\left(X_s^{\varepsilon},Y_s^{\eta}
  \right)ds+ \frac{2p}{\sqrt{\eta}} \int_r^t \partial_2 \tau\left(X_s^{\varepsilon},Y_s^{\eta}
  \right)dW_s^2 -\frac{p}{\eta}\int_r^t \partial_2 \tau\left(X_s^{\varepsilon},Y_s^{\eta}
  \right)^2 ds}\right) \\
  &= \mathbb{E}\bigg( e^{\frac{2p}{\sqrt{\eta}}\int_r^t  \partial_2 \tau\left(X_s^{\varepsilon},Y_s^{\eta}
  \right)dW_s^2 -\frac{2p^2}{\eta}\int_r^t \partial_2 \tau\left(X_s^{\varepsilon},Y_s^{\eta}
    \right)^2ds}\\
  &\qquad\qquad\qquad\qquad\qquad\qquad e^{\frac{1}{\eta}\int_r^t \left[ 2p \partial_2f\left(X_s^{\varepsilon},Y_s^{\eta}
  \right)+  p \left( 2p - 1 \right)\partial_2 \tau\left(X_s^{\varepsilon},Y_s^{\eta}
    \right)^2 \right]ds}\bigg).
\end{align*}
Using Assumption \ref{A:Assumptiongeneral}, we have that
$\sup_{r\leq s \leq t}\left\{ 2p \partial_2f\left(X_s^{\varepsilon},Y_s^{\eta}
  \right)+  p \left( 2p - 1 \right)\partial_2 \tau\left(X_s^{\varepsilon},Y_s^{\eta}
    \right)^2  \right\} \leq -K <0$, so that
\begin{align*}
\mathbb{E}\left(\abs{Z_{r,2}(t)}^{2p} \right)  &\leq  e^{-\frac{K}{\eta}(t-r)}\mathbb{E}\left(e^{\frac{2p}{\sqrt{\eta}}\int_r^t  \partial_2 \tau\left(X_s^{\varepsilon},Y_s^{\eta}
  \right)dW_s^2 -\frac{2p^2}{\eta}\int_r^t \partial_2 \tau\left(X_s^{\varepsilon},Y_s^{\eta}
    \right)^2ds}\right)\\
  &= e^{-\frac{K}{\eta}(t-r)}
\end{align*}
as $\left\{ e^{\frac{2p}{\sqrt{\eta}}\int_r^t  \partial_2 \tau\left(X_s^{\varepsilon},Y_s^{\eta}
  \right)dW_s^2 -\frac{2p^2}{\eta}\int_r^t \partial_2 \tau\left(X_s^{\varepsilon},Y_s^{\eta}
    \right)^2ds} \colon r \leq t \leq 1 \right\}$ is a martingale.
\end{proof}

\begin{prop}
  \label{propinitialconditionDYnointegration}
Let $r \leq t \leq 1$, $p \geq 1$ be an integer and $\left\{ Z_{r,2}(t)\colon r \leq t \leq 1
\right\}$ be the stochastic process defined by
\eqref{expressionofZr2t}. For any $\mathcal{F}_r$-measurable random
variable $W$ in $L^{2p} \left( \Omega \right)$, one has
\begin{align*}
\mathbb{E}\left(\abs{\frac{1}{\sqrt{\eta}}Z_{r,2}(t)W}^{2p}
  \right)\leq \mathbb{E}\left(
  \abs{\frac{1}{\sqrt{\eta}}W}^{2p}
                 \right)e^{-\frac{K}{\eta}(t-r)}.
\end{align*}
\end{prop}
\begin{proof}
Using the It\^o product formula and the stochastic differential
equation \eqref{sdeforZr2} satisfied by $Z_{r,2}(t)$, we can write
\begin{align*}
\frac{1}{\sqrt{\eta}}Z_{r,2}(t)W &= \frac{1}{\sqrt{\eta}}W +
  \frac{1}{\eta}\int_r^t \frac{1}{\sqrt{\eta}}Z_{r,2}(s)W \partial_2f\left(X_s^{\varepsilon},Y_s^{\eta}
                                   \right)ds  \\
  &\quad +
  \frac{1}{\sqrt{\eta}}\int_r^t \frac{1}{\sqrt{\eta}}Z_{r,2}(s)W
  \partial_2 \tau\left(X_s^{\varepsilon},Y_s^{\eta}
    \right)dW_s^2.
\end{align*}
Hence, applying the It\^o formula with the function $x \mapsto x^{2p}$
and taking expectation yields
\begin{align*}
&\mathbb{E}\left(
  \abs{\frac{1}{\sqrt{\eta}}Z_{r,2}(t)W}^{2p}
  \right) = \mathbb{E}\left(\abs{\frac{1}{\sqrt{\eta}}W}^{2p}
 \right) \\
  &\qquad\quad +
  \frac{1}{\eta}\mathbb{E}\left( \int_r^t \abs{\frac{1}{\sqrt{\eta}}Z_{r,2}(s)W}^{2p}\left[ 2p \partial_2f\left(X_s^{\varepsilon},Y_s^{\eta}
                                   \right) + p(2p-1)\abs{\partial_2\tau\left(X_s^{\varepsilon},Y_s^{\eta}
                                   \right)}^2 \right] ds \right).
\end{align*}
Differentiating this equality with respect to $t$ yields
\begin{align*}
& \frac{d}{dt}\mathbb{E}\left(
  \abs{\frac{1}{\sqrt{\eta}}Z_{r,2}(t)W}^{2p}
                 \right) \\
  &\qquad\qquad = \frac{1}{\eta}\mathbb{E}\left( \abs{\frac{1}{\sqrt{\eta}}Z_{r,2}(t)W}^{2p}\left[ 2p \partial_2f\left(X_t^{\varepsilon},Y_t^{\eta}
                                   \right) + p(2p-1)\partial_2\tau\left(X_t^{\varepsilon},Y_t^{\eta}
                                   \right)^2 \right]  \right).
\end{align*}
Using Assumption
\ref{A:Assumptiongeneral}, we get the differential inequality
\begin{align*}
\frac{d}{dt}\mathbb{E}\left(
  \abs{\frac{1}{\sqrt{\eta}}Z_{r,2}(s)W}^{2p}
                 \right) \leq -\frac{K}{\eta}\mathbb{E}\left( \abs{
            \frac{1}{\sqrt{\eta}}Z_{r,2}(t)W}^{2p} \right),
\end{align*}
which once solved yields
\begin{align*}
\mathbb{E}\left(
  \abs{\frac{1}{\sqrt{\eta}}Z_{r,2}(t)W}^{2p}
                 \right) \leq \mathbb{E}\left(
  \abs{\frac{1}{\sqrt{\eta}}W}^{2p}
                 \right)e^{-\frac{K}{\eta}(t-r)}.
\end{align*}
\end{proof}

\begin{prop}
  \label{propinitialconditionDY}
Let $r \leq t \leq 1$, $p \geq 1$ be an integer, $q \in \left\{ 1,2 \right\}$ and $\left\{ Z_{r,2}(t)\colon r \leq t \leq 1
\right\}$ be the stochastic process defined by
\eqref{expressionofZr2t}. For any $\mathcal{F}_r$-measurable random
variable $W$ in $L^{2p} \left( \Omega \right)$, one has
\begin{align*}
\mathbb{E}\left(\left(\int_r^t \abs{\frac{1}{\sqrt{\eta}}Z_{r,2}(s)W}^q
  ds\right)^{\frac{2p}{q}} \right)\leq \eta^{\frac{p(2-q)}{q}}\mathbb{E}\left( \abs{W}^{2p} \right) .
\end{align*}
\end{prop}
\begin{proof}
Note that
\begin{align*}
\mathbb{E}\left(\left(\int_r^t \abs{\frac{1}{\sqrt{\eta}}Z_{r,2}(s)W}^q
  ds\right)^{\frac{2p}{q}} \right) = \int_r^t \cdots \int_r^t \mathbb{E}\left(
  \prod_{i=1}^{2p/q} \abs{\frac{1}{\sqrt{\eta}}Z_{r,2}(s_i)W}^q
  \right)ds_1\cdots ds_{2p/q},
\end{align*}
so that by repeated applications of H\"older's inequality and
Proposition \ref{propinitialconditionDYnointegration},
\begin{align*}
\mathbb{E}\left(\left(\int_r^t \abs{\frac{1}{\sqrt{\eta}}Z_{r,2}(s)W}^q
  ds\right)^{\frac{2p}{q}} \right) &= \int_r^t \cdots \int_r^t \prod_{i=1}^{2p/q} \mathbb{E}\left(
   \abs{\frac{1}{\sqrt{\eta}}Z_{r,2}(s_i)W}^{2p}
  \right)^{q/2p}ds_1\cdots ds_{2p/q}\\
  &\leq \frac{1}{\eta^{p}}\mathbb{E}\left( \abs{W}^{2p} \right) \int_r^t \cdots \int_r^t \prod_{i=1}^{2p/q} e^{-\frac{qK}{2p
    \eta}(s_i - r)}ds_1\cdots ds_{2p/q}\\
  & = \frac{1}{\eta^{p}}\mathbb{E}\left( \abs{W}^{2p} \right) \left( \int_r^t e^{-\frac{qK}{2p
    \eta}(s - r)} ds \right)^{\frac{2p}{q}}\\
  & \leq \eta^{\frac{p(2-q)}{q}}\mathbb{E}\left( \abs{W}^{2p} \right).
\end{align*}
\end{proof}

\begin{prop}
  \label{propintZr1minr2tosomepower}
Let $r_1,r_2 \leq t \leq 1$, $p \geq 1$ be a natural number, $q\in \left\{ 1,2 \right\}$, and $\left\{ Z_{r,2}(t)\colon r \leq t \leq 1
\right\}$ be the stochastic process defined by
\eqref{expressionofZr2t}. Then, one has
\begin{align*}
\mathbb{E}\left(\abs{\int_{r_1 \vee r_2}^t Z_{r_1 \wedge r_2,2}(s)^q
  ds}^{\frac{2p}{q}} \right)\leq C \eta^{\frac{2p}{q}} e^{-\frac{K}{
    \eta}(r_1 \vee r_2 - r_1 \wedge r_2)}.
\end{align*}
\end{prop}
\begin{proof}
Note that
\begin{align*}
\mathbb{E}\left(\abs{\int_{r_1 \vee r_2}^t Z_{r_1 \wedge r_2,2}(s)^q
  ds}^{\frac{2p}{q}} \right) = \int_{r_1 \vee r_2}^t \cdots \int_{r_1 \vee r_2}^t\mathbb{E}\left(
  \prod_{i=1}^{2p/q} Z_{r_1 \wedge r_2,2}(s_i)^q
  \right)ds_1\cdots ds_{2p/q},
\end{align*}
so that by repeated applications of H\"older's inequality and
Lemma \ref{lemmaonExpofZr2tothe2pq},
\begin{align*}
\mathbb{E}\left(\abs{\int_{r_1 \vee r_2}^t Z_{r_1 \wedge r_2,2}(s)^q
  ds}^{\frac{2p}{q}} \right) &= \int_{r_1 \vee r_2}^t \cdots \int_{r_1 \vee r_2}^t\prod_{i=1}^{2p/q}\mathbb{E}\left(
   Z_{r_1 \wedge r_2,2}(s_i)^{2p}
  \right)^{\frac{q}{2p}}ds_1\cdots ds_{2p/q}\\
  &\leq  \int_{r_1 \vee r_2}^t \cdots \int_{r_1 \vee r_2}^t \prod_{i=1}^{2p/q} e^{-\frac{qK}{
    2p\eta}(s_i - r_1 \wedge r_2)}ds_1\cdots ds_{2p/q}\\
  & =  \left( \int_{r_1 \vee r_2}^t e^{-\frac{qK}{2p
    \eta}(s - r_1 \wedge r_2)} ds \right)^{\frac{2p}{q}}\\
  & \leq C \eta^{\frac{2p}{q}} e^{-\frac{K}{
    \eta}(r_1 \vee r_2 - r_1 \wedge r_2)}.
\end{align*}
\end{proof}

\begin{prop}
  \label{lemmaonEintZr12Zr22q}
Let $p\geq 1$ be a  natural number and let $q \in \left\{ 1,2 \right\}$. Then, for any $r_1 \vee r_2 \leq t
\leq 1$, one has, for some constant $C>0$,
\begin{equation*}
\mathbb{E}\left(\left(\int_{r_1 \vee r_2}^t
    \abs{Z_{r_1,2}(s)Z_{r_2,2}(s)}^{q}ds\right)^{\frac{2p}{q}}\right)
\leq C \eta^{\frac{2p}{q}}e^{-\frac{K}{2\eta}(r_1\vee r_2 - r_1
    \wedge r_2)}.
\end{equation*}
\end{prop}
\begin{proof}
We can write, using H\"older's inequality,
\begin{align*}
\mathbb{E}\left(\left(\int_{r_1 \vee r_2}^t
    Z_{r_1,2}^q(s)Z_{r_2,2}^q(s)ds\right)^{\frac{2p}{q}}\right)
  &= \int_{r_1 \vee r_2}^t\cdots \int_{r_1 \vee r_2}^t \mathbb{E}\left(
  \prod_{i=1}^{2p/q}
    Z_{r_1,2}^q(s_i)Z_{r_2,2}^q(s_i)\right)ds_1\cdots ds_{2p/q}\\
  &\leq \int_{r_1 \vee r_2}^t\cdots \int_{r_1 \vee r_2}^t\prod_{i=1}^{2p/q} \mathbb{E}\left(
    Z_{r_1,2}^{2p}(s_i)Z_{r_2,2}^{2p}(s_i)\right)^{\frac{q}{2p}}ds_1\cdots
    ds_{2p/q}\\
  &\leq \left( \int_{r_1 \vee r_2}^t \mathbb{E}\left(
    Z_{r_1,2}^{4p}(s)\right)^{\frac{q}{4p}}\mathbb{E}\left(
    Z_{r_2,2}^{4p}(s)\right)^{\frac{q}{4p}}ds \right)^{\frac{2p}{q}}.
\end{align*}
An application of Lemma \ref{lemmaonExpofZr2tothe2pq} yields
\begin{align*}
\mathbb{E}\left(\left(\int_{r_1 \vee r_2}^t
    Z_{r_1,2}^q(s)Z_{r_2,2}^q(s)ds\right)^{\frac{2p}{q}}\right)
  &\leq \left( \int_{r_1 \vee r_2}^t e^{-\frac{qK}{4p\eta}(2s - r_1 -
    r_2)}ds \right)^{\frac{2p}{q}}\\
  &\leq C \eta^{\frac{2p}{q}}e^{-\frac{K}{2\eta}(2r_1\vee r_2 - r_1 -
    r_2)}\\
  &= C \eta^{\frac{2p}{q}}e^{-\frac{K}{2\eta}(r_1\vee r_2 - r_1
    \wedge r_2)}.
\end{align*}
\end{proof}

\begin{lemma}
  \label{lemmaonEDW1YZr22p}
Let $p\geq 1$ be a  natural number. Then, for any $r_1 \vee r_2 \leq t
\leq 1$, there exists a positive constants $C,D >0$ such that
\begin{equation*}
\mathbb{E}\left(  \abs{D^{W^{1}}_{r_1} Y_{t}^{\eta}Z_{r_2,2}(t)}^{2p}\right) \leq C\varepsilon^{p} e^{-\frac{D}{\eta}(t-r_1\vee r_2)}.
\end{equation*}
\end{lemma}
\begin{proof}
Recall that $D^{W^{1}}_{r_1} Y_{t}^{\eta}$ satisfies Equation
\eqref{geneqforDY} and note that per It\^o's formula, $Z_{r_2,2}(t)$ satisfies the affine
equation
\begin{equation}
\label{sdeforZr2}
Z_{r_2,2}(t) = 1 + \frac{1}{\eta}\int_{r_2}^tZ_{r_2,2}(s)\partial_2f\left(X_s^{\varepsilon},Y_s^{\eta}
  \right)ds + \frac{1}{\sqrt{\eta}}\int_{r_2}^tZ_{r_2,2}(s)\partial_2\tau\left(X_s^{\varepsilon},Y_s^{\eta}
  \right)dW_s^2.
\end{equation}
Using It\^o's product formula hence yields
\begin{align*}
D^{W^{1}}_{r_1} Y_{t}^{\eta}Z_{r_2,2}(t)&=D^{W^{1}}_{r_1}
                                          Y_{r_2}^{\eta}
                                          \mathds{1}_{\left\{ r_2 \geq
                                          r_1\right\}} \\
&\quad +\frac{1}{\eta}\int_{r_1\vee r_2}^{t}\Big[\Big(2\partial_2
                                                            f\left(X_s^{\varepsilon},Y_s^{\eta}\right)+\partial_2\tau\left(X_s^{\varepsilon},Y_s^{\eta}\right)^{2}\Big)D^{W^{1}}_{r_1} Y_{s}^{\eta}Z_{r_2,2}(s)\\
  &\qquad\qquad +\left(\partial_1 f\left(X_s^{\varepsilon},Y_s^{\eta}\right)+\partial_1\tau\left(X_s^{\varepsilon},Y_s^{\eta}\right) \partial_2\tau\left(X_s^{\varepsilon},Y_s^{\eta}\right)\right)D^{W^{1}}_{r_1} X_{s}^{\varepsilon}Z_{r_2,2}(s)\bigg]ds\\
&\quad +\frac{1}{\sqrt\eta}\int_{r_1\vee r_2}^{t}\Big[2\partial_2
                                                                                                                                                              \tau\left(X_s^{\varepsilon},Y_s^{\eta}\right) D^{W^{1}}_{r_1} Y_{s}^{\eta}Z_{r_2,2}(s)\\
  &\qquad\qquad\qquad\qquad\qquad\qquad\qquad\qquad  +\partial_1\tau\left(X_s^{\varepsilon},Y_s^{\eta}\right) D^{W^{1}}_{r_1} X_{s}^{\varepsilon}Z_{r_2,2}(s)\Big]dW^{2}_{s}.
\end{align*}
Let us set $\Lambda_{t}=D^{W^{1}}_{r_1}
Y_{t}^{\eta}Z_{r_2,2}(t)$. Then, It\^o's formula yields
\begin{align*}
\mathbb{E}\left(\abs{\Lambda_t}^{2p}\right)&=\mathbb{E}\left(\abs{D^{W^{1}}_{r_1}
                                          Y_{r_2}^{\eta}}^{2p}\right)
                                          \mathds{1}_{\left\{ r_2 \geq
                                          r_1\right\}}\\
&\quad +\frac{2p}{\eta}\int_{r_1\vee r_2}^{t}\mathbb{E}\left(\left[2\partial_2 f\left(X_s^{\varepsilon},Y_s^{\eta}\right)+(4p-2)\partial_2\tau\left(X_s^{\varepsilon},Y_s^{\eta}\right)^{2}\right]\abs{\Lambda_{s}}^{2p}\right)ds\\
&\quad +\frac{2p}{\eta}\int_{r_1\vee
                                                                                                    r_2}^{t}\mathbb{E}\Big([\partial_1
                                                                                                                                                                                       f\left(X_s^{\varepsilon},Y_s^{\eta}\right)\\
  &\qquad\qquad\qquad\qquad +(4p-2)\partial_1\tau\left(X_s^{\varepsilon},Y_s^{\eta}\right) \partial_2\tau\left(X_s^{\varepsilon},Y_s^{\eta}\right)]D^{W^{1}}_{r_1} X_{s}^{\varepsilon}Z_{r_2,2}(s)\Lambda_{s}^{2p-1}\Big)ds\\
&\quad +\frac{2p}{\eta}\int_{r_1\vee r_2}^{t}\mathbb{E}\left(\frac{2p-1}{2}\partial_1\tau\left(X_s^{\varepsilon},Y_s^{\eta}\right)^{2} \abs{D^{W^{1}}_{r_1} X_{s}^{\varepsilon}Z_{r_2,2}(s)}^{2}\Lambda_s^{2p-2}\right)ds.
\end{align*}
We can now differentiate this equality with respect to $t$ and apply
Young's inequality for products in order to get, for some constants
$D_1,D_2>0$ to be chosen later,
\begin{align*}
\frac{d}{dt}\mathbb{E}\left( \abs{\Lambda_t}^{2p} \right)
&\leq\frac{2p}{\eta}\mathbb{E}\left(\left[2\partial_2 f\left(X_t^{\varepsilon},Y_t^{\eta}\right)+(4p-2)\partial_2\tau\left(X_t^{\varepsilon},Y_t^{\eta}\right)^{2}\right]\abs{\Lambda_{t}}^{2p}\right)\nonumber\\
&\quad +\frac{D_{1}^{2p}}{\eta}\mathbb{E}\left( \abs{\left(\partial_1
f\left(X_t^{\varepsilon},Y_t^{\eta}\right)+(4p-2)\partial_1\tau\left(X_t^{\varepsilon},Y_t^{\eta}\right) \partial_2\tau\left(X_t^{\varepsilon},Y_t^{\eta}\right)\right)D^{W^{1}}_{r_1} X_{t}^{\varepsilon}Z_{r_2,2}(t)}^{2p}\right)\\
  &\quad +
    \frac{2p-1}{\eta}D_{1}^{-\frac{2p}{2p-1}}\mathbb{E}\left(\abs{\Lambda_{t}}^{2p}\right)
    +
    \frac{2p-1}{\eta}D_{2}^{p}\mathbb{E}\left( \abs{\partial_1\tau\left(X_t^{\varepsilon},Y_t^{\eta}\right)
    D^{W^{1}}_{r_1} X_{t}^{\varepsilon}Z_{r_2,2}(t)}^{2p}\right)\\
  &\quad +\frac{2p-2}{\eta}D_{2}^{-\frac{p}{p-1}}\mathbb{E}\left( \abs{\Lambda_t}^{2p} \right).
\end{align*}
Using Assumption \ref{A:Assumptiongeneral}, we get
\begin{align*}
\frac{d}{dt}\mathbb{E}\left( \abs{\Lambda_t}^{2p} \right)
  &\leq\frac{2p}{\eta}\left(-K+\frac{2p-1}{2p}D_{1}^{-\frac{2p}{2p-1}}+\frac{p-1}{p}D_{2}^{-\frac{p}{p-1}}\right)\mathbb{E}\left(\abs{\Lambda_t}^{2p}\right)\\
  &\quad +\frac{C}{\eta}\left( D_1^{2p} + (2p-1)D_2^p \right)\mathbb{E}\left( \abs{D^{W^{1}}_{r_1} X_{t}^{\varepsilon}Z_{r_2,2}(t)}^{2p}\right).
\end{align*}
We can now pick the constants $D_1,D_2$ to be sufficiently large so
that
\begin{equation*}
2p\left(-K+\frac{2p-1}{2p}D_{1}^{-\frac{2p}{2p-1}}+\frac{p-1}{p}D_{2}^{-\frac{p}{p-1}}\right)\leq
-K_0<0
\end{equation*}
for some constant $K_0 >0$. We hence get that for some constant $C>0$,
\begin{align*}
\frac{d}{dt}\mathbb{E}\left( \abs{\Lambda_t}^{2p} \right)
&\leq-\frac{ K_0}{\eta}\mathbb{E}\left( \abs{\Lambda_t}^{2p} \right)
+\frac{C}{\eta}\mathbb{E}\left( \abs{D^{W^{1}}_{r_1} X_{t}^{\varepsilon}Z_{r_2,2}(t)}^{2p}\right).
\end{align*}
Solving this differential inequality and using Propositions \ref{boundonEDW1Xs2p} and
\ref{boundonEDW1Yt2pb} yields
\begin{align}
  \label{solvingdiffineqforLambdat}
\mathbb{E}\left( \abs{\Lambda_t}^{2p} \right)&\leq e^{-\frac{K_0}{\eta}(t-r_1\vee
                                               r_2)} \mathbb{E}\left(
                                               \abs{D^{W^{1}}_{r_1}
                                               Y^{\eta}_{r_2}}^{2p}\right) \mathds{1}_{\left\{ r_2
                                               \geq r_1\right\}}\nonumber\\
  &\quad +\frac{C}{\eta}e^{-\frac{K_0}{\eta}t}\int_{r_1\vee r_2}^{t}e^{\frac{K_0}{\eta}s}\mathbb{E}\left( \abs{D^{W^{1}}_{r_1} X_{s}^{\varepsilon}Z_{r_2,2}(s)}^{2p}\right)ds\nonumber\\
&\leq  C \varepsilon^p e^{-\frac{K_0}{\eta}(t-r_1\vee r_2)}
                                                                                                                                      \mathds{1}_{\left\{ r_2   \geq  r_1\right\}}\nonumber\\
  &\quad +\varepsilon^{p}\frac{C}{\eta}e^{-\frac{K_0}{\eta}t}\int_{r_1\vee
                                                                                                                                                                                        r_2}^{t}e^{\frac{K_0}{\eta}s}\mathbb{E}\left(
    \abs{Z_{r_2,2}(s)}^{2p q_1}\right)^{\frac{1}{q_1}}ds.
\end{align}
Lemma \ref{lemmaonExpofZr2tothe2pq} allows us to continue by writing
\begin{align*}
\mathbb{E}\left( \abs{\Lambda_t}^{2p} \right)&\leq C\varepsilon^{p}  \left[e^{-\frac{K_0}{\eta}(t-r_1\vee r_2)}  \mathds{1}_{\left\{ r_2\geq r_1\right\}}+\frac{1}{\eta}e^{-\frac{K_0}{\eta}t}\int_{r_1\vee r_2}^{t}e^{\frac{K_0}{\eta}s}e^{-\frac{K}{q_1\eta}(s-r_2)}ds\right]\nonumber\\
&=  C\varepsilon^{p} \left[e^{-\frac{K_0}{\eta}(t-r_1\vee r_2)}  \mathds{1}_{\left\{ r_2\geq r_1\right\}}+\frac{1}{\eta}e^{-\frac{K_0}{\eta}t}e^{\frac{K}{q_1\eta}r_2}\int_{r_1\vee r_2}^{t}e^{\frac{K_0-\frac{K}{q_1}}{\eta}s}ds\right]\nonumber\\
&\leq  C\varepsilon^{p} \left[e^{-\frac{K_0}{\eta}(t-r_1\vee r_2)}  \mathds{1}_{\left\{ r_2\geq r_1\right\}}+e^{-\frac{K_0}{\eta}t}e^{\frac{K}{q_1\eta}r_2}\left(e^{\frac{K_0-\frac{K}{q_1}}{\eta}t}-e^{\frac{K_0-\frac{K}{q_1}}{\eta}r_1\vee r_2}\right)\right]\nonumber\\
&\leq  C\varepsilon^{p} \left[e^{-\frac{K_0}{\eta}(t-r_1\vee r_2)}  \mathds{1}_{\left\{ r_2\geq r_1\right\}}+e^{-\frac{K}{q_1\eta}(t-r_{2})}\right]\nonumber\\
&\leq C\varepsilon^{p} e^{-\frac{K_0\wedge \frac{K}{q_1}}{\eta}(t-r_1\vee r_2)},
\end{align*}
which concludes the proof.
\end{proof}

\begin{lemma}
  \label{lemmaonEVr1veer2sZr22p}
For any $r_1 \vee r_2 \leq t \leq 1$, define
\begin{equation*}
V_{r_1 \vee r_2}(t) = \int_{r_1 \vee r_2}^t Z_{r_1 \vee r_2,2}(s)D_{r_1}^{W^1}Y_s^{\eta}Q_{r_2,1}^{\eta}(s)ds,
\end{equation*}
where $Q_{r_2,1}^{\eta}(s)$ is given by \eqref{defQ1}. Let $p\geq 1$ be a  natural number. Then, for any $r_1 \vee r_2 \leq t
\leq 1$, there exists a positive constants $C,D >0$ such that
\begin{equation*}
\mathbb{E}\left(  \abs{\frac{1}{\eta}Z_{r_2,2}(t)V_{r_1 \vee r_2}(t)}^{2p}\right)
\leq C \left( \frac{\varepsilon}{\eta} \right)^p e^{-\frac{D}{\eta}(t
  - r_1 \vee r_2)}.
\end{equation*}
\end{lemma}
\begin{proof}
Defining $\Lambda_{r,2}(t) = \frac{1}{\eta}Z_{r,2}(t)V_r(t)$ and
following the same line of arguments as in the proof of Lemma
\ref{lemmaonEDW1YZr22p} (with $\Lambda_{r,2}(t)$ instead of
$\Lambda_t$) yields, for some constant $K_0 >0$,
\begin{align*}
\frac{d}{dt}\mathbb{E}\left( \abs{\Lambda_{r,2}(t)}^{2p} \right)
&\leq-\frac{ K_0}{\eta}\mathbb{E}\left( \abs{\Lambda_{r,2}(t)}^{2p} \right)
+\frac{C}{\eta}\mathbb{E}\left( \abs{D_{r_1}^{W^1}Y_t^{\eta}Q_{r_2,1}^{\eta}(t)}^{2p}\right).
\end{align*}
Solving this differential inequality in the same manner as was done in
\eqref{solvingdiffineqforLambdat} in the proof of Lemma
\ref{lemmaonEDW1YZr22p} get us
\begin{align*}
\mathbb{E}\left(\abs{\Lambda_{r_1\vee r_2,2}(t)}^{2p}\right)&\leq \frac{C}{\eta}e^{-\frac{K_0}{\eta}t}\int_{r_1\vee r_2}^{t}e^{\frac{K_0}{\eta}s}\mathbb{E}\left( \abs{D_{r_1}^{W^{1}}Y_{s}^{\eta}Q^{\eta}_{r_2,1}(s)}^{2p}\right)ds\nonumber\\
&\leq  C\left( \frac{\varepsilon}{\eta} \right)^p \frac{1}{\eta}e^{-\frac{K_0}{\eta}t}\int_{r_1\vee r_2}^{t}e^{\frac{K_0}{\eta}s}\mathbb{E}\left(\abs{Z_{r_2,2}(s)}^{2p q_1}\right)^{\frac{1}{q_1}}ds,
\end{align*}
where we have used H\"older's inequality, the fact that $Q^{\eta}_{r_2,1}(s)=\frac{1}{\sqrt{\eta}} Z_{r_2,2}(s)
                   \tau\left(X_{r_2}^{\varepsilon},Y_{r_2}^{\eta}
                   \right)$, and Proposition \ref{boundonEDW1Yt2p}. Finally, using Lemma \ref{lemmaonExpofZr2tothe2pq} and integrating concludes the proof.
                 \end{proof}

                 \begin{prop}
\label{expintegratedFuGudu}
Let $r_1,r_2 \leq t \leq 1$, $p \geq 1$ be an integer, $q \in
\left\{ 1,2 \right\}$, and $\left\{ Z_{r,2}(t)\colon r \leq t \leq 1
\right\}$ be the stochastic process defined by
\eqref{expressionofZr2t}. For any
stochastic processes $\left\{ F_t\colon r_1 \wedge r_2 \leq t \leq 1
\right\}$, $\left\{ G_t\colon r_1 \wedge r_2 \leq t \leq 1
\right\}$ in $L^{4p} \left( \Omega \right)$, one has
\begin{align*}
&\mathbb{E}\left(\left(\int_{r_1 \vee r_2}^t
  \abs{\frac{1}{\eta}Z_{r_1
  \vee
  r_2,2}(s)\int_{r_1 \vee r_2}^s Z_{r_1
  \vee
  r_2,2}^{-1}(u)F_uG_udu}^q ds \right)^{\frac{2p}{q}}
                 \right)\\
  & \qquad\qquad\qquad\qquad\qquad\qquad\qquad\qquad\qquad \leq  C \sup_{r_1\vee r_2 \leq s \leq t}\mathbb{E}\left(
  \abs{F_s}^{4p}
    \right)^{1/2}\mathbb{E}\left(
  \abs{G_s}^{4p}
    \right)^{1/2}.
\end{align*}
\end{prop}
\begin{proof}
We have
\begin{align*}
&\mathbb{E}\left(\left(\int_{r_1 \vee r_2}^t
  \abs{\frac{1}{\eta}Z_{r_1
  \vee
  r_2,2}(s)\int_{r_1 \vee r_2}^s Z_{r_1
  \vee
  r_2,2}^{-1}(u)F_uG_udu}^q ds \right)^{\frac{2p}{q}}
                 \right) \\
  &\qquad\qquad\qquad\qquad \leq C\mathbb{E}\left(\int_{r_1 \vee r_2}^t
  \abs{\frac{1}{\sqrt{\eta}}Z_{r_1
  \vee
  r_2,2}(s)\int_{r_1 \vee r_2}^s Z_{r_1
  \vee
  r_2,2}^{-1}(u)F_uG_udu}^{2p} ds \right).
\end{align*}
Using Proposition \ref{proponthedsintegrals} yields
\begin{align*}
&\mathbb{E}\left(\int_{r_1 \vee r_2}^t
  \abs{\frac{1}{\eta}Z_{r_1
  \vee
  r_2,2}(s)\int_{r_1 \vee r_2}^s Z_{r_1
  \vee
  r_2,2}^{-1}(u)F_uG_udu}^{2p} ds
                 \right) \\
  &\qquad\qquad\qquad\qquad = \int_{r_1 \vee r_2}^t
  \mathbb{E}\left(\abs{\frac{1}{\eta}Z_{r_1
  \vee
  r_2,2}(s)\int_{r_1 \vee r_2}^s Z_{r_1
  \vee
  r_2,2}^{-1}(u)F_uG_udu}^{2p}
                 \right)ds \\
  &\qquad\qquad\qquad\qquad \leq C \sup_{r_1\vee r_2 \leq s \leq t}\mathbb{E}\left(
  \abs{F_sG_s}^{2p}
    \right)\\
  &\qquad\qquad\qquad\qquad \leq C \sup_{r_1\vee r_2 \leq s \leq t}\mathbb{E}\left(
  \abs{F_s}^{4p}
    \right)^{1/2}\mathbb{E}\left(
  \abs{G_s}^{4p}
    \right)^{1/2}.
\end{align*}
\end{proof}

\begin{prop}
\label{expintegratedFuGudWu}
Let $r_1,r_2 \leq t \leq 1$, $p \geq 1$ be an integer, $q \in
\left\{ 1,2 \right\}$, and $\left\{ Z_{r,2}(t)\colon r \leq t \leq 1
\right\}$ be the stochastic process defined by
\eqref{expressionofZr2t}. For any
stochastic processes $\left\{ F_t\colon r_1 \wedge r_2 \leq t \leq 1
\right\}$, $\left\{ G_t\colon r_1 \wedge r_2 \leq t \leq 1
\right\}$ in $L^{4p} \left( \Omega \right)$, one has
\begin{align*}
&\mathbb{E}\left(\left(\int_{r_1 \vee r_2}^t
  \abs{\frac{1}{\sqrt{\eta}}Z_{r_1
  \vee
  r_2,2}(s)\int_{r_1 \vee r_2}^s Z_{r_1
  \vee
  r_2,2}^{-1}(u)F_uG_udW_u^2}^q ds \right)^{\frac{2p}{q}}
                 \right)\\
  & \qquad\qquad\qquad\qquad\qquad\qquad\qquad\qquad\qquad \leq  C \sup_{r_1\vee r_2 \leq s \leq t}\mathbb{E}\left(
  \abs{F_s}^{4p}
    \right)^{1/2}\mathbb{E}\left(
  \abs{G_s}^{4p}
    \right)^{1/2}.
\end{align*}
\end{prop}
\begin{proof}
We have
\begin{align*}
&\mathbb{E}\left(\left(\int_{r_1 \vee r_2}^t
  \abs{\frac{1}{\sqrt{\eta}}Z_{r_1
  \vee
  r_2,2}(s)\int_{r_1 \vee r_2}^s Z_{r_1
  \vee
  r_2,2}^{-1}(u)F_uG_udW_u^2}^q ds \right)^{\frac{2p}{q}}
                 \right) \\
  &\qquad\qquad\qquad\qquad \leq C\mathbb{E}\left(\int_{r_1 \vee r_2}^t
  \abs{\frac{1}{\sqrt{\eta}}Z_{r_1
  \vee
  r_2,2}(s)\int_{r_1 \vee r_2}^s Z_{r_1
  \vee
  r_2,2}^{-1}(u)F_uG_udW_u^2}^{2p} ds \right).
\end{align*}
Using Proposition \ref{proponthedWintegrals} yields
\begin{align*}
&\mathbb{E}\left(\int_{r_1 \vee r_2}^t
  \abs{\frac{1}{\sqrt{\eta}}Z_{r_1
  \vee
  r_2,2}(s)\int_{r_1 \vee r_2}^s Z_{r_1
  \vee
  r_2,2}^{-1}(u)F_uG_udW_u^2}^{2p} ds
                 \right) \\
  &\qquad\qquad\qquad\qquad = \int_{r_1 \vee r_2}^t
  \mathbb{E}\left(\abs{\frac{1}{\sqrt{\eta}}Z_{r_1
  \vee
  r_2,2}(s)\int_{r_1 \vee r_2}^s Z_{r_1
  \vee
  r_2,2}^{-1}(u)F_uG_udW_u^2}^{2p}
                 \right)ds \\
  &\qquad\qquad\qquad\qquad \leq C \sup_{r_1\vee r_2 \leq s \leq t}\mathbb{E}\left(
  \abs{F_sG_s}^{2p}
    \right)\\
  &\qquad\qquad\qquad\qquad \leq C \sup_{r_1\vee r_2 \leq s \leq t}\mathbb{E}\left(
  \abs{F_s}^{4p}
    \right)^{1/2}\mathbb{E}\left(
  \abs{G_s}^{4p}
    \right)^{1/2}.
\end{align*}
\end{proof}

\begin{prop}
\label{proponthedWintegrals}
Let $r \leq t \leq 1$, $p \geq 1$ be an integer and $\left\{ Z_{r,2}(t)\colon r \leq t \leq 1
\right\}$ be the stochastic process defined by \eqref{expressionofZr2t}. For any
stochastic process $\left\{ A_t\colon r \leq t \leq 1
\right\}$ in $L^{2p} \left( \Omega \right)$, one has
\begin{align*}
\mathbb{E}\left(
  \abs{\frac{1}{\sqrt{\eta}}Z_{r,2}(t)\int_r^tZ_{r,2}^{-1}(s)A_sdW_s^2}^{2p}
  \right)\leq C \sup_{r\leq s \leq t}\mathbb{E}\left(
  \abs{A_s}^{2p}
  \right).
\end{align*}
\end{prop}
\begin{proof}
Recall that $Z_{r,2}(t)$ satisfies the stochastic differential equation
\eqref{sdeforZr2}. For any $r \leq t \leq 1$, let us write
\begin{equation*}
U_r(t) = \int_r^tZ_{r,2}^{-1}(s)A_sdW_s^2.
\end{equation*}
Then, the It\^o product formula yields
\begin{align*}
\frac{1}{\sqrt{\eta}}Z_{r,2}(t)U_r(t) &= \frac{1}{\sqrt{\eta}}\int_r^tZ_{r,2}(s)dU_r(s) +
  \frac{1}{\sqrt{\eta}}\int_r^tU_r(s)dZ_{r,2}(s) + \frac{1}{\sqrt{\eta}}\int_r^t d \left\langle Z_{r,2},U_r
  \right\rangle_s\\
  &= \frac{1}{\sqrt{\eta}}\int_r^t \left( A_s + \frac{1}{\sqrt{\eta}}Z_{r,2}(s)U_r(s)\partial_2\tau\left(X_s^{\varepsilon},Y_s^{\eta}
    \right) \right)dW_s^2\\
  &\quad + \frac{1}{\eta}\int_r^t \left( A_s\partial_2\tau\left(X_s^{\varepsilon},Y_s^{\eta}
    \right) + \frac{1}{\sqrt{\eta}}Z_{r,2}(s)U_r(s)\partial_2f\left(X_s^{\varepsilon},Y_s^{\eta}
    \right) \right)ds
\end{align*}
Hence, applying the It\^o formula with the function $x \mapsto x^{2p}$
and taking expectation yields
\begin{align*}
\mathbb{E}\left(
  \abs{\frac{1}{\sqrt{\eta}}Z_{r,2}(t)U_r(t)}^{2p}
  \right) &= \frac{1}{\eta}\mathbb{E}\left( \int_r^t 2p \left(\frac{1}{\sqrt{\eta}}Z_{r,2}(s)U_r(s)\right)^{2p}\partial_2f\left(X_s^{\varepsilon},Y_s^{\eta}
            \right)ds \right)\\
  &\quad + \frac{1}{\eta}\mathbb{E}\left( \int_r^t 2p \left(\frac{1}{\sqrt{\eta}}Z_{r,2}(s)U_r(s)\right)^{2p-1}\partial_2\tau\left(X_s^{\varepsilon},Y_s^{\eta}
            \right)A_sds \right)\\
  &\quad + \frac{1}{\eta}\mathbb{E}\Bigg( \int_r^t p(2p-1)
    \left(\frac{1}{\sqrt{\eta}}Z_{r,2}(s)U_r(s)\right)^{2p-2} \\
  &\qquad\qquad\qquad\qquad\qquad \left[
    A_s + \frac{1}{\sqrt{\eta}}Z_{r,2}(s)U_r(s)\partial_2\tau\left(X_s^{\varepsilon},Y_s^{\eta}
  \right) \right]^2ds \Bigg).
\end{align*}
By differentiating this equality with respect to $t$, one can write
\begin{align*}
\frac{d}{dt}\mathbb{E}\left(
  \abs{\frac{1}{\sqrt{\eta}}Z_{r,2}(t)U_r(t)}^{2p}
  \right) &= \frac{2p}{\eta}\mathbb{E}\left( \left(\frac{1}{\sqrt{\eta}}Z_{r,2}(t)U_r(t)\right)^{2p}\partial_2f\left(X_t^{\varepsilon},Y_t^{\eta}
            \right) \right)\\
  &\quad + \frac{2p}{\eta}\mathbb{E}\left(  \left(\frac{1}{\sqrt{\eta}}Z_{r,2}(t)U_r(t)\right)^{2p-1}\partial_2\tau\left(X_t^{\varepsilon},Y_t^{\eta}
            \right)A_t \right)\\
  &\quad + \frac{p(2p-1)}{\eta}\mathbb{E}\Bigg(
    \left(\frac{1}{\sqrt{\eta}}Z_{r,2}(t)U_r(t)\right)^{2p-2} \\
  &\qquad\qquad\qquad\qquad\qquad \left[
    A_t + \frac{1}{\sqrt{\eta}}Z_{r,2}(t)U_r(t)\partial_2\tau\left(X_t^{\varepsilon},Y_t^{\eta}
  \right) \right]^2 \Bigg).
\end{align*}
Using Young's inequality for products yields
\begin{align*}
\frac{d}{dt}\mathbb{E}\left(
  \abs{\frac{1}{\sqrt{\eta}}Z_{r,2}(t)U_r(t)}^{2p}
                 \right) &\leq \frac{1}{\eta}\mathbb{E}\Bigg(
            \abs{\frac{1}{\sqrt{\eta}}Z_{r,2}(t)U_r(t)}^{2p}\Big[  2p\partial_2f\left(X_t^{\varepsilon},Y_t^{\eta}
            \right) + (2p-1)\abs{\partial_2\tau\left(X_t^{\varepsilon},Y_t^{\eta}
                           \right)}\\
  &\qquad\qquad\qquad\qquad + 2p(2p-1)\abs{\partial_2\tau\left(X_t^{\varepsilon},Y_t^{\eta}
            \right)}^2 + (2p-1)(2p-2) \Big] \Bigg)\\
  &\quad +  \frac{1}{\eta}\mathbb{E}\Bigg(
    \abs{A_t}^{2p} \left[ \abs{\partial_2\tau\left(X_t^{\varepsilon},Y_t^{\eta}
            \right)} + 2(2p-1) \right] \Bigg).
\end{align*}
Using Assumption
\ref{A:Assumptiongeneral}, we finally get
\begin{align*}
\frac{d}{dt}\mathbb{E}\left(
  \abs{\frac{1}{\sqrt{\eta}}Z_{r,2}(t)U_r(t)}^{2p}
  \right)  & \leq
    -\frac{K}{\eta}\mathbb{E}\left(\abs{\frac{1}{\sqrt{\eta}}Z_{r,2}(t)U_r(t)}^{2p}
    \right)  +  \frac{M}{\eta} \mathbb{E}\left(\abs{A_t}^{2p}\right).
\end{align*}
Solving this differential inequality yields
\begin{align*}
\mathbb{E}\left(
  \abs{\frac{1}{\sqrt{\eta}}Z_{r,2}(t)U_r(t)}^{2p}
  \right) \leq \frac{M}{K}\sup_{r \leq s \leq t}\mathbb{E}\left(\abs{A_s}^{2p}\right).
\end{align*}
\end{proof}

\begin{prop}
  \label{proponthedsintegrals}
Let $r \leq t \leq 1$, $p \geq 1$ be an integer and $\left\{ Z_{r,2}(t)\colon r \leq t \leq 1
\right\}$ be the stochastic process defined by \eqref{expressionofZr2t}. For any
stochastic process $\left\{ B_t\colon r \leq t \leq 1
\right\}$ in $L^{2p} \left( \Omega \right)$, one has
\begin{align*}
\mathbb{E}\left(
  \abs{\frac{1}{\eta}Z_{r,2}(t)\int_r^tZ_{r,2}^{-1}(s)B_sds}^{2p}
  \right)\leq C \sup_{r\leq s \leq t}\mathbb{E}\left(
  \abs{B_s}^{2p}
  \right).
\end{align*}
\end{prop}
\begin{proof}
Recall that $Z_{r,2}(t)$ satisfies the stochastic differential equation
\eqref{sdeforZr2}. For any $r \leq t \leq 1$, let us write
\begin{equation*}
V_r(t) = \int_r^tZ_{r,2}^{-1}(s)B_sds.
\end{equation*}
Then, the It\^o product formula yields
\begin{align*}
\frac{1}{\eta}Z_{r,2}(t)V_r(t) &= \frac{1}{\eta}\int_r^tZ_{r,2}(s)dV_r(s) +
  \frac{1}{\eta}\int_r^tV_r(s)dZ_{r,2}(s) \\
  &= \frac{1}{\eta}\int_r^t \left(  B_sds +
   \frac{1}{\eta}Z_{r,2}(s)V_r(s) \partial_2f\left(X_s^{\varepsilon},Y_s^{\eta}
    \right)\right)ds\\
  &\quad +\frac{1}{\sqrt{\eta}}\int_r^t \frac{1}{\eta}Z_{r,2}(s)V_r(s) \partial_2\tau\left(X_s^{\varepsilon},Y_s^{\eta}
    \right)dW_s^2.
\end{align*}
Hence, applying the It\^o formula with the function $x \mapsto x^{2p}$
and taking expectation yields
\begin{align*}
\mathbb{E}\left(
  \abs{\frac{1}{\eta}Z_{r,2}(t)V_r(t)}^{2p}
  \right) &= \frac{1}{\eta}\mathbb{E}\left(\int_r^t 2p \left(
            \frac{1}{\eta}Z_{r,2}(s)V_r(s) \right)^{2p-1}B_sds\right)\\
  &\quad + \frac{1}{\eta}\mathbb{E}\left(\int_r^t 2p \left(
            \frac{1}{\eta}Z_{r,2}(s)V_r(s) \right)^{2p}\partial_2f\left(X_s^{\varepsilon},Y_s^{\eta}
    \right)ds\right) \\
  &\quad + \frac{1}{\eta}\mathbb{E}\left(\int_r^t p(2p-1) \left(
            \frac{1}{\eta}Z_{r,2}(s)V_r(s) \right)^{2p}\abs{\partial_2\tau\left(X_s^{\varepsilon},Y_s^{\eta}
    \right)}^2ds\right).
\end{align*}
By differentiating this equality with respect to $t$, one can write
\begin{align*}
\frac{d}{dt}\mathbb{E}\left(
  \abs{\frac{1}{\eta}Z_{r,2}(t)V_r(t)}^{2p}
  \right) &= \frac{2p}{\eta} \mathbb{E}\left( \left(
            \frac{1}{\eta}Z_{r,2}(t)V_r(t) \right)^{2p-1}B_t\right)\\
  &\quad + \frac{2p}{\eta} \mathbb{E}\left( \left(
            \frac{1}{\eta}Z_{r,2}(t)V_r(t) \right)^{2p}\partial_2f\left(X_t^{\varepsilon},Y_t^{\eta}
    \right)\right) \\
  &\quad + \frac{p(2p-1)}{\eta} \mathbb{E}\left( \left(
            \frac{1}{\eta}Z_{r,2}(t)V_r(t) \right)^{2p}\abs{\partial_2\tau\left(X_t^{\varepsilon},Y_t^{\eta}
    \right)}^2\right).
\end{align*}
Using Young's inequality for products yields
\begin{align*}
\frac{d}{dt}\mathbb{E}\left(
  \abs{\frac{1}{\eta}Z_{r,2}(t)V_r(t)}^{2p}
                 \right) &\leq \frac{1}{\eta}\mathbb{E}\Bigg(
            \abs{\frac{1}{\eta}Z_{r,2}(t)V_r(t)}^{2p}\Big[  2p\partial_2f\left(X_t^{\varepsilon},Y_t^{\eta}
            \right) + p(2p-1)\abs{\partial_2\tau\left(X_t^{\varepsilon},Y_t^{\eta}
                           \right)}^2\\
  &\qquad\qquad\qquad\qquad\qquad\qquad\qquad\qquad + (2p-1) \Big] \Bigg) +  \frac{1}{\eta}\mathbb{E}\Bigg(
    \abs{B_t}^{2p} \Bigg).
\end{align*}
Using Assumption
\ref{A:Assumptiongeneral}, we finally get
\begin{align*}
\frac{d}{dt}\mathbb{E}\left(
  \abs{\frac{1}{\eta}Z_{r,2}(t)V_r(t)}^{2p}
  \right)  & \leq
    -\frac{K}{\eta}\mathbb{E}\left(\abs{\frac{1}{\eta}Z_{r,2}(t)V_r(t)}^{2p}
    \right)  +  \frac{1}{\eta} \mathbb{E}\left(\abs{B_t}^{2p}\right).
\end{align*}
Solving this differential inequality yields
\begin{align*}
\mathbb{E}\left(
  \abs{\frac{1}{\eta}Z_{r,2}(t)V_r(t)}^{2p}
  \right) \leq \frac{1}{K}\sup_{r \leq s \leq t}\mathbb{E}\left(\abs{B_s}^{2p}\right).
\end{align*}
\end{proof}

\begin{lemma}\label{L:BoundIntergalTermFinalSecondMalliavinTerm}
 Let $1\ll k<\infty$ and $0<T<\infty$ be given. Then, there exists a constant $C<\infty$ that may depend on $T$, but is independent of $k$, such that for $k$ large enough the bound holds
\begin{align}
\int_{[0,T]^4} e^{-k(u\vee v - u \wedge v)} e^{-k(u\vee w - u \wedge w)}e^{-k(s\vee v - s \wedge v)}e^{-k(s\vee w - s \wedge w)}dw ds dv du&\leq C\left(k^{-3}+k^{-2}e^{-2k}\right)\nonumber
\end{align}
\end{lemma}
\begin{proof}

There are $4!=24$ orderings of the elements in $\{u,v,s,w\}$. We will examine a representative one and we note that all others yield the same bound by symmetry. Let us set $A=\{(u,v,s,w)\in[0,T]^{4}: 0\leq u<v<s<w\leq T\}$. We have
\begin{align}
&\int_{A} e^{-k(u\vee v - u \wedge v)} e^{-k(u\vee w - u \wedge w)}e^{-k(s\vee v - s \wedge v)}e^{-k(s\vee w - s \wedge w)}dw ds dv du=\int_{A}e^{-2k(w-u)}dw ds dv du\nonumber\\
&\qquad=\int_{0}^{T}\int_{u}^{T}\int_{v}^{T}\int_{s}^{T}e^{-2k(w-u)} dw ds dv du\nonumber\\
&\qquad=\int_{0}^{T}\int_{u}^{T}\int_{v}^{T}\frac{1}{2k}\left(e^{-2k(s-u)}- e^{-2k(T-u)}\right) ds dv du\nonumber\\
&\qquad=\frac{1}{2k}\int_{0}^{T}\int_{u}^{T}\left(\frac{1}{2k}\left(e^{-2k(v-u)}-e^{-2k(T-u)}\right)-(T-v) e^{-2k(T-u)}\right) dv du\nonumber\\
&\qquad=\frac{1}{2k}\int_{0}^{T}\left(\frac{1}{4k^{2}}\left(1-e^{-2k(T-u)}\right)-\frac{1}{2k}(T-u)e^{-2k(T-u)}-\left(\frac{T^{2}}{2}-Tu+\frac{u^{2}}{2}\right)e^{-2k(T-u)}\right) du\nonumber\\
&\qquad\leq C\left(k^{-3}+k^{-2}e^{-2kT}\right).\nonumber
\end{align}
\end{proof}

\bibliographystyle{plain}

\begin{thebibliography}{NPR09}

\bibitem[BLP78]{BLP}A. Bensoussan, J.L. Lions, G. Papanicolaou, \textit{Asymptotic
Analysis for Periodic Structures}, Vol 5, Studies in Mathematics and its
Applications, North-Holland Publishing Co., Amsterdam, 1978.


\bibitem[BGS21]{BGS21}
Solesne Bourguin, Siragan Gailus and Konstantinos Spiliopoulos, \emph{Discrete-time inference for slow-fast systems driven by fractional Brownian motion},  SIAM Journal on Multiscale Modeling and Simulation, Vol. 19, No. 3, (2021), pp. 1333--1366.

\bibitem[CDGOS22]{OttobreUiT2022}
D. Crisan, P.Dobson, B. Goddard, M. Ottobre and I. Souttar, \textit{Poisson Equations with locally-Lipschitz coefficients and Uniform in Time Averaging for Stochastic Differential Equations via Strong Exponential Stability}, arXiv: 2206.06776, (2022).

\bibitem[DS12]{DupuisSpiliopoulos}P. Dupuis and K. Spiliopoulos, \emph{Large deviations
for multiscale problems via weak convergence methods}, Stochastic
Processes and their Applications, Vol. 122, (2012), pp. 1947-1987.



\bibitem[F78]{Freidlin1978}
M.I. Freidlin, \emph{The Averaging Principle and Theorems on Large Deviations}, Russian Mathematical Surveys, Vol. 33, No. 5, (1978), pp. 117-176.

\bibitem[FS99] {FS}M.I. Freidlin, R. Sowers, \emph{A comparison of homogenization and large
deviations, with applications to wavefront propagation}, Stochastic
Process and Their Applications, Vol. 82, Issue 1, (1999), pp. 23--52.

\bibitem[GS17]{GS17}
Siragan Gailus and Konstantinos Spiliopoulos, \emph{Statistical Inference for Perturbed Multiscale Dynamical Systems}, Stochastic Processes and their Applications , Volume 127, Issue 2, (2017), pp. 419-448.

\bibitem[G03]{Guillin}
A. Guillin, \emph{Averaging principle of SDE with small diffusion: moderate deviations},  Annals of Probability, Vol. 31, No. 1, (2003), pp. 413-443.


\bibitem[IRS19] {IRS19}
P. Imkeller and G. dos Reis and W. Salkeld, \emph{Differentiability of
{SDE}s with drifts of super-linear growth},
Electron. J. Probab., Vol. 24, (2019), paper No. 3.

\bibitem[LS90]{LiptserPaper}
Robert Liptser and Jordan Stoyanov, \emph{Stochastic version of the averaging principle for diffusion type processes}, Stochastics Stochastics Rep., Stochastics and Stochastics Reports, Vol. 32, No. 3-4, (1990), pp. 145--163.


\bibitem[NPR09]{nourdin_second_2009}
Ivan Nourdin, Giovanni Peccati, and Gesine Reinert, \emph{Second order
  {Poincaré} inequalities and {CLTs} on {Wiener} space}, J. Funct. Anal.
  \textbf{257} (2009), no.~2, 593--609. \MR{2527030}

\bibitem[Nua06]{nualart_malliavin_2006}
David Nualart, \emph{The {Malliavin} calculus and related topics}, second ed.,
  Probability and its {Applications} ({New} {York}), Springer-Verlag, Berlin,
  2006.

\bibitem[PV01]{PardouxVeretennikov1}E. Pardoux, A.Yu. Veretennikov, \emph{On Poisson
equation and diffusion approximation I}, Annals of Probability, Vol.
29, No. 3, (2001), pp. 1061-1085.


\bibitem[PV03]{PardouxVeretennikov2}E. Pardoux, A.Yu. Veretennikov, \emph{On Poisson
equation and diffusion approximation 2}, Annals of Probability, Vol.
31, No. 3, (2003), pp. 1166-1192.

\bibitem[RX21a]{RocknerAveraging2021a} M. R\"{o}ckner, L. Xie, \emph{Averaging principle and normal deviations for multiscale stochastic systems}, Communications in Mathematical Physics, 383, (2021), pp. 1889-1937.

\bibitem[RX21b]{RocknerAveraging2021b} M. R\"{o}ckner, L. Xie, \emph{Diffusion approximation for fully coupled stochastic differential equations}, The Annals of Probability, Vol. 49, No. 3, (2021), pp. 1205-1236.

\bibitem[S14]{Spiliopoulos_CLT_Multiscale_2014}
Konstantinos Spiliopoulos, \emph{Fluctuation analysis and short time asymptotics for multiple scales diffusion processes},  Stochastics and Dynamics, Vol. 14, No.3, (2014), pp. 1350026
\end{thebibliography}

\end{document}